\documentclass[a4paper,reqno]{amsart}
 
\usepackage{amssymb}
          \usepackage{amsmath}
         
          \usepackage{amsthm}
          \usepackage{amsfonts}
          \usepackage[english]{babel}
          \usepackage[utf8]{inputenc}
          \usepackage{enumerate}
          \usepackage{graphicx}
 
\usepackage[english]{babel}
\usepackage{amssymb}
\usepackage{amssymb}
\usepackage{amsthm}
\usepackage{amsmath}
\usepackage{bbm}
\usepackage{amscd}
\usepackage[cal=boondoxo]{mathalfa}
\usepackage[all]{xy}
\usepackage{subfigure}
\usepackage{microtype}
\usepackage{booktabs}
\usepackage{enumerate}
\usepackage{pgf}
\usepackage{cite}
\usepackage{pgf,tikz}
\usepackage{url}
\usepackage{tikz-cd}
\usetikzlibrary{arrows}

\usepackage{xcolor}
\usepackage{hyperref}

\usepackage{bookmark,hyperref}
\hypersetup{
     colorlinks=true,
     linkcolor=blue,
     filecolor=blue,
     citecolor = blue,
     urlcolor=cyan,
     }
 
\usepackage{enumitem}
\usepackage{multicol}
 
\newlist{multienum}{enumerate}{1}
\setlist[multienum]{
    label=\alph*),
    before=\begin{multicols}{2},
    after=\end{multicols}
}
 
\newlist{multiitem}{itemize}{1}
\setlist[multiitem]{
    label=\textbullet,
    before=\begin{multicols}{2},
    after=\end{multicols}
}

\xyoption{arc}
\usetikzlibrary{arrows,shapes,calc,positioning}
\usepackage{manfnt}

\newcommand{\noi}{{\noindent}}
\newcommand{\cF}{\mathcal{F}}
\newcommand{\M}{\mathcal{M}}
\newcommand{\R}{\mathbb{R}}
\newcommand{\Z}{\mathbb{Z}}
\newcommand{\Q}{\mathbb{Q}}
\newcommand{\C}{\mathbb{C}}
\newcommand{\T}{\mathbb{T}}

\newcommand{\sS}{\mathbb{S}}
\newcommand{\N}{\mathbb{N}}
\newcommand{\PSL}{PSL(2,\R)}
\renewcommand{\H}{\mathbb{H}}
\newcommand{\cL}{\mathcal{L}}
\newcommand{\cD}{\mathcal{D}}
\newcommand{\cS}{\mathcal{S}}
\newcommand{\cU}{\mathcal{U}}

\newcommand{\cA}{\mathcal{A}}

\newcommand{\hz}{\hat\Z}
\newcommand{\tg}{\tilde{g}}
\newcommand{\tM}{\tilde{M}}
\newcommand{\tx}{\tilde{x}}
\newcommand{\kh}{\mathfrak{h}}
\newcommand{\kg}{\mathfrak{g}}
\newcommand{\wg}{\widehat{\Gamma}}
\newcommand{\wpi}{\widehat{\pi_1(M)}}

\newcommand{\ie}{{\em i.e.,\,\,}}
 
\newtheorem{theorem}{Theorem}
\newtheorem{lemma}{Lemma}
\newtheorem{proposition}{Proposition}
\newtheorem{corollary}{Corollary}

\theoremstyle{definition}
\newtheorem{definition}{Definition}
\newtheorem{example}{Example}
\newtheorem{remark}{Remark}

\newtheoremstyle{named}{}{}{\itshape}{}{\bfseries}{.}{.5em}{\thmnote{#3} #1}
\theoremstyle{named}

\begin{document}
 
\title{Low-dimensional solenoidal manifolds}

\author[ Alberto Verjovsky]{Alberto Verjovsky}
\address{Instituto de Matem\'aticas,\\
Universidad Nacional Aut\'onoma de M\'exico\\
Apartado Postal 273,
Admon. de correos \#3,
C.P. 62251 Cuernavaca,
Morelos, México.}
\email{alberto@matcuer.unam.mx}
 
\date{}
\thanks{}
\maketitle
\vspace{-.7cm}
{\footnotesize
\centerline{\it To Dennis Sullivan on the occasion of his 80th birthday,}
\centerline{\it for all the inspiration and mathematical ideas he has shared with me.}
}
\vspace{-.4cm}

\begin{abstract} In this paper we survey $n$-dimensional solenoidal manifolds for $n=1,2$ and 3, and present new results about them.
Solenoidal manifolds of dimension $n$ are metric spaces locally modeled on the product of a Cantor set and an open $n$-dimensional disk. Therefore,
they can be ``laminated'' (or ``foliated'') by $n$-dimensional leaves.
 By a theorem of A. Clark and S. Hurder, topologically homogeneous, compact solenoidal manifolds are McCord solenoids
\ie are obtained as the inverse limit of an increasing tower of finite, regular covering spaces of a compact manifold with an infinite and residually finite fundamental group. In this case, their structure is very rich since they are principal
Cantor-group bundles over a compact manifold and behave like ``laminated'' versions of compact manifolds, thus they share many of their properties. These objects codify the commensurability properties of manifolds.
\end{abstract}

\tableofcontents

\section{Introduction} \label{introduction}

\noi  Given an integer $k\geq0$, a topological $k$-dimensional {\it solenoidal manifold} or {\it solenoidal lamination} or simply {\it $k$-dimensional solenoid}, is a metric space which is locally the product of a euclidean open $k$-disk and an infinite, compact, perfect and totally disconnected set  (\ie a Cantor set). If the space is locally the product of an open subset of the closed $n$-disk and the Cantor set, the space is a {\it solenoidal manifold with boundary}.
The spaces are laminated with $k$-dimensional leaves and with a Cantor set transverse structure. When $k=1$ sometimes we will simply call them {\it solenoids} and when
$k=2$ we will refer to them as {\it Riemann surface solenoidal laminations}.
These solenoidal manifolds appear naturally in many branches of mathematics.
For example,
in topology, the Vietoris-Van Dantzig $p$-adic solenoids \cite{VD, Vi} are some of the fundamental examples which motivated the development of new homology and cohomology theories that could apply to these spaces, for instance in the papers by Steenrod \cite{St1, St2}.

\noindent Solenoids appear naturally also as Pontryagin duals of discrete locally compact Hausdorff abelian groups. For instance, if ${\mathbb Q}$ denotes additive rationals with the discrete topology then
its Pontryagin dual   ${\mathbb Q}^*$ is the abelian universal 1-dimensional solenoid which is a compact abelian group that fibers over the circle ${\mathbb S}^1$ via an epimorphism
$p:{\mathbb Q}^*\to{\mathbb S}^1$ where the fiber is the Cantor group $\hz$ which is the profinite completion of the integers $\mathbb Z$.
This fact has a close relationship with the {\it adèles} and {\it idèles}  and its properties are the first steps in Tate's thesis \cite{T}.

\noi  Cantor groups \ie topological groups homeomorphic to the Cantor set,
 play an important role in the present paper. In particular,
 the profinite groups which are the profinite completions of
 residually finite groups.
 
 \noi A compact 0-dimensional, solenoidal manifold is just a Cantor set.
 
\noi Solenoids have a relevant role in dynamical systems, as basic sets of Axiom A diffeomorphisms in the sense of Smale \cite{Sm}. In particular, one-dimensional expanding attractors are solenoidal manifolds and were studied extensively by Robert Williams \cite{Wi}. In \cite{ALM} it is shown that
a similar construction as that in \cite{Wi} holds in higher dimensions: compact solenoidal manifolds without holonomy are obtained as inverse limits of coverings of branched manifolds.
 
\noi Some solenoidal manifolds encapsulate tiling spaces \cite{AP}, \cite{Gh1}, \cite{LS}, \cite{SW},
\cite{BBG}.
 
 \noi Laminations play an important role in holomorphic dynamics
\cite{Gh1,DNS, LM, LMM, Su2}.
Some results on dynamics on laminations (geodesic flows, horocycle flows,
harmonic measures) can be found, for instance, in  
\cite{ADMV, ADMV2, BM ,De, De1, MMV}.

\noi The papers \cite{CH}, \cite{CHL}, \cite{CHL1}, \cite{HL} by A. Clark, S. Hurder and O. Lukina deal with solenoidal manifolds which they call {\it matchbox manifolds}.
 
\noi A beautiful theory, covering many topics of solenoids, in particular regarding their ergodic properties, has been developed by Vicente Muñoz and Ricardo Pérez Marco \cite{MPM1, MPM2, MPM3, MPM4, MPM5, MPM6}.
 
\noi Higher-dimensional laminations with hyperbolic leaves have been considered
by Misha Kapovich in \cite{Ka2}.
 
\noi Lastly, one of the most fascinating subjects and achievements in mathematics of the first quarter of this century is the work of Peter Scholze and his collaborators on {\it Perfectoid spaces} \cite{Sch1, Sch2}. These spaces are closely related to solenoids. In fact, the prototype of a perfectoid space is the $p$-adic solenoid.
 
\smallskip
\noi The paper is organized as follows. In {\it section} \ref{Preliminaries} we introduce the definition of {\it solenoidal manifolds} and establish some of their properties. In {\it section} \ref{thsm} we study topologically homogeneous solenoidal manifolds. Some of the most important examples of compact $n$-dimensional solenoidal manifolds are obtained as inverse limits of an infinite tower of finite regular coverings of compact $n$-dimensional manifolds, with an infinite and residually finite fundamental group.
 
\noi A theorem of M. C. McCord implies that a solenoid obtained by such a tower
is homogeneous, in fact, it is a principal locally trivial fiber bundle over a compact $n$-manifold with fiber a Cantor group \cite{McC}.
 
\noi The reciprocal is a theorem of A. Clark and S. Hurder \cite{CH}
(see {\it Theorem} \ref{ClarkHurder} {\it section} \ref{thsm}). It states that
compact solenoidal manifolds which are also topologically homogeneous are McCord solenoids \ie they have
the property of being obtained as the inverse limit of
an infinite tower of regular coverings
of increasing order of a compact manifold. If we take the complete tower of
regular coverings the solenoid is the {\it algebraic universal covering} of $M$.
The Cantor group fiber is the {\it profinite completion} of the fundamental group.
All non simply connected compact surfaces have infinite, residually finite, fundamental groups, therefore we have a solenoidal 2-manifold obtained by an infinite tower of finite {\it regular} coverings of order $\geq2$.
Since the subgroups of finite index in the tower are normal,
the inverse limit is a principal fiber bundle over the surface with fiber the Cantor group which is the profinite completion of the sequence of finite groups of deck transformations. In  particular this
Cantor group acts transitively on the Cantor fiber and one can show that the Riemann surface lamination is topologically homogeneous. Haar measure is a holonomy--invariant transverse measure. In particular, these laminations are minimal (\ie every leaf is dense). MacCord solenoids endowed with a laminated Riemannian metric have a very natural ``volume form'' which is a measure that disintegrates into the volume form
(with respect to the Riemannian metric) in the leaves and disintegrates into the
Haar measure in the Cantor group transversals. In section \ref{Differentiable-McCord-solenoids}
we define {\it differentiable McCord solenoids} and in {\it proposition} \ref{smooth-McCord} (and {\it corollary} \ref{smooth}) we describe some of their properties. In \ref{exotic-examples} we propose examples of smooth 7-dimensional solenoidal manifolds that
we conjecture are homeomorphic but not diffeomorphic.
 We also give examples of solenoids that are potentially non-smoothable (and perhaps their leaves are not homeomorphic to a simplicial complex).
In \ref{suspension} we define the solenoids which are obtained as suspensions of representations of fundamental groups of compact manifolds.  In \ref{solenoidal-lifting} we deal with the problem of
lifting maps between solenoidal coverings (with the caveat that, because of lack of local connectedness, liftings of homeomorphisms to noncompact solenoids might not exist or might not be homeomorphisms).
 
\noi In {\it section} \ref{lhm} we introduce the notion of harmonic maps between Riemannian solenoids and we establish the solenoidal version of the theorem of Eells-Sampson: a differentiable map between Riemannian solenoids with target a solenoid with a laminated Riemannian metric having all leaves with negative sectional curvature is homotopic to a harmonic map.
 
\noi Compact, {\it oriented}, 1-dimensional solenoids are in one-to-one correspondence
with conjugacy classes of homeomorphism of the Cantor set. This is because Dennis Sullivan proved that these solenoids are mapping tori of homeomorphisms
of the Cantor set \cite{Su1,V}. In the first part of {\it section} \ref{1dsm} we use this fact to give some properties of these solenoids.
We show, following D. Sullivan \cite{Su1}, that these
1-dimensional laminations are null-cobordant \ie there is a compact, oriented,
2-dimensional Riemann surface lamination with boundary equal to the given
1-dimensional solenoid.
 
\noi  In {\it subsection} \ref{hom1d}
we describe these compact, oriented
1-dimensional solenoids in the case when they are topologically transitive.
By the results in {\it section} \ref{thsm}, compact solenoidal manifolds which are also topologically homogeneous have
the property of being obtained as the inverse limit of
an infinite tower of regular coverings
of increasing order of a compact manifold.  Therefore, a compact connected, topologically homogeneous 1-dimensional solenoidal manifold is a 1-dimensional compact abelian group,
since it is obtained as the inverse limit of a tower of coverings of the circle
$\sS^1$ and every covering of the circle is equivalent to a homomorphism of the form
$z\mapsto{z^n},\,n\in\N$. These groups are Pontryagin duals of dense subgroups of the rationals.
 
\noi {\it Section} \ref{srs} deals with solenoidal surfaces, their uniformization, {\it theorem} \ref{uniformization}, the Ricci flow and the uniformization of
hyperbolic solenoidal surfaces by the Ricci flow,
 {\it theorem} \ref{continuity} in {\it subsection} \ref{Ricci_Flow}.
 
\noi  In {\it section} \ref{UHL} we describe Dennis Sullivan's 2-dimensional {\it universal hyperbolic solenoid}. We use our version of the Eells-Sampson theorem to prove a version of the Earle-Eells theorem in \ref{EaE}.
 
 \noi In {\it section} \ref{hom2d} we study the
$n$-dimensional solenoidal compact abelian groups and the {\it universal
euclidean solenoid} when $n=2$. We indicate the complexity of their
classification when $n>1$.
 
\smallskip
\noi In {\it section} \ref{g3dsm} we study 3-dimensional locally homogeneous solenoidal manifolds. A theorem of Hempel \cite{He}, plus the geometrization theorem of Perelman
\cite{KL}, imply that the fundamental group of any 3-manifold is residually finite, in particular, every non-simply connected 3-manifold has a proper finite-sheeted covering space. Locally homogeneous, 3-dimensional solenoidal manifolds are the inverse limit of an infinite tower of regular coverings of finite order of a compact 3-manifold having one of the 7 non-spherical Thurston geometries.
A version of {\it Mostow rigidity theorem} (\ref{Mostow}), {\it Thurston hyperbolization theorem} \ref{hyperbolization} and {\it Virtual
fibration theorem} \ref{vft} for hyperbolic laminations hold.
These profinite solenoidal coverings of the 3-manifolds provide commensurability invariants.
 
\noi In the hyperbolic and solvable case we remark the relation of these laminations with number fields and their group of units.
 
\noi {\it We mainly deal with the hyperbolic and solvable case but we make a few comments for the other geometries.}

\section{Preliminaries}\label{Preliminaries}
 
\begin{definition}\label{definition:lamination}
 
\noi For the general definition of a {\em lamination} or {\em foliated
space}, we refer the reader to Chapter 11 of the book \cite{CC}.
A topological {\em solenoidal $n$-dimensional manifold} $M$ is a metrizable space $M$ endowed with an atlas $\cA=(U_\alpha, \varphi_\alpha)_{_{\alpha \in \mathfrak{I}}}$,  where $\left\{U_\alpha\right\}_{\alpha\in \mathfrak I}$ is an open cover of M and  $\mathfrak I$ a set of indices,
such that:
for each $\alpha \in \mathfrak{I}$, $\varphi_\alpha:U_\alpha\to{D\times T_\alpha}$ is a homeomorphism from $U_\alpha$ to a product $D\times T_\alpha$, where $D$
 is  the open unit disk in $\R^n$ and
 $T_\alpha$ is homeomorphic to the Cantor set. The inverse of
 $\varphi_\alpha^{-1}:=\overline\varphi_\alpha:D\times{T_\alpha}\to{U_\alpha}$ is called a \emph{local parametrization.}
 
\noi It follows from the definition that
 whenever $U_\alpha\cap U_\beta \neq \emptyset$, the
 {\it change of coordinates}
 $$\Psi_{\alpha\beta}:=\varphi_\beta\circ\varphi_\alpha^{-1}:\varphi_\alpha(U_\alpha\cap{U_\beta})
 \to\varphi_\beta(U_\alpha\cap{U_\beta})$$ is of the form
 $(x,t)\mapsto (\lambda_{\alpha\beta}(x,t), \tau_{\alpha\beta}(t))$.
\end{definition}

 \begin{definition}[Smooth solenoidal manifolds]\label{definition:difflamination}
The {\em solenoidal $n$-dimensional manifold} $M$ is called
a {\em smooth solenoidal $n$-dimensional manifold} if
for each fixed $t$ the map of change of coordinates
 \begin{equation}\label{changeofcoordinates}
  \varphi_\alpha({U_\alpha\cap U_\beta})\ni(\,\cdot\,,t)
  \xrightarrow{f_{_{(\alpha\beta,t)}}}(\lambda_{\alpha\beta}(\,\cdot\,,t), \tau_{\alpha\beta}(t))
 \end{equation}
  is smooth (of class $C^\infty$) in the $x$ variable and its derivatives
  of all orders depend continuously on $t$.

\noi For each $\alpha \in \mathfrak{I}$ we can orient $D$ with
 the standard orientation of $\R^n$. If for each fixed $t$ the map
 $(\cdot,t)\mapsto(\lambda_{\alpha\beta}(\cdot,t), \tau_{\alpha\beta}(t))$ preserves the orientation  we say that the lamination is {\it oriented}.
 
\end{definition}

\begin{definition}If for each $\alpha \in \mathfrak{I}$,
$\varphi_\alpha:U_\alpha\to{V_\alpha}\subset{\overline{D}\times T_\alpha}$ is a homeomorphism from $U_\alpha$ to an open set $V_\alpha$ in the  product
$\overline{D}\times T_\alpha$, where $\overline{D}$ is
the closed unit $n$-disk in $\R^n$ with boundary $\sS^{n-1}$, we call
$M$ an {\em solenoidal $n$-dimensional manifold with boundary}.
The boundary is defined in the same fashion as the usual definition for manifolds with boundary and it is an $(n-1)$-solenoidal manifold.
\end{definition}
 
\begin{remark}
We will also call a solenoidal $n$-dimensional manifold
an {\em $n$-dimensional Cantor lamination} or simply a
{\em lamination} if the Cantor transverse structure is understood.
 
\noi Of course, smooth solenoidal manifolds
are generalizations of smooth foliated manifolds, but the local holonomy is represented by local homeomorphisms of Cantor sets. We could define
foliations of Hölder class $C^r$ for any $r>0$ by imposing the
changes of charts to be of class $C^r$.
\end{remark}
\begin{definition}\label{definition:charts} The atlas $\cA=(U_\alpha, \varphi_\alpha)_{_{\alpha \in \mathfrak{I}}}$ is called a \emph{lamination atlas} and the maps $\varphi_\alpha$
are called \emph{flow boxes} or \emph{foliation charts}.
The sets $\varphi_\alpha^{-1}(D\times\{t\})$ are called {\em plaques}. Condition (2)
says that the plaques glue together to form maximal connected $n$-dimensional smooth manifolds,
called {\em leaves}, which are immersed in $M$.
 
\noi A set of the form $\varphi_\alpha^{-1}(\{y\}\times T_\alpha)$ is called
{\em transversal}.
 The associated {\it lamination}
$\cL$ is the partition of $M$ into leaves.
More precisely, the leaves are smooth $n$-dimensional manifolds whose local
charts are the restrictions of the maps $p_\alpha\circ\varphi_\alpha$ to plaques,
where $p_\alpha:D_\alpha\times T_\alpha\to D_\alpha$
is the projection onto the first factor.
\end{definition}
 
\begin{definition}\label{transverse-measure} If $M$ is an $n$-dimensional solenoidal manifold, a {\em holonomy invariant measure} is a family of Radon measures in all transversal Cantor sets (see {\it definition} \ref{definition:charts}) which are invariant
under the maps $\tau_{\alpha\beta}$ in formula (\ref{changeofcoordinates}),
{\it definition} \ref{definition:lamination},
 when restricted to their domains and
ranges. Analogously, a {\em holonomy invariant metric} is a metric in all transversals which is invariant under
the maps $\tau_{\alpha\beta}$ when restricted to their domains and
ranges. This means that these measures, or metrics, are invariant under
the holonomy pseudogroup.
\end{definition}

\begin{remark}\label{always_connected}
{\it We will be dealing mostly with solenoidal manifolds which are
compact, oriented, and connected. However, sometimes we consider ``covering'' solenoidal manifolds that are neither connected (or even locally connected) nor compact}.
 
\noi If the laminated atlas $\cA$ is understood we omit it
and simply refer to $M$ as a {\it solenoidal manifold}.
\end{remark}
 
\noi An important class of laminations consists of minimal laminations:
\begin{definition}\label{minimal} A lamination is said to be {\em minimal} if all the leaves are dense.
\end{definition}
\noi On a smooth lamination one can define the notions of differential topology and geometry along the leaves. For instance, one has concepts such as {\it laminated smooth functions, laminated Riemannian metrics, laminated curvature, de Rham theory}, etc. \cite{MS},\cite{CC}. In particular, we have the following:
\begin{definition}\label{smoothfunctions} $C^{(\infty,0)}(M)$ consists of
those continuous functions on $M$ which are smooth when restricted to the leaves
with the standard definition with respect to foliated charts. Analogously
we can define laminated smooth vector-valued functions $C^{(\infty,0)}(M,\R^n)$.
\end{definition}
One can define
in a natural way the ``laminated''
tangent bundle $T(M)$ of the lamination by considering all the
tangent spaces of the leaves at all the points of such leaves. All of the notions about differential topology and differential geometry of manifolds can be adapted to smooth laminations. In particular one has the notion of smooth vector fields, differential forms, tensors, Riemannian metrics, laminated connections, curvature along the leaves,
etc. We require these objects to be continuous and differentiable when restricted to each leaf. {\it In this paper we will mostly deal with differentiable solenoidal laminations}.
 
\begin{definition}\label{foliated_map} If $M_1$ and $M_2$ are solenoidal manifolds, a {\em smooth foliated map}
(or {\em smooth laminated map}) is a continuous map $f:M_1\to{M_2}$ such that $f$ sends leaves of $M_1$ into leaves of $M_2$ and, in terms of local charts, the restrictions
of $f$ to plaques of $M_1$ are differentiable functions into plaques of $M_2$.
Such a map we also call {\it leafwise differentiable map}. In particular, a diffeomorphism between two solenoidal manifolds $M_1$ and $M_2$ is a homeomorphism
$f:M_1\to{M_2}$ such that $f$ and $f^{-1}$ are leafwise differentiable.
\end{definition}
 
\section{Topologically homogeneous solenoidal manifolds}\label{thsm}
 
\noi Since all non-empty, compact, metric spaces which are perfect and totally disconnected are homeomorphic (to the standard ternary Cantor set) we simply speak about ``the'' Cantor set $K$. In this paper we will use very frequently the following:
 
\begin{definition}\label{homcantor} $\mathcal{H}(K)$ will denote the group of homeomorphisms of the Cantor set.
 \end{definition}

The most canonical examples of compact $n$-dimensional solenoidal manifolds
are locally-trivial fibrations $p: X\to{M}$ over a compact $n$-dimensional
connected manifold $M$ with fiber the Cantor set $K$. We assume that $X$ is connected. Fibrations like $p$ have the
{\it unique path-lifting property} and therefore they are a natural generalization
of the notion of covering spaces (Steenrod \cite{St0},  $\S 13$). This property allows us to define
holonomy along a path $f:[0,1]\to {M}$: a homeomorphism
$h_f:p^{-1}(f(0))\to{p^{-1}(f(1))}$. The holonomy $h_f$ only depends on the
homotopy class of $f$ with endpoints $f(0)$ and $f(1)$ fixed. If $x\in{M}$
then holonomy of loops based at $x$ gives a representation
$\rho:\pi_1(M,x)\to\mathcal{H}(K)$ called the holonomy of the fibration.
The leaves of the lamination are the path components of $X$.
 
 Let us recall some facts about compact solenoidal manifolds $\cL$ which are topologically
  homogeneous \ie the group of homeomorphisms $\mathcal{H}(\cL)$ of $\cL$ acts transitively on $\cL.$
 
 \medskip
 \noi In \cite{McC} M. C. McCord proves the following:
 
\begin{theorem}[M.C. McCord]\label{McCord}
If  $\mathfrak{S}$ is a compact solenoidal manifold which is obtained as the inverse limit of an infinite tower of {\em regular} finite
coverings of degree $d\geq2$ of a compact, connected, manifold $M$ then $\mathfrak{S}$ is {\bf topologically homogeneous} \cite{McC}. The solenoid $\mathfrak{S}$ is a connected principal
fiber bundle with fiber a Cantor group.
\end{theorem}
 
\begin{definition}[McCord Solenoids]\label{McCord-Solenoids} We will call solenoidal manifolds
$\mathfrak{S}$ obtained as the previous theorem {\em McCord solenoidal manifold based on the manifold $M$}.
\end{definition}
\noi The reciprocal of McCord theorem for {\em topologically homogeneous} solenoidal manifolds  $\mathfrak{S}$ (for any two points $x,y\in \mathfrak{S}$
there exists a homeomorphism $f$ of $\mathfrak{S}$ such that $f(x)=y$) was obtained by A. Clark, and S. Hurder in \cite{CH}.
(Related results are found in \cite{AF}, \cite{ALM}, \cite{CH}, \cite{CHL},\cite{Kee}, \cite{FO} and \cite{Kw}):

\begin{theorem}[A. Clark, S. Hurder]\label{ClarkHurder} If $\mathfrak{S}$ is a compact, connected, topologically homogeneous solenoidal manifold then $\mathfrak{S}$ is a McCord solenoidal manifold.
In particular, it is minimal: all leaves are dense.
\end{theorem}

\begin{remark}\label{chain-cofinal}
Any {\it countable} directed set $(A, \preceq)$ contains a cofinal subset
$C$ which is totally ordered and well-ordered, under $\preceq$ (\ie it is a chain). Therefore
in theorems (\ref{McCord}) and (\ref{ClarkHurder}) the tower of coverings
can be taken to be a linear sequence of finite-sheeted coverings:
\begin{equation}\label{chain-directed-set}
 \cdots \xrightarrow{p_4}{M_3}\xrightarrow{p_3}
 {M_2}\xrightarrow{p_2}{M_1}
 \xrightarrow{p_1}{M_0}=M,
\end{equation}
\noi corresponding to the sequence of normal subgroups of finite index of $\pi_1(M)$:
\[
\cdots\subset{G_4}\subset{G_3}\subset{G_2}\subset{G_1}\subset{G_0}=\pi_1(M)
\]
Therefore, in this paper, we can suppose that we are dealing with inverse limits of chains of finite-sheeted regular
coverings. Such inverse limits are given by special spaces of sequences:
\[
\cL={\lim_{\overset{p_{_{i}}}{\longleftarrow}}} \,\,
\{p_{_{i}}:
M_i\to{M_{i-1}},\, i\in\N\}=
\left\{\{x_{i-1}\}_{i\in\N}\,\, \mathbb{:}\,\forall i\geq1,\,\,x_i\in{M_i},\,\, p_i(x_i)=x_{i-1}\right\},
\]
where $p_i$ ($i\in\N$) is a finite-sheeted regular covering map of degree $d\geq2$.
Therefore $\cL$ has a natural topology as a compact subset of $\overset{\infty}{\underset{i=0}\prod}M_i$.
 \begin{definition}\label{shift}
 The {\it shift
map} $\sigma$ is the map $\{x_{i-1}\}_{i\in\N}\overset{\sigma}
\mapsto\{x_{i}\}_{i\in\N}$.
The map $\sigma$ is a homeomorphism. One has a {\it canonical projection}
\begin{equation}\label{cp}
\Pi:\cL\to{M}, \quad\quad \Pi(\{x_{i-1}\}_{i\in\N})=x_0
\end{equation}
\end{definition}
\end{remark}
 
\subsection{Differentiable McCord solenoids}\label{Differentiable-McCord-solenoids}
 Let $\cL$ be an $n$-dimensional McCord solenoid. Then, as pointed in {\it remark} \ref{chain-cofinal}, $\cL$ is obtained as the inverse limit of an infinite chain of finite-sheeted regular coverings,  
$\,\cdots\xrightarrow{f_3}M_2\xrightarrow{f_2}M_1\xrightarrow{f_{1}}{M_0}=M$, where $M$ is a connected topological $n$-manifold. By {\it theorem} \ref{McCord},
the canonical projection  $\Pi:\cL\to{M}$ (equation \ref{cp})
is a principal $\Gamma$-bundle where $\Gamma$
is a Cantor group. This group is obtained as a quotient of the profinite completion of the fundamental group of $M$ as explained in {\it remark} \ref{UPM} in
subsection \ref{profinite-McCord} below.
The Cantor group $\Gamma$ acts properly and freely on
the right on $\cL$ and preserves the fibers. For $g\in\Gamma$ and $x$ in
$\Gamma$ let
$x\cdot{g}$ denote the action of $g$ on $x$.
 
If $M$ admits a differentiable structure $\cD$, given by the smooth subatlas
$\mathcal{A}=\left\{\varphi_\alpha:\cU_\alpha\to\R^n:\alpha\in{I}\right\}$,
where $\cU_\alpha\subset{M}$ is open and contractible, then $\cL$ can be endowed with the
subatlas
$\tilde{\mathcal{A}}=\left\{\tilde\varphi_\alpha:\Pi^{-1}(\cU_\alpha)\to\R^n\times\Gamma\,:\alpha\in{I}\right\}$, where $\tilde\varphi_\alpha$ is a local trivialization determined
by a local section $\sigma_\alpha:\cU_\alpha\to\Pi^{-1}(\cU_\alpha)$
(which exists since $\cU_\alpha$ is contractible). If $x\in\cU_\alpha$, then any element of the fiber
$\Pi^{-1}(\{x\})$
is of the form $\sigma_\alpha(x)\cdot{g}$ for a unique $g\in\Gamma$ (here we use the group
structure in the fiber). The local trivialization $\tilde\varphi_\alpha$ is obtained as follows: for each $\alpha\in{I}$, there exists a continuous
projection $p_\alpha:\Pi^{-1}(\cU_\alpha)\to\Gamma$ defined by
$p_\alpha(z)=g$ if $z=\sigma_\alpha(\Pi(z))\cdot{g}$,
then $\tilde\varphi_\alpha(z)=(\varphi_\alpha(\Pi(z)),p_\alpha(z))$
(se \cite{St0} \S 13 pages 59--67). Therefore
${\tilde\varphi_\alpha}^{-1}(x,g)=\sigma_\alpha(\varphi_{\alpha}^{-1}
(x))\cdot{g}$, and we have:
\[
\tilde\varphi_{_\beta}\circ\tilde\varphi_\alpha^{-1}(x,g)=
\tilde\varphi_{_\beta}\left(\sigma_\alpha(\varphi_{\alpha}^{-1}(x))
\cdot{g}\right)
=\Big(\varphi_{_\beta}(\varphi_{\alpha}^{-1}(x)),p_\beta\big(\sigma_\alpha(\varphi_{\alpha}^{-1}(x))\big)\cdot{g}\Big),
\]
where $x\in\varphi_\beta(\cU_\alpha\cap\cU_\beta), \, g\in\Gamma$.
We note that $g_{\alpha\beta}:=p_\beta\big(\sigma_\alpha(\varphi_{\alpha}^{-1}(x))\big)\in\Gamma$ is independent of
$x\in\varphi_\beta(\cU_\alpha\cap\cU_\beta)$. Hence the change of coordinates
$\tilde\varphi_{_\beta}\circ\tilde\varphi_\alpha^{-1}$ is of the form:
\begin{equation}\label{smooth-changes}
\tilde\varphi_{_\beta}\circ\tilde\varphi_\alpha^{-1}(x,g)=
\Big(\varphi_{_\beta}(\varphi_{\alpha}^{-1}(x)),g_{\alpha\beta}\cdot{g}\Big)
\end{equation}
We endow $\cL$ with the differentiable structure given by the atlas $\tilde\cA$.
It follows directly from equation \ref{smooth-changes} the following:
\begin{proposition}\label{smooth-McCord}
Let $\cL$ be an $n$-dimensional McCord solenoid, obtained as the inverse limit of an infinite chain of finite-sheeted regular coverings
\[\cdots\xrightarrow{f_3}M_2\xrightarrow{f_2}M_1\xrightarrow{f_{1}}{M_0}=M,\]
\noi where $M$ is a  $\text{\bf differentiable}$ $n$-manifold. Then $\cL$ has a natural differentiable
structure (induced by the differentiable structure on $M$) so that $\cL$ is a
smooth $n$-dimensional solenoidal manifold as in {\it definition \ref{definition:difflamination}}. Furthermore, the canonical projection
$\Pi:\cL\to{M}$
is a differentiable principal $\Gamma$-bundle. More precisely:
\begin{enumerate}
\item $\Pi$ is a differentiable map and the restriction of $\Pi$ to a leaf
$L$ is a smooth covering map from $L$ to $M$.
\item \label{shift-equivalent} The shift map $\sigma:\cL\to\cL$
({\it definition} \ref{shift}) is a diffeomorphism. Therefore,
for all integers $k\geq0$, the solenoids
corresponding to the chains
\[\cdots \xrightarrow{p_4}{M_3}\xrightarrow{p_3}
 {M_2}\xrightarrow{p_2}{M_1}\xrightarrow{p_1}{M_0}=M\]
 \noi and
 \[\cdots \xrightarrow{p_{k+4}}{M_{k+3}}\xrightarrow{p_{k+3}}
 {M_{k+2}}\xrightarrow{p_{k+2}}{M_{k+1}}\xrightarrow{p_{k+1}}{M_k}\]
\noi are diffeomorphic solenoids.
 \item The right action of $\Gamma$ on
 $\cL$, given by the formula $F_g(z)=z\cdot{g}$,
is by solenoidal diffeomorphisms
\ie the map $z\overset{R_g}\mapsto{z\cdot{g}}\,\,(z\in\cL, g\in\Gamma)$, is differentiable.
In particular, the restriction of $R_g$ to a leaf $L_z$ ($z\in{L_z}$) is a diffeomorphism
from $L_z$ to $L_{z\cdot{g}}$.
\end{enumerate}
\end{proposition}
\begin{corollary}\label{smooth} Let $\cL$ and $M$, $\Pi$ and $\Gamma$ be as in {\it proposition} \ref{smooth-McCord}.
Let us assume that $M$ has a differentiable structure $\cD$. Let
$p:\tilde{M}\to{M}$ be the universal covering projection. We endow the universal covering manifold $\tilde{M}$ with the pullback of the differentiable structure $\cD$
(\ie the unique differentiable structure $\tilde\cD$ that makes $p$ a local diffeomorphism). Let
$\,\tilde\Pi:\tilde\cL\to\tilde{M}$ be the pullback under $p$ of the principal
$\Gamma$-bundle $\Pi:\cL\to{M}$. Since $\tilde\Pi$ admits a section, $\tilde\cL$
is the trivial bundle $\tilde{M}\times\Gamma$. On the other hand, since the bundle
$\Pi:\cL\to{M}$ has the {\it unique path lifting property} it follows that
if we lift loops based at a point $x\in{M}$ they determine homeomorphisms
of the fiber $F_x:=\Pi^{-1}(\{x\})\simeq\Gamma$. These homeomorphisms depend only
on the homotopy classes of the loops based at $x$, and are right-translations on the fiber $F_x$.
These liftings determine a homomorphism $\chi_x:\pi_1(M,x)\to\Gamma$ called the
{\it monodromy} of the fibration
(called \text{{\bf characteristic class}} in the book by Steenrod \cite{St0} \S 13).
If $x_1,x_2\in{M}$ then $\chi_{x_1}$ $\chi_{x_2}$
differ only by inner conjugation by an element in $\Gamma$.
The conjugation classes determine the principal
$\Gamma$-bundles over $M$ (\cite{St0} \S 138). The fundamental group
$\pi_1(M,x)$
acts freely, differentiable, and properly discontinuously on
$\tilde\cL=\tilde{M}\times\Gamma$ by the formula
$\Psi_g(x,h)=(\gamma_g(x), h\cdot\chi_x(g))$
(on the first factor the action is by deck
transformations and on the second by right-translations). Furthermore
\[
\cL=(\tilde{M}\times\Gamma)/\pi_1(M,x).
\]
\end{corollary}

\begin{remark} The differentiable structure on $\cL$ in {\it proposition}
 \ref{smooth-McCord} depends on the differentiable structure
 of the base manifold $M$. Furthermore, there might exist solenoidal
 manifolds whose leaves cannot be homeomorphic to simplicial complexes. The following two examples show the subtleties of these concepts.
\end{remark}
\subsection{Examples}\label{exotic-examples}
\begin{example} Let $f:\sS^6 \to \sS^6$ be an orientation-preserving
diffeomorphism of the 6-sphere which is not isotopic to the identity. Such a diffeomorphism exists by Milnor's celebrated result on the non-uniqueness of differentiable structures on the 7-sphere.
In fact, Milnor proves that there are 28
isotopy classes of orientation-preserving diffeomorphisms of $\sS^6$.
Let $M_f^7$ be the mapping torus
of $f$ endowed with its natural differentiable structure. Then $M^7_f$ is homeomorphic to $\sS^6\times\sS^1$ but not diffeomorphic to it. In fact
$M_f^7$ is an example of a locally trivial differentiable $\sS^6$-bundle
which is not the boundary of a {\it smooth} $\mathbb{D}^7$-bundle over the circle (the structure group of the sphere bundle cannot be reduced to a Lie group,
see \cite{Ve3} for details). Since the covering of order 28 of $M^7_f$ is the mapping
torus $M^7_{f^{28}}$ of $f^{28}$, which is isotopic to the identity
(thus $M^7_{f^{28}}$ is diffeomorphic to $\sS^6\times\sS^1$),  
it follows from {\it proposition} \ref{smooth-McCord} item \ref{shift-equivalent},
that the towers of all finite-sheeted
coverings of $M^7_f$ and $\sS^6\times\sS^1$ give smooth solenoidal
7-dimensional manifolds which are diffeomorphic solenoids. I conjecture that
if $p$ is a prime and we consider the infinite sequence of coverings of
$\sS^6$ and $M_f^7$ of indices $p, p^2, \cdots, p^n,\cdots,$ then the corresponding solenoids are not diffeomorphic.
On the other hand, results by W.C. Hsiang, and C.T.C. Wall \cite{HW}, show that
for $n>5$ every smooth fake $n$-torus $\tau^n$ is smoothly covered by the
standard torus $\T^n$. Therefore, the two solenoids obtained by the corresponding towers of all
finite-sheeted normal coverings are diffeomorphic.
 
\end{example}
A more drastic example is the following.
\begin{example}
By \cite{DFL}, for any $n>6$ there exist closed aspherical
$n$–manifolds $M^n$ that cannot be triangulated
(\ie they are not homeomorphic to a simplicial complex). Since the universal
covering of $M^n$ is contractible $\pi_1(M^n)$ has no
torsion; suppose, in addition, it is also residually finite. Then the $n$-dimensional compact solenoid obtained by all the finite-sheeted coverings is a topological solenoid and most probably its leaves cannot be triangulated
(and therefore the solenoid would not admit a differentiable structure). I conjecture that such an example exists. Of course, it is an interesting question to have a version of the {\it Hauptvermutung} (a subject very dear to Dennis Sullivan \cite{Su4}, \cite{Su5}) for solenoidal manifolds.
\end{example}
 
\begin{remark}[Differentiably homogeneous solenoidal manifolds] One can adapt the results in \cite{CH} and \cite{ALM} to show that if $\cL$ is a
{\it differentiable} $n$-dimensional solenoid which is {\it differentiably homogeneous}
(\ie given $x,y\in\cL$ there exists a diffeomorphism
({\it definition}   \ref{foliated_map})
$f:\cL\to\cL$ such that $f(x)=y$) then $\cL$ is diffeomorphic to the inverse limit of an
infinite tower of smooth regular coverings of compact smooth manifolds and the differentiable structure of $\cL$ is equivalent to the differentiable structure given by {\it proposition} \ref{smooth-McCord}. However, we will not deal with that here, and instead, we will assume that our solenoids
are of the aforesaid type.

\end{remark}
 
\begin{remark}\label{all-smooth} Henceforth we will assume that all the McCord solenoids are obtained
by towers of smooth coverings over a smooth manifold and are endowed with the
differentiable structure of {\it proposition} \ref{smooth-McCord}.
\end{remark}
 
 \subsection{Profinite completions and McCord Solenoids}\label{profinite-McCord}
 
\noi We recall, once again, that on a smooth lamination one can use all the tools of differential topology and geometry along the leaves. Concepts such as
laminated smooth functions, laminated Riemannian metric, laminated curvature, de Rham theory, etc. can be defined \cite{MS},\cite{CC}.

\noi Let $M$ be a {\it smooth}, compact, connected, manifold whose fundamental group
is infinite and residually finite. Let $\tilde{M}$ be its universal covering. Let
$g$ be a Riemannian metric on $M$.
 
\noi We take the tower of {\it all} of its finite pointed coverings
$p_{\beta}:(\tilde{M}_\beta,\tilde{x}_\beta)\to(M,x_0)$
corresponding to the directed system under inclusion $\mathfrak{B}$ of subgroups of finite index $\beta\subset\pi_1(M,x_0)$. Hence
if ${\beta_1},\, {\beta_2}\in\mathfrak{B}$ and $\beta_1\subset\beta_2$
we simply write the order as
$\beta_1\preceq\beta_2$ and one has
a corresponding pointed coverings
$(\tilde{M}_{\beta_1},\tilde{x}_{\beta_1})$ and $(\tilde{M}_{\beta_2},\tilde{x}_{\beta_2})$ and a bonding pointed covering map:
\[
p_{_{(\beta_1,\beta_2)}}:
(\tilde{M}_{\beta_1},\tilde{x}_{\beta_1})\to
(\tilde{M}_{\beta_2},\tilde{x}_{\beta_2}),\quad\quad \beta_1\preceq\beta_2
\]
 
\noi We endow each $\tilde{M}_\beta$
with the pullback metric under $p_\beta$.
Then the inverse limit
$$
\mathbb{M}={\lim_{\overset{p_{_{(\beta_1,\beta_2)}}}{\longleftarrow} }} \,\,
\{p_{_{(\beta_1,\beta_2)}}:
(\tilde{M}_{\beta_1},\tilde{x}_{\beta_1})\to
(\tilde{M}_{\beta_2},\tilde{x}_{\beta_2}),\,\,\,
{\beta_1},\,{\beta_2}\in\mathfrak{B},\,\,\beta_1\preceq\beta_2\}
$$
has the structure of a smooth lamination by
{{\it proposition} \ref{smooth-McCord}. In addition, $\mathbb{M}$
is a minimal n-dimensional solenoidal manifold (see {\it proposition} \ref{McCord-minimal} below). The pulled-back metrics
induce a smooth laminated Riemannian metric $\hat{\mathbf{g}}$ on $\mathbb{M}$. The foliation is minimal: all leaves are dense and all leaves are isometric to $\tilde{M}$ with the pull-back metric $\tilde{g}$ under the universal covering
$\tilde{M}\to{M}$. Furthermore, by construction, each covering projection $p_\beta$ is a local isometry.
 
\noi Since the directed set $\mathfrak{N}$ of all the normal subgroups of finite index of $\pi_1(M,x)$ is cofinal in $\mathfrak{B}$, the inverse limit corresponding to
$\mathfrak{N}$ defines the same lamination (the directed set of characteristic subgroups of finite index is cofinal so also defines the same lamination). We have:
 
\begin{equation}\label{towercovers}
\mathbb{M}={\lim_{\overset{p_{_{(\beta_1,\beta_2)}}}{\longleftarrow} }} \,\,
\{p_{_{(\beta_1,\beta_2)}}:
(\tilde{M}_{\beta_1},\tilde{x}_{\beta_1})\to
(\tilde{M}_{\beta_2},\tilde{x}_{\beta_2}),\,\,\,
{\beta_1},\,{\beta_2}\in\mathfrak{N},\,\beta_1\preceq\beta_2\}
\end{equation}
\[
\mathbb{M}=\left\{\{x_\beta\}_{_{\beta\in\mathfrak{N}}} \in\prod_{\beta\in\mathfrak{N}}\tilde{M}_{\beta}
\;:\;x_{_{\beta_1}}=p_{_{(\beta_1,\beta_2)}}(x_{_{\beta_2}}){\text{ for all }}\beta_1\preceq\beta_2{\text{ in }}\mathfrak{N}\right\}.
\]

\begin
{definition}\label{univ-alg-cov} $\mathbb{M}$ is called the {\em algebraic universal covering} of $M$. We also refer to. $\mathbb{M}$ as the
{\em algebraic universal solenoidal covering space}.
For each fixed $\alpha\in\mathfrak{N}$ there is a projection
\begin{equation}\label{canonical-projection}
\Pi_\alpha:\mathbb{M}\to{M_\alpha},\quad \quad \Pi_\alpha(\{x_\beta\}_{_{\beta\in\mathfrak{N}}})=x_\alpha
\end{equation}
\noi When $\beta_0=\pi_1(M,x)$, we have $M_{\beta_0}=M$. Let $\Pi:\mathbb{M}\to{M}$
denote the
corresponding projection $\Pi(\{x_\beta\}_{_{\beta\in\mathfrak{N}}})=x_0\in{M}$. We call $\Pi$ {\it the canonical projection}.
\end{definition}
\noi If $\beta\in\mathfrak{N}$ is a {\em normal }subgroup of finite index let $H_\beta=\pi_1(M)/\beta$ be the finite group which is the group
of deck transformations of the finite covering $p_\beta:\tilde{M}_\beta\to{M}$. The finite group $H_\beta$ acts by isometries on $M_\beta$.
 
\noi If $\beta_1\preceq\beta_2$, one has the epimorphism $q_{_{\beta_1,\beta_2}}:H_{\beta_1}\to{H_{\beta_2}}$.
\smallskip
 
\noi Then one has the following standard fact:
\begin{proposition}
The map $\Pi_\beta$ describes $\mathbb M$ as a principal fiber bundle over $M$ with group fiber the Cantor group
$\widehat{\pi_1(M)}$  where $\widehat{\pi_1(M)}$ is the profinite group which is given as the inverse limit:
\begin{equation}\label{profincompletion}
\widehat{\pi_1(M)}={\lim_{\overset{q_{_{\beta_1,\beta_2}}}{\longleftarrow}}} \,\,\{q_{_{\beta_1,\beta_2}}:H_{\beta_1}\to{H_{\beta_2}}
,\quad \beta_1,\beta_2\in\mathfrak{N},\,\beta_1\preceq\beta_2\}.
\end{equation}
\end{proposition}

\begin{definition}\label{alg-fund-grp} The group
$\widehat{\pi_1(M)}$  is called the {\bf profinite completion} of the fundamental group or {\bf algebraic fundamental group}
of $M$. It is a profinite Cantor group. Recall that the profinite completion of a group $\Gamma$ is the inverse limit of the directed
system of finite quotients of $\Gamma$. Note that $\widehat{\pi_1(M)}$ is a Cantor group and
$\widehat{\pi_1(M)}\subset\prod_{\beta\in\mathfrak{N}}H_{\beta_1}$.
We refer to \cite{RZ} for properties of profinite groups.
\end{definition}

\begin{definition}\label{inclusion-j} Since $\pi_1(M)$ is residually finite there is a {\it canonical monomorphism} $j:\pi_1(M)\to{\widehat{\pi_1(M)}}$ given by $j(g)=(({\mathfrak{p}}_\beta(g)))_{\beta\in\mathfrak{B}}$
 where $\mathfrak{p}_\beta$ is the canonical epimorphism
$\mathfrak{p}_\beta:\pi_1(M)\to\frac{\pi_1(M)}{H_\beta}$.
\end{definition}

\begin{remark}\label{G-injects-densely} Since $\pi_1(M)$ is residually finite it follows that $j$ is injective and $j(\pi_1(M))$ is dense in $\widehat{\pi_1(M)}$.
\end{remark}
\noi Let $\tilde{M}$ be the universal covering of $M$ corresponding to the trivial subgroup. By {\it corollary} \ref{smooth}, the group $\pi_1(M)$ acts differentiably on $\tilde{\mathbb{M}}:=\tilde{M}\times{\widehat{\pi_1(M)}}$
as follows:
\begin{equation}\label{profiniteaction}
f_{\gamma}(x,h)=(\gamma(x), j(\gamma)h),
\end{equation}
\noi where the action of $\gamma\in\pi_1(M)$ on $\tilde{M}$ is by isometric deck transformations with respect to $\tilde{g}$ and on $\widehat{\pi_1(M)}$ by left translation by $j(\gamma)$. The action is differentiable ({\it corollary} \ref{smooth}) free and properly discontinuous.

Let $j:\pi_1(M)\to\widehat{\pi_1(M)}$ be the canonical injection.
Then $\pi_1(M)\times\widehat{\pi_1(M)}$ acts on $\tilde{M}\times\widehat{\pi_1(M)}$
differentiably
(by {\it proposition} \ref{smooth-McCord} and {\it corollary} \ref{smooth}) as follows:
\begin{equation}\label{action-Gamma-hatGamma}
F_{(\gamma,\mathfrak{g})}(\tilde{x}, \mathfrak{h})=
({\gamma}(\tilde{x}),j(\gamma)\cdot\mathfrak{h}\cdot\mathfrak{g}^{-1})
\quad \left(\tilde{x}\in\tilde{M},\gamma\in\Gamma, \,\, \mathfrak{h},\,\mathfrak{g}
\in\widehat{\pi_1(M)}\right)
\end{equation}
This is indeed a left action:
 
\[
F_{(\gamma_2\cdot\gamma_1,\mathfrak{g}_2\cdot \mathfrak{g}_1)}
(\tilde{x}, \mathfrak{h})=
(({\gamma_2}{\gamma_1})(\tilde{x}),\,j(\gamma_2\gamma_1)\cdot
\mathfrak{h}\cdot(\mathfrak{g}_2\cdot\mathfrak{g}_1)^{-1})=
\]
 
\[
({\gamma_2}({\gamma_1}(\tilde{x})),\,j(\gamma_2)\cdot{j(\gamma_1)}\cdot
\mathfrak{h}\cdot\mathfrak{g}_1^{-1}\cdot\mathfrak{g}_2^{-1})=
F_{(\gamma_2,\mathfrak{g}_2)}
(F_{(\gamma_1,\mathfrak{g}_1)}(\tilde{x}, \mathfrak{h}))
\]

\noi We have:
 
\begin{proposition}\label{algcover}The solenoid $\mathbb M$ is obtained from
$\tilde{\mathbb M}=\tilde{M}\times{\widehat{\pi_1(M)}}$
as the orbit space of the left action (\ref{profiniteaction}) of $\pi_1(M)$ \ie
\begin{equation}\label{barM}
\,\mathbb{M}=\pi_1(M)\backslash\tilde{\mathbb M}=\pi_1(M)\backslash({\tilde{M}}
\times{\widehat{\pi_1(M)}})
\end{equation}
\end{proposition}
\noi $\mathbb{M}$ is sometimes denoted as $M\times_{_{\pi_1(M)}}\widehat{\pi_1(M)}$.
Consider the canonical projection:
\begin{equation}\label{pihat}
\hat{\Pi}:\tilde{\mathbb{M}}\to\mathbb{M},
\end{equation}
 
\noi Then $\hat{\Pi}$ is a solenoidal covering projection \ie each point in
$\mathbb{M}$ has an open neighborhood $U$ so that
${\hat\Pi}^{-1}(U)$ is a disjoint union of open sets and the restriction to
each of these open sets is a homeomorphism {\it onto} $U$ and ${\hat\Pi}$ has
the unique path-lifting property \cite{St0} \S 13.
 
\begin{proposition}$\mathbb{M}$ is compact and the solenoidal lamination of $\mathbb{M}$ is minimal.\label{McCord-minimal}
\end{proposition}
 
\begin{proof} It follows from the fact that $j(\pi_1(M))$ is dense in
$\widehat{\pi_1(M)}$ ({\it remark} \ref{G-injects-densely}).
 Also $\widehat{\pi_1(M)}$ acts on the right without fixed points on $\tilde{M}\times{\widehat{\pi_1(M)}}$ and acts simply-transitive on $\{\tilde{x}\}\times\widehat{\pi_1(M)}$
 by the formula $(\tilde{x},h)\overset{g}\mapsto(\tilde{x},hg)$ (formula (\ref{action-Gamma-hatGamma}) when $\gamma=Id$). Thus
the leaves on $\mathbb{M}$ are dense and isometric to the
universal covering $\tilde{M}$ with the lifted metric $\tilde{g}$ from the metric $g$ on $M$, in particular, they are simply connected.
\end{proof}
\begin{remark} Of course not every laminated Riemannian metric on $\mathbb{M}$ is obtained
pulling back a Riemannian metric $g$ on $M$, using the canonical
projection $\Pi$, even if
the laminated metric is very homogeneous. For
instance, compact 2-dimensional smooth McCord solenoids obtained from an infinite tower of
smooth regular coverings of a smooth compact, orientable, surface $\Sigma$ of genus
$h$ greater than one, have an infinite dimensional Teichmüler space of metrics of constant negative curvature -1 whereas the Teichmüller space of $\Sigma$ is of real dimension $6h-h$. In fact, in \cite{BV} it is shown that the Teichmüler space
of Sullivan's universal hyperbolic solenoid $\cL_\mathcal{h}$ \ie
the algebraic universal covering of $\Sigma$ is homeomorphic to the space
of continuous maps from the Cantor group $\widehat{\pi_1(\Sigma)}$ to the
Teichmüller space $\mathcal{T}(\cL_\mathcal{h})$
of $\Sigma$ (which shows that $\mathcal{T}(\Sigma)$ is
an infinite-dimensional Kähler manifold). See {\it subsection} \ref{UHL} below.
 \end{remark}
 
\noi Let $M$ be a compact manifold with infinite and
residually finite fundamental group $\pi_1(M)$.
Let
$\mathcal{G}:={G_1}\subset{G_2}\subset\cdots\subset{G}_i\subset\cdots$
and
$\mathcal{G}':{G'_1}\subset{G'_2}\subset\cdots\subset{G'}_i\subset\cdots$
be two infinite chains in the directed system of normal subgroups of
finite index of $\pi_1(M)$. Assume that $\mathcal{G}$ and $\mathcal{G}'$
correspond to the McCord solenoids $\cS_1$ and $\cS_2$.
 
\noi We say that $\mathcal{G}$ is equivalent to $\mathcal{G}'$ if there exists
a chain $\mathcal{G}''$ of normal subgroups such that
$\mathcal{G}''\supset\mathcal{G}'$,  $\,\,\mathcal{G}''\supset\mathcal{G}$
and both $\mathcal{G}'$ and $\mathcal{G}$ are cofinal in $\mathcal{G}''$.
This is an equivalence relation.
Two chains in the same equivalence class have homeomorphic corresponding inverse limits of coverings. These equivalence classes determine closed subgroups of the profinite completion $\widehat{\pi_1(M)}$.

By \cite{GS} there exists uncountably many non-homomorphic normal subgroups $\Gamma_\alpha$ of $\widehat{\pi_1(M)}$ if $M$ is a hyperbolic manifold.
 
\noi Furthermore, uncountably many of these subgroups are closed normal subgroups which are Cantor groups. Restricting, for each $\alpha$, the action (\ref{action-Gamma-hatGamma})
to the subgroup $\pi_1(M)\times\Gamma_\alpha$, of the previous type, we obtain that
the restricted action is free and proper (but not discontinuous). The orbit-space quotient
\begin{equation}\label{form-of-solenoids}
\cS_\alpha:=(\pi_1(M)\times{\Gamma_\alpha})\backslash(\tilde{M}\times\widehat{\pi_1(M)})
\end{equation}
is a McCord solenoid which is a principal bundle over $M$ with fiber
the Cantor group $\hat\Gamma_\alpha=\widehat{\pi_1(M)}/\Gamma_\alpha$. So that
$\Gamma_\alpha$ acts fiberwise on $\tilde{\mathbb{M}}$ and
$\cS_\alpha=\tilde{\mathbb{M}}/\Gamma_\alpha$.
Furthermore, there
is a natural solenoidal covering map $\Pi_\alpha:\mathbb{M}\to\cS_\alpha$
(see {\it definition} \ref{Solenoidal-covering-spaces} below).
\noi Summarizing:
 
\begin{proposition}[Profinite Galois correspondence of solenoidal coverings of hyperbolic manifolds]
\label{uncountable-subgroups} Let $M$ be a compact hyperbolic manifold. Then ${\widehat{\pi_1(M)}}$ has uncountably many closed $\rm{normal}$ subgroups
$\Gamma_\alpha$
with quotients, a Cantor group \cite{GS}. Therefore there exist uncountably many non-homeomorphic McCord solenoids $\cS_\alpha$ that are principal Cantor bundles over $M$. For every $\Gamma_\alpha$ there exists a solenoidal covering map
$\pi_\alpha:\mathbb{M}\to{\cS_\alpha}$ ({\it definition} \ref{Solenoidal-covering-spaces} below).
\end{proposition}
 
\begin{remark}[Universal property of $\mathbb{M}$]\label{UPM} Let $M$ be a compact
smooth manifold (not necessarily hyperbolic).
 Any McCord solenoid over $M$ corresponds to a closed normal subgroup
 $\Gamma_\alpha$ of
 ${\widehat{\pi_1(M)}}$ with quotient $\hat\Gamma_\alpha$ a Cantor group and it is of the form given by the
 formula (\ref{form-of-solenoids})
 \ie $\cS_\alpha=\tilde{\mathbb{M}}/\Gamma_\alpha$, and $\Gamma_{\alpha}$
 acts fiberwise and properly by right translations on the principal
 $\widehat{\pi_1(M)}$-bundle $\mathbb{M}$.
 \end{remark}
\begin{remark} This result is also valid for the profinite completion of the fundamental group
of any 3-manifold $M$ with infinite fundamental group. Therefore there exist uncountably many non-homeomorphic McCord solenoids $\cS_\alpha$ that are principal Cantor bundles over $M$, corresponding to closed normal subgroups $\Gamma_\alpha$ of
$\widehat{\pi_1(M)}$ such that $\pi_1(M)/\Gamma_\alpha$ is a Cantor group.
\end{remark}
 
\noi The algebraic universal covering solenoidal lamination $\mathbb{M}$ has an invariant normalized transverse measure obtained from
the Haar measure on the Cantor group $\widehat{\pi_1(M)}$. The same holds for an invariant transverse metric; one obtains one, for instance, by the formula
\[
d'(g,h)=\sup\left\{d(xgy,xhy): x,y  \in \widehat{\pi_1(M)}\right\},
\]
where $d$ is any distance in the Cantor group $\widehat{\pi_1(M)}$:
\begin{definition}[Transverse invariant measure and distance]\label{Transverse-measure} For each $x\in{M}$ let $K_x=\Pi^{-1}(\{x\})$ be the
fiber of the canonical projection (\ref{canonical-projection}) over $x$. Since $\Pi$ is a principal bundle, the fiber $K_x$ is homeomorphic to the Cantor
group $\widehat{\pi_1(M)}$, by choice of any point in the fiber.
Thus
we have the Haar measure $d\mu_x$ in $K_x$.
This measure is called the {\it canonical transverse measure} at $K_x$. This measure is invariant under the holonomy pseudogroup.
 
\medskip
\noi Let $d$ be any bi-invariant distance on the Cantor group
$\widehat{\pi_1(M)}$.  Then $d_x$ the corresponding distance on the fiber $K_x$
is a distance that is invariant under the holonomy pseudogroup.
\end{definition}
 
\noi To simplify the notation let $G=\widehat{\pi_1(M)}$, $\,\tilde{\mathbb{M}}=\tilde{M}\times{G}$, and $p_2:\tilde{M}\times{G}\to{G}$ be the projection into the second factor. Let $\rm{d}\mathbf{h}$ be the normalized Haar measure on $G$.

\noi Let $g$ be any laminated Riemannian metric on $\mathbb{M}$. We can pull-back
$g$ by the canonical projection
$\hat\Pi:\tilde{M}\times{G}\to\mathbb{M}$ to obtain a
laminated Riemannian metric $\tilde{g}$ on $\tilde{M}\times{G}$.
For each $\mathbf{k}$ in the Cantor group $G$ let
$\mathrm{d}\nu(\mathbf{k})$ be the measure given by the volume form
(of the metric $\tilde{g}$) on the
leaf $p_2^{-1}(\mathbf{k})=\tilde{M}\times{\{\mathbf{k}\}}$.
 
\begin{definition}\label{solenoidalmeasure} Let $\tilde\mu_g$ be the measure on $\tilde{\mathbb{M}}=\tilde{M}\times{G}$ such that for each continuous function with compact support $f:\tilde{M}\times{G}\to\R$ one has:

\[\int_{\tilde{\mathbb{M}}}f(\tilde{x},\mathbf{k})\,\mathrm{d}\tilde\mu_g
=\int _{G}\int _{p_2^{-1}(\mathbf{k})}f(\tilde{x},\mathbf{k})\,
\mathrm{d}\mathbf{h}\,\mathrm{d} \nu (\mathbf{k}).
\]
 In particular, for each measurable bounded set $E\subset\tilde{\mathbb{M}}$: \[\tilde\mu_g (E)=\int _{G}\,f_{_{E}}(\mathbf{k})\,\mathrm {d}\mathbf{h}(\mathbf{k}),\]
 
\noi with  $E_{\mathbf{k}}=E\cap{p_2^{-1}(\mathbf{k})}$ and $f_{_{E}}(\mathbf{k})$  the
 measure of $E_{\mathbf{k}}$ with respect to $\mathrm{d} \nu (\mathbf{k})$.
 \end{definition}
 
 \noi  The measure $\tilde\mu_g$ is invariant under the action of $G$ on
 $\tilde{\mathbb{M}}$ and therefore descends to a measure $\mu_g$ on the
 algebraic covering space $\mathbb{M}$. In a few words: the measure $\mu_g$
 is the measure that disintegrates into the Riemannian volume measure along the leaves and the Haar measure along the Cantor group transversals.
 \begin{remark} This type of measure can be defined for every McCord Solenoid
 or, equivalently, for any principal bundle with fiber a Cantor group or, again equivalently, a compact topologically homogeneous solenoidal manifold.
\end{remark}
 
\begin{remark}[Equicontinuity of McCord solenoids]\label{Equicontinuity of McCord solenoids}
The fact that the compact topologically homogeneous solenoidal manifolds admit
a transverse metric invariant under holonomy implies that these laminations are equicontinuous in the sense of \cite{AC}, thus they behave very
similarly to Riemannian foliations \cite{DHL},\cite{Mol}. In fact, it follows from S. Matsumoto in \cite{Mat}  that $\mathbb{M}$, being a minimal equicontinuous, compact lamination, is uniquely ergodic \ie it has a unique (up to scaling) transverse invariant Radon measure.
 \end{remark}
 
\begin{definition}[Solenoidal volume form]\label{lamvol} If $\cL$ is a smooth McCord solenoid with a laminated Riemannian metric $g$, the measure $\mu_g$ defined above is called the
{\it laminated volume form} of the McCord with respect to the laminated Riemannian metric $g$.
 
\noi The volume of $\cL$ with respect to the laminated volume form is:
\begin{equation}vol(\cL)=\int_{{\cL}}\,\rm{d}\mu_{g,h}
\end{equation}\label{volume}

\end{definition}

\subsection{The suspension of a representation}\label{suspension} Let $M$  be a compact $n$-dimensional connected
smooth manifold.
Let $\rho:\pi_1(M)\to \mathcal{H}(C)$ be a representation of $\pi_1(M)$ into
the group of homeomorphisms of the Cantor set $\mathcal{H}(C)$. As before,
let $\tilde{M}$ be the universal covering of $M$. Let $\pi_1(M)$ act on
$\tilde{M}\times{K}$ as follows:
\begin{equation}\label{actionsusp}
\gamma(m,\mathbf{k})=(\gamma(m),\rho(\gamma)(\mathbf{k})),\quad \gamma\in\pi_1(M),
 \,\,\,\tilde{m}\in\tilde{M},  \,\,\mathbf{k}\in{K}.
\end{equation}
 \noi The action on $\tilde{M}$ is by deck transformations and in the Cantor set
 is via the representation $\rho$.
 The following holds:
 \begin{proposition}\label{actionsusp-free-disc} The action of $\pi_1(M)$ given by formula (\ref{actionsusp})
 is free and properly discontinuous and co-compact \ie the orbit space is compact with respect to the quotient topology.
 \end{proposition}
 \begin{definition}
 By {\it proposition} \ref{actionsusp-free-disc} the quotient
  \begin{equation}
 \mathbb{M}_\rho=\pi_1(M)\backslash(\tilde{M}\times{\widehat{\pi_1(M)}})
 \end{equation}
 is a compact Hausdorff space
called the {\it suspension} of the representation $\rho$. The suspension
 of $\rho$ is of sometimes denoted $M\times_\rho{K}$.
 
 \noi There is a natural projection $\mathfrak{p}: \mathbb{M}_\rho\to{M}$ which is
 a locally trivial fibration with fiber $K$. The holonomy at any point $x\in{M}$
 is conjugate to $\rho$.
 
 \noi The space $\tilde{M}\times{K}$ is obviously a noncompact solenoidal manifold and the canonical projection
 $\Pi:\tilde{M}\times{K}\to{\mathbb{M}_\rho}=M\times_\rho{K}$ sends the leaves
 $(\tilde{M}\,,\,\left\{\mathbf{k}\right\})$ into the leaves of $M\times_\rho{K}$.
It is a solenoidal covering as in definition \ref{Solenoidal-covering-spaces} below.
\end{definition}
 
\begin{remark} As we have seen, a McCord solenoid is the suspension
of a representation $\rho:\pi_1(M)\to{G}$, where $G$ is a Cantor group acting on itself by left translations.
The suspension of a representation $\rho:\pi_1(M)\to\mathcal{H}(K)$ is not, in general, a McCord solenoid. For this to be true it is necessary that
the representation is by (left) translations on a Cantor group and every orbit
in the Cantor set, under the representation, must be dense.
\end{remark}
 
 \subsection{Lifting maps between McCord Solenoids}\label{solenoidal-lifting} Let
 \begin{equation}\label{lifting-canonical-projection}
\Pi_1:\mathbb{M}_1\to{M_1},\,\,\,\text{and} \quad\Pi_2:\mathbb{M}_2\to{M_2}.
\end{equation}
\noi be the two algebraic universal coverings of the compact, connected, smooth manifolds $M_1$ and $M_2$ (both with infinite and residually finite fundamental groups), respectively.
 
\noi $\mathbb{M}_1$ and $\mathbb{M}_2$ are compact solenoids
which are Cantor group fiber bundles over $M_1$ and $M_2$, respectively, and
are laminated with simply connected leaves homeomorphic to the universal coverings of $M_1$ and $N_2$, respectively.
 
\noi Let $G_1=\widehat{\pi_1(M_1)}$, $\,\tilde{\mathbb{M}}_1=\tilde{M}_1\times{G_1}$, and $G_2=\widehat{\pi_1(M_2)}$, $\,\tilde{\mathbb{M}}_2=\tilde{M}_2\times{G_2}$.
Let $\hat\Pi_1$ and $\hat\Pi_2$ be the corresponding solenoidal coverings (formula (\ref{pihat})):
\begin{equation}\label{universal-canonical-projection}
\hat\Pi_1:\tilde{\mathbb{M}}_1\to{\mathbb{M}}_1,\,\,\,\text{and} \quad\hat\Pi_2:\tilde{\mathbb{M}}_2\to{\mathbb{M}}_2.
\end{equation}
\begin{remark}\label{leaves-homeo-leaves}It follows from the definitions that
$\hat\Pi_i$  
maps leaves of $\tilde{\mathbb{M}}_i$ homeomorphically
onto the leaves of ${\mathbb{M}}_i$, $i=1,2$.
\end{remark}
 
\begin{definition} Let $f_1, f_2:\mathbb{M}_1\to\mathbb{M}_2$ be two
continuous maps. We say that $f_1$ is {\it leafwise equivalent} to $f_2$
if $f_1(L)$ and $f_2(L)$ are contained in the same leaf of $\mathbb{M}_2$ for
every leaf $L$ of $\mathbb{M}_1$. Two homotopic maps are leafwise equivalent.
 
\end{definition}
\noi Henceforth in this subsection we assume that $\mathbb{M}_2$ has a laminated  Riemannian metric $h$ which renders all leaves with negative sectional curvature.
 Then all the leaves of $\mathbb{M}_2$ are simply connected Hadamard manifolds and
 therefore two points in the same leaf are connected by a unique geodesic parametrized by arc length.
 
 \noi Furthermore, unit speed geodesics in Hadamard manifolds depend continuously on their endpoints.
 Using these unique parametrized geodesics it is very easy to obtain a homotopy between any two leafwise equivalent
 maps from $\mathbb{M}_1$ to $\mathbb{M}_2$.
 
 \noi Since maps between solenoids necessarily send leaves into leaves, in our case,  we only need to know how a map ``shuffles'' the leaves to determine its homotopy class.
 
\noi By analogy with the usual definition of covering maps one has the following definition:
 \begin{definition}[Solenoidal covering spaces]\label{Solenoidal-covering-spaces}
  Let $\cS_1$ and $\cS_2$ be two $n$-dimensional solenoidal manifolds (not necessarily compact or connected).
 A continuous map $p:\cS_1\to\cS_2$ is called a {\it solenoidal covering map} if
 every point of $\cS_2$ is {\it evenly covered} by $p$.
 This means that every
 point $x\in\cS_2$ has an open neighborhood of the form:
$\mathcal{U}=D\times{K}$, where $D$ is an open disk in $\R^n$ and $K$ the Cantor set, such that
$p^{-1}(\mathcal{U})=\underset{\alpha\in\mathcal{I}} \bigsqcup
 \,\mathcal{V}_\alpha$, where $\mathcal{V}_\alpha$ is open, and
the restriction of $p$ to each $\mathcal{V}_\alpha$
 is a homeomorphism onto $\mathcal{U}$.
 \end{definition}
 
\begin{remark}\label{unique-path-lifting}
 Solenoidal covering maps have the unique path-lifting property and the homotopy-lifting property. However, since the spaces are neither connected nor locally connected the standard lifting theorems do not apply. For instance, if
 $f:\cS\to\cS$ is a homeomorphism of a compact McCord solenoid it may happen that
 a lifting to $\tilde\cS$ is not bijective. Examples happen even in dimension one
 (see, for instance, examples in \cite{Kw}).
 \end{remark}
 \noi The mappings  $\hat\Pi_1$ and $\hat\Pi_2$ in equation
  (\ref{universal-canonical-projection}) are solenoidal coverings maps.
 
 \noi We have the following lifting theorem for maps between McCord solenoids:
 \begin{theorem}\label{solenoid-lifting} Let $\mathbb{M}_1$ and $\mathbb{M}_2$
 be McCord solenoids and $\hat\Pi_1$, $\hat\Pi_2$ as in equation \ref{universal-canonical-projection}.
 Let $f:\mathbb{M}_1\to\mathbb{M}_2$ be a continuous map.
 There exists a continuous map
 $\tilde{f}:\tilde{\mathbb{M}}_1\to\tilde{\mathbb{M}}_2$ such that
 $f\circ\hat\Pi_1=\hat\Pi_2\circ\tilde{f}$. The map $\tilde{f}$ is of the form
 \begin{equation}\label{lift}
 \tilde{f}(y,\mathfrak{g_1})=(\tilde{f}_{\mathfrak{g_1}}(y),h(\mathfrak{g_1})),
 \end{equation}
 where $h:G_1\to{G_2}$ is continuous, and $\tilde{f}_{\mathfrak{g_1}}$ is a continuous map from $\tilde{M}_1\times\{\mathfrak{g_1}\}$ to
 $\tilde{M}_2\times\{h(\mathfrak{g_1})\}$.
 \end{theorem}
 \begin{proof} Let $p:\tilde{M}_1\to{M}_1$ be the universal covering of
 $M_1$.  Let $x\in{M}_1$ and $\tilde{x}$ such that $p(\tilde{x})=x$. Let
 $G_1(x)=\Pi_1^{-1}(\{x\})\sim{G_1}$ be the Cantor fiber over $x$ of the
 fibration $\Pi_1:\tilde{\mathbb{M}}_1\to{M}_1$. Let $f_x:G_1(x)\to\mathbb{M}_2$
 be the restriction of $f$ to $G_1(x)$.
 
 \noi  We can cover the Cantor set $G_1(x)$  with
 a finite number of disjoint {\it clopen} sets (Cantor sets) of small diameter, such that their images under $f$ have diameter smaller than the Lebesgue number of
 a cover of $\mathbb{M}_2$ by open
 sets that are evenly covered.  We can restrict $f_x$ to these small Cantor sets and ``lift'' them using the inverses of the local homeomorphisms
 on the evenly covered open sets. Then, using these liftings we see that
 $f_x$ admits a ``lifting'' \ie there exists a continuous function
 $\tilde{f}_x:\{\tilde{x}\}\times{G_1}\subset\tilde{\mathbb{M}}_1
 \rightarrow\tilde{M}_2\times{G_2}=\tilde{\mathbb{M}}_2$ such that
 ${f_x}\circ{i_x}=\hat\Pi_2\circ\tilde{f}_x$. Here ${i_x}$ is
 the restriction of $\hat\Pi_1$ to $\{\tilde{x}\}\times{G_1}$ so that
 ${i_x}(\{\tilde{x}\}\times{G_1})=G_1(x)$.
 
 \noi Let
 $p_2:\tilde{M}_2\times{G_2}\to{G_2}$ be the projection onto the second factor.
 The map $\tilde{f}_x$ induces a continuous map $h:G_1\to{G_2}$ as follows:
 $h(\mathfrak{g_1})=p_2(\tilde{f}_x(\{\tilde{x}\}, \mathfrak{g}_1))$.
 
 \noi The leaves of $\tilde{\mathbb{M}}_1$ and $\tilde{\mathbb{M}}_2,$
 are of the form
 $\tilde{L}^1_{\mathfrak{g}_1}:=\tilde{M}_1\times\{\mathfrak{g}_1\}\,\,$
 and $\tilde{L}^2_{\mathfrak{g}_2}:=\tilde{M}_2\times\{\mathfrak{g}_2\}$,  
 with $\mathfrak{g}_1\in{G_1}$ and
 $\mathfrak{g}_2\in{G_2}$, respectively.
 
 \noi By {\it remark} \ref{leaves-homeo-leaves}
 $L^1_{\mathfrak{g}_1}:=\hat\Pi_1(\tilde{L}^1_{\mathfrak{g}_1})$
 is a leaf of $\mathbb{M}_1$
 and $L^2_{h({\mathfrak{g}_1})}:=\hat\Pi_2(\tilde{L}^2_{h(\mathfrak{g}_1)})$ is a leaf of $\mathbb{M}_2$, and  $f$ maps $L^1_{\mathfrak{g}_1}$ into
 $L^2_{h({\mathfrak{g}_1})}$. The restrictions of $\hat\Pi_1$ and $\hat\Pi_2$
 to these leaves are homeomorphisms (with respect to the manifold topologies of the leaves). {\it We denote these restricted homeomorphisms by the same symbols}.
 
 \noi Define  $\tilde{f}_{\mathfrak{g}_1}:\tilde{L}^1_{\mathfrak{g}_1}
 \to\tilde{L}^2_{h(\mathfrak{g}_1)}$ as the following composition:
 \[
 \tilde{f}_{\mathfrak{g}_1}= (\hat\Pi_2)^{-1}
\circ{f}\circ\hat\Pi_1.
\]
\noi With this map  $\tilde{f}_{\mathfrak{g}_1}$, which depends continuously
upon ${\mathfrak{g}_1}$,
 the  function given by the formula (\ref{lift}) is a lifting of $f$.
\end{proof}
\begin{remark}[{\bf Caveat:} liftings of homeomorphisms might not be homeomorphisms]
\label{liftfail} The usual lifting properties can fail.
Let $M$ be a compact manifold with infinite residually finite fundamental group
 $\pi_1(M)$. Let $G=\widehat{\pi_1(M)}$. Let $\tilde{\mathbb{M}}=\tilde{M}\times{G}$ and $\hat\Pi:\tilde{\mathbb{M}}\to{\mathbb{M}}$ the canonical projection onto the algebraic universal covering $\Pi:\mathbb{M}\to{M}$ given in {\it definition} \ref{univ-alg-cov}, then $\hat\Pi$ is a solenoidal covering map.
 
\noi  Endow $\pi_1(M)$
 with the discrete topology. Let $\varphi:G\to\pi_1(M)$ be a continuous map
 and let $F_\varphi:\tilde{\mathbb{M}}=\tilde{M}\times{G}\to
 \tilde{\mathbb{M}}=\tilde{M}\times{G}$ be given by the
 formula:
 \[
 F_\varphi(x,\mathfrak{g})=
 (\varphi(\mathfrak{g})(x), j(\varphi(\mathfrak{g}))\cdot{\mathfrak{g}}),
 \]
 where $j:\pi_1(M)\to{G}$ is the canonical monomorphism. As before,
 the action on the left of $\pi_1(M)$ on $\tilde{M}$ is by deck transformations
 and on $G$ by left translations.
 
 \noi Then $\hat\Pi\circ{F_{\varphi}}=\hat\Pi$,
 therefore $F_\varphi$ is a lifting to $\tilde{\mathbb{M}}$ of the identity map
 $I:\mathbb{M}\to\mathbb{M}$. There are uncountably many such liftings. {\it Notice that we don't have a notion of uniqueness of these liftings}, however in this paper we use only the existence of a lifting.
\end{remark}
 
\noi  We have the following corollary using the construction of a lifting in {\it theorem} \ref{solenoid-lifting}:
 
\begin{corollary} With the same protagonists of the previous {\it remark} \ref{liftfail}, suppose
that $\mathbb{M}$ has a laminated metric with leaves of strictly negative curvature. Then:
 
\noi given any continuous map $f:\mathbb{M}\to\mathbb{M}$ such
that $f(L)=L$ for every leaf of $\mathbb{M}$ (\ie $f$
does not ``shuffle'' the leaves)  there exist a lift
$\tilde{f}:\tilde{\mathbb{M}}\to\tilde{\mathbb{M}}$ of the form:
\begin{equation}\label{goodlift}
\tilde{f}(x,\mathfrak{g})=(\tilde{f}_\mathfrak{g}(x),\mathfrak{g})
\end{equation}
\end{corollary}
 
\begin{proof} We follow the proof of {\it theorem} \ref{solenoid-lifting} when
$\mathbb{M}_1=\mathbb{M}_2=\mathbb{M}$ , $G_1=G_2=G$. We have to prove
that $h$ can be chosen to be the identity map on $G$.
 
\noi Let $x\in{M}$ and $G(x)=\Pi^{-1}(\{x\})\subset\mathbb{M}$ the
Cantor fiber over $x$. By choosing a point in $G(x)$  we can identify
$G(x)$ with the Cantor group $G$.
 
 Since the leaves of $\mathbb{M}$
are Hadamard spaces, and $f$ preserves each leaf, for each
$\mathfrak{g}\in{G(x)}$ there exists a unique parametrized
geodesic joining $\mathfrak{g}$ to $f(\mathfrak{g})$. We can reparametrize the geodesic to have a map
$\gamma_\mathfrak{g}:[0,1]\to{\mathbb{M}}$,
$\gamma_\mathfrak{g}(0)=\mathfrak{g}$,
$\gamma_\mathfrak{g}(1)=f(\mathfrak{g})$.
Since the metric is continuous in the transversal Cantor structure we can assume that the map $\mathfrak{g}\mapsto\gamma_\mathfrak{g}$ is a continuous map
from $G(x)$ to $C^0([0,1],\mathbb{M})$.
 
\noi Let $p:\tilde{M}\to{M}$ be the universal covering of
 $M$ and $\tilde{x}\in\tilde{M}$ be such that $p(\tilde{x})=x$. Let ${i_x}$ be
 the restriction of $\hat\Pi$ to $\{\tilde{x}\}\times{G}$ so that
 ${i_x}(\{\tilde{x}\}\times\{\mathfrak{g}\})=\mathfrak{g}$ and
 ${i_x}(\{\tilde{x}\}\times{G})=G(x)$.
 
\noi Each $\gamma_\mathfrak{g}$ can be lifted to a map
$\tilde\gamma_\mathfrak{g}:[0,1]\to\tilde{\mathbb{M}}$ so that
$\tilde\gamma_\mathfrak{g}(0)=(\tilde{x},\mathfrak{g})$. Then as in
{\it theorem} \ref{solenoid-lifting}
we can lift
the restriction $f_x$ of $f$ to $G(x)$ as the function
$\tilde{f}_x(\tilde{x},\mathfrak{g})=\tilde\gamma_\mathfrak{g}(0)$ This implies
that $h:G\to{G}$ is the identity and therefore the lifting granted by theorem (\ref{solenoid-lifting})
is of the form given by the formula (\ref{goodlift}).
\end{proof}

\section{Laminated harmonic maps and the Eells-Sampson Theorem}\label{lhm}
\noi In 1964
groundbreaking paper \cite{ES} J. Eells and J. Sampson started the theory of harmonic maps between Riemannian manifolds. They showed that for certain Riemannian manifolds, arbitrary maps could be deformed into harmonic maps. In particular, they showed that given $(M, g)$ and $(N, h)$, two smooth and closed Riemannian manifolds, where the sectional curvature of $(N, h)$ is nonpositive, then for any $C^\infty$ map $f:M\to{M}$ the maximal harmonic heat flow
$\left\{f_t\,:\,0<t<T \right\}$, with $f_0=f$, can be prolonged to $T=\infty$.
Using a result by Philip Hartman the maps $f_t$  converge strongly in the $C^\infty$ topology to a harmonic map.
In particular, this shows that, if $(N, h)$, has strictly negative
   curvature then every continuous map is homotopic to a unique harmonic map.
Their work was the inspiration for Richard Hamilton's initial work on the Ricci flow and the crowning work of G. Perelman about the {\it Geometrization Theorem}.

\noi Let $M_1$ and $M_2$ be two compact smooth manifolds with infinite residually finite fundamental groups.
 
\noi Let
$(\mathcal{S}_1, g)$ and $(\mathcal{S}_2, h)$ be the two compact, smooth, McCord solenoids of dimensions $m$ and $n$, respectively, which are algebraic coverings of the compact manifolds $M_1$ and $M_2$, respectively.
 
\noi Endow $(\mathcal{S}_1, g)$ and $(\mathcal{S}_2, h)$
with laminated Riemannian metrics $g$ and $h$, respectively. Thus we may think of
$\mathcal{S}_1$ and $\mathcal{S}_1$ as principal bundles with fibers the Cantor
groups $G_1$ and $G_2$, respectively.
 
\noi Let $f:\mathcal{S}_1\to\mathcal{S}_2$ be a leafwise $C^\infty$ map. By continuity, such a map sends leaves to leaves and it is a differentiable map
from one leaf into its image. For such maps, we define the {\it energy} functional
as follows:
\begin{definition} The {\it Dirichlet energy} of $f$ is defined by the equation:
\begin{equation}\label{energy}
E(f)=\frac12\int_{\mathcal{S}_1}||df(x)||^2\,\rm{d}\mu_{g}(x)
\end{equation}
\end{definition}
 
\noi In this formula $df$ is a section of the bundle  
$T^*(\mathcal{S}_1)\otimes{f^{-1}T(\mathcal{S}_2)}$ and
$\mu_{g}(x)$ is the  laminated volume form of $\mathcal{S}_1$, with respect to $g$,
as in {\it definition} \ref{lamvol}. The laminated metrics on
$\mathcal{S}_1$ and $\mathcal{S}_2$ induce a bundle metric on this bundle
and $||\cdot||_{(g,h)}$ is the associated norm (Hilbert-Schmidt norm).  
Since $f$ sends leaves to leaves, $f$ is expressed in terms of local coordinates
as follows:
\begin{equation}\label{maplocally}
f(x_1,\cdots,x_m; \mathfrak{g})=
 \end{equation}\label{mapincoordinates}
\hfil $\left(f_1(x_1,\cdots,x_m; h(\mathfrak{g})),
f_2(x_1,\cdots,x_m;h(\mathfrak{g})),\cdots,
f_n(x_1,\cdots,x_m;h(\mathfrak{g}))\right)$,
\newline with $\mathfrak{g}\in{G_1}$ and $h:\mathcal{V}\to{G_2}$ a continuous
map from the open set $\mathcal{V}\in{G_1}$ to $G_2$. For each $\mathfrak{g}\in{G_1}$
fixed, the map in formula (\ref{maplocally}) is $C^\infty$-differentiable with respect to the variables $x_i's$. In terms of these local laminated coordinates
the Hilbert-Schmidt norm is given by the formula:
\[
||df(x,\mathfrak{g})||^2=
\]
\begin{equation}\label{energycoordinates}
g^{ij}(x,\mathfrak{g})h_{\alpha\beta}(f(x,\mathfrak{g}))\frac{\partial{f}_\alpha(x,\mathfrak{g})}{\partial{x_i}}
\frac{\partial{f}_\beta(x,\mathfrak{g})}{\partial{x_i}},\quad  x=(x_1,\cdots,x_m),\, \mathfrak{g}\in{G_1}
\end{equation}
Here, we use the usual notation in terms of coordinates: $g^{ij}$ are the coefficients of the inverse of the matrix corresponding
to the metric $g$ and $h_{\alpha\beta}$ the coefficients of the matrix corresponding to the metric $h$ and we use the Einstein summation convention.
 
\begin{definition}\label{energy-density} The {\it energy density} of $f$ is the function $e(f):\mathcal{S}_1\to\R^{\geq0}$
defined, in local coordinates, by the formula:
\begin{equation}
e(f)(x,\mathfrak{g})=\frac12||df(x,\mathfrak{g})||^2
\end{equation}
\end{definition}
 
\begin{definition} A leafwise $C^\infty$ map
 $f:\mathcal{S}_1\to\mathcal{S}_2$ is set to be {\it leafwise harmonic} if it is an extremal of the energy functional (\ref{energy}).
\end{definition}
 
\noi The Euler-Lagrange equations corresponding to this variational problem imply
$f:\mathcal{S}_1\to\mathcal{S}_1$ is harmonic if and only if
its {\it tension vector field}  $\tau(f)$ vanishes:
\begin{equation}\label{harmonic-equation}
\tau(f)=\rm{trace}_g\,\nabla{df}=0
\end{equation}
 Hence $f$ is harmonic if and only if $\tau(f)=0$.
 
\noi  The expression of the tension field in terms of local coordinates can be found
 in \cite{ES}.  
 \noi The following is the solenoidal version of the theorem by Eells-Sampson
 \cite{ES}:
 \begin{theorem}[Solenoidal Eells-Sampson Theorem]\label{Solenoidal-Eells-Sampson}
 Let $(\mathcal{S}_1, g)$ and $(\mathcal{S}_2, h)$ be two compact
 McCord solenoids of dimensions $m$ and $n$, and
laminated Riemannian metrics $g$ and $h$, respectively. Suppose that the leaves
of $\mathcal{S}_2$ have strictly negative sectional curvature with respect to  $h$.
Let $f_0:(\mathcal{S}_1, g)\to(\mathcal{S}_2, h)$ be a smooth map (necessarily sends leaves
of $(\mathcal{S}_1, g)$ into leaves of $(\mathcal{S}_2, h)$) then $f_0$ is leafwise
homotopic to a unique harmonic map $f$.
 \end{theorem}
\noi{\bf Sketch of the proof.}
 The proof follows almost {\it verbatim} the steps of Eells-Sampson\cite{ES} using the heat flow.
 We also refer to Chapter 5 of \cite{LW} or Chapter 9 in \cite{Jo} for other detailed presentations.
Consider the evolution equation:
 
\begin{equation}\label{evoeque}
\begin{cases}
 & \partial_t f(x,t)=\tau(f(x,t)\\
 & f(x,0)=f_0(x)
\end{cases}
\end{equation}
\noi where $f(\cdot,t)\in{C^\infty(\mathcal{S}_1,\mathcal{S}_2)}$
(\ie it is a laminated smooth map).
 
\medskip
 
\noi The proof of theorem 11.1 consists of 6 steps:
\begin{enumerate}
\item The existence of $\epsilon>0$ such that $f(x, t)$ exists for
$0\leq{t}<\epsilon$.
\item The existence of $f(x, t)$ for all $t>0$.
\item The existence of $\underset{t\to\infty}\lim{f(x, t)}=f(x,\infty)=f(x)$.
\item The proof that $f(x)$ is harmonic.
\item The proof that $f(x)$ is homotopic to $f_0$.
\item The uniqueness of $f$.
\end{enumerate}
 The uniqueness follows by a theorem of Philip Hartman \cite{Har}. The reason
 that all the steps can be achieved is that the {\it a priori} estimates of Eells
 and Sampson hold in our case because we endowed the solenoids with a measure
 which is very tame, therefore the energy  $E(f(\cdot,t))$
 remains bounded by a constant independent of $t$. The proof uses the Bochner identity
 which also holds in our laminated case:
 let $f:\mathcal{S}_1\to\mathcal{S}_2$ be a smooth map. Let $df$ denote the derivative (pushforward) of $f$,
 $\nabla$ the laminated gradient \ie the covariant derivative on
 $T^*(\mathcal{S}_1)\otimes{f^{-1}T(\mathcal{S}_2)}$.
 
 \vskip.05cm
 
 \noi Let $\Delta^g$ the laminated  Laplace–Beltrami operator on $\mathcal{S}_1$,
 $\rm{Riem}^{\mathcal{S}_2}$ the laminated Riemann curvature tensor on
 $\mathcal{S}_2$ and $\rm{Ric}^{\mathcal{S}_1}$ the Ricci curvature tensor on
 $\mathcal{S}_1$.
 
 \noi Then if $\{f_t\}_t$ satisfies the evolution equation
 (heat flow equation) (\ref{evoeque}) then one has the {\it Bochner identity}:

\begin{equation}\label{Bochner}
{\Big (}{\frac{\partial }{\partial t}}-\Delta^{g}{\Big )}e(f_t)=-{\big |}\nabla (df_t){\big |}^{2}-{\big \langle }\rm{Ric}^{g},f_t^{\ast }h{\big \rangle }_{g}+{\rm{scal}} ^{g}{\big (}f_t^{\ast }\rm{Riem}^{h}{\big )}.
\end{equation}
 
\noi  In terms of the local coordinates (\ref{maplocally}),
 $f_t=(f_{t,1},\cdots,f_{t,n})$, and using Einstein's summation, the terms are defined as follows:
\begin{equation}\label{nabla}
{\big |}\nabla (df_t)(x,\mathfrak{g}){\big |}^{2}=
\underset{i,j,k}\sum{\Big [}\frac{\partial^2{f_{t,k}}}{\partial{x_i}\partial{x_j}}
(x,\mathfrak{g}){\Big ]}^2
\end{equation}
 
\begin{equation}\label{Ricci}
{\big \langle }\rm{Ric}^{g}(x,\mathfrak{g}),
(f_t^{\ast }h)(x,\mathfrak{g}){\big \rangle }_{g}={\rm{Ric}}^{\mathcal{S}_1}_{_{ij}}(x,\mathfrak{g})h_{\alpha\beta}(f_t(x,\mathfrak{g}))\frac{\partial{f}_{t,\alpha}(x,\mathfrak{g})}{\partial{x_i}}
\frac{\partial{f}_{t,\beta}(x,\mathfrak{g})}{\partial{x_j}}
\end{equation}
 
\begin{equation}\label{Riem}
{\rm{scal}} ^{g}{\big (}f_t^{\ast }\rm{Riem}^{h}{\big )}(x,\mathfrak{g})=
\end{equation}
\[
\rm{Riem}^{\mathcal{S}_2}_{\alpha\beta\gamma\delta}(f_t(x,\mathfrak{g}))
{\Big [}\frac{\partial{f}_{t,\alpha}(x,\mathfrak{g})}{\partial{x_i}}
\frac{\partial{f}_{t,\beta}(x,\mathfrak{g})}
{\partial{x_j}}\frac{\partial{f}_{t,\gamma}(x,\mathfrak{g})}{\partial{x_k}}
\frac{\partial{f}_{t,\delta}(x,\mathfrak{g})}{\partial{x_l}}{\Big ]}
g^{ik}(x,\mathfrak{g})g^{jl}(x,\mathfrak{g})
\]
We see in equation (\ref{Ricci}) that the $m\times{m}$ symmetric matrix
 ${\Big (}{\rm{Ric}}^{\mathcal{S}_1}_{_{ij}}{\Big )}$ can be compared with the
 positive-definite symmetric $m\times{m}$ matrix $g_{ij}(x,\mathfrak{g})$
and, since the leafwise sectional curvature of $\mathcal{S}_2$ is negative
comparing (\ref{energycoordinates}) with (\ref{Ricci}) we obtain the following:
\begin{proposition} There exists  a constant $C>0$ depending only on the leafwise Ricci curvature of $\mathcal{S}_1$ such that:
\begin{equation}\label{Bochnerinequality}
{\Big (}{\frac{\partial }{\partial t}}-\Delta^{g}{\Big )}e(f_t)\leq{C}e(f_t)
\end{equation}
\end{proposition}
\noi  Once the Bochner identity is valid, using proposition (\ref{Bochnerinequality})
 and the theorem of Moser-Harnack  Lemma (Lemma 5.3.4, page 115 in \cite{LW})
we can prove steps (1) to (4) exactly as in references \cite{ES}, \cite{LW}, \cite{Jo}, \cite{Jo1}.
The proofs of (6) and (7) follow Hartman's proof \cite{Har}. In fact, a direct
calculation shows that if $f_t$ is a solution of the heat equation (\ref{evoeque}) then the Dirichlet energy satisfies:
 
\begin{equation} \frac{\rm{d}}{\rm{dt}}\,E(f_t)=
-\int_{\mathcal{S}_1}||\partial_t{f_t}||^2\,\rm{d}\mu\leq0
\end{equation}\label{energydecreasing}
 
\begin{equation}\label{energyconvex} \frac{\rm{d^2}}{\rm{dt^2}}\,E(f_t)=
-\frac{\partial}{\partial{t}} \int_{\mathcal{S}_1}||\partial_t{f_t}||^2\,\rm{d}\mu=
-\int_{\mathcal{S}_1}\partial_t||\partial_t{f_t}||^2\,\rm{d}\mu=
||\nabla\partial_t{f_t}||^2\geq0
\end{equation}
 
\noi To be sure that the leafwise evolution of the heat flow is as in the case of \cite{ES} we proceed as follows. The solenoids
are of the form $\,\cS_1=\tilde{M}_1\times{G}_1/G_1$ and
$\,\cS_2=\tilde{M}_2\times{G}_1/G_1$
where $\tilde{M}_1$ and $\tilde{M}_2$ are the universal coverings of
the compact manifolds $M_2$ and $M_2$ and $G_1$, $G_2$ are the respective profinite completions of their fundamental groups.
 
\noi By theorem (\ref{solenoid-lifting}) applied to the
solenoidal coverings $\hat\Pi_1:\tilde{M}_1\times{G}_1\to\cS_1$
and $\hat\Pi_2:\tilde{M}_2\times{G}_2\to\cS_2$, the map $f$ lifts to a map
$\tilde{f}:\tilde{M}_1\times{G}_1\to\tilde{M}_2\times{G}_2$.
 
\noi The map $\tilde{f}$
is of the form: $\tilde{f}(x,\mathfrak{g_1})=(\tilde{f}_{\mathfrak{g_1}}(x),h(\mathfrak{g_1}))$,
where $h:G_1\to{G_2}$ is continuous, $\mathfrak{g_1}\in{G_1}$,
$x\in{\tilde{M}_1}\times\{\mathfrak{g_1}\}$ and
 \[
 \tilde{f}_{\mathfrak{g_1}}:\tilde{M}_1
 \times\{\mathfrak{g_1}\}\to\tilde{M}_2
 \times\{\mathfrak{h(g_1)}\}
\]
is a smooth map from the leaf $\tilde{M}_1
 \times\{\mathfrak{g_1}\}$ into the leaf $\tilde{M}_2\times\{\mathfrak{h(g_1)}\}$.
 
\noi The Cantor group $G_1$ acts freely, properly discontinuously, and cocompactly on
$\tilde{M}_1\times{G}_1$, therefore
the solenoidal covering projection $\hat\Pi_1:\tilde{M}_1\times{G}_1\to\cS_1$
has a fundamental domain $\mathbb{D}$ of the form
$\mathbb{D}=\mathfrak{D}\times{G_1}\subset\tilde{M}_1\times{G}_1$ with
$\mathfrak{D}\subset\tilde{M}_1$ compact. If we restrict $\tilde{f}$ to
$\mathbb{D}$ we see that all the previous estimates, like Bochner's identity and Moser-Harnack inequality, hold for each
$\tilde{f}_{\mathfrak{g_1}}$, uniformly
with respect to $\mathfrak{g}_1\in{G_1}$.
 
\noi Hence, the heat equation evolves
in each leaf as in the case of compact manifolds. The main fact is:
{\it if the sectional curvature of the leaves of $\cS_2$ is strictly negative, the
Dirichlet integral \ref{energy} is a strictly decreasing,
and convex function of $f$.}
 
\noi Summarizing, we have:
 
\begin{proposition} Under our hypothesis, there is a unique harmonic map in each homotopy class of maps in $C^\infty(\cS_1,\cS_2)$. Furthermore, the harmonic map is the unique map in its homotopy class which minimizes the energy. The harmonic map is stable and depends continuously on the initial map $f_0$.
\end{proposition}
\section{1-dimensional solenoidal manifolds}\label{1dsm}
\subsection{1-dimensional solenoidal manifolds as mapping tori.} Part of this subsection
is based on the paper by Dennis Sullivan \cite{Su1} and its companion \cite{V}.
 
\noi A (smooth) compact 1-dimensional, oriented, solenoidal manifold $\cS$ can be given a foliated Riemannian metric and thus we can define a unit vector field along the leaves and, as a consequence, one can define a nonsingular flow
 $\varphi_t:\cS\to{\cS}$.
 
\begin{proposition}[Sullivan \cite{Su1}] Let $\cS$ be a compact, oriented,
1-dimensional solenoidal manifold. Then, $\cS$ admits a global transversal Cantor set
$K$. More precisely, the flow $\phi_t$ has a global Poincaré cross-section. Therefore,
$\cS$ is the mapping torus or suspension of a homeomorphism of $K$.
\end{proposition}\label{1-dimensional-mapping-torus}
\begin{proof} One can choose a finite set of transversals intersecting every leaf and such that starting at a point on one transversal and going forward (with respect to the orientation of the leaves) one first meets a different transversal.
 
\noi This presents the solenoid as a {\it mapping torus} or {\it suspension} of a homeomorphism $f:K\to{K}$ on the Cantor set $K$, where $f$ is the
first-return map of the flow to the cross-section.
\end{proof}
 
\begin{corollary} 1-dimensional compact oriented solenoids are classified by the conjugacy classes of the group $\mathcal{H}(K)$ of homeomorphisms of the Cantor set.
\end{corollary}
 
\subsection{1-dimensional, compact, solenoidal manifolds are null-cobordant}
 
\noi By {\it proposition} \ref{1-dimensional-mapping-torus} any oriented 1-dimensional solenoidal manifold is the suspension of a homeomorphism $f$ of the Cantor set and this implies the following:
 
\begin{proposition}[Sullivan \cite{Su1}]\label{nullcobordant} Any 1-dimensional, compact, oriented, solenoidal manifold $\cS$ is null-cobordant: there exists a compact 2-dimensional solenoidal manifold with boundary and this boundary is the given solenoidal 1-dimensional manifold.
 \end{proposition}
 \begin{proof} The proof is based on an idea by Thurston and it follows from the fact that any homeomorphism of the Cantor set is a finite product of commutators, \ie ${\mathcal H}(K)$  ({\it definition} \ref{homcantor}) is a perfect group.
 The proof of the fact that the group of homeomorphisms of the Cantor set is perfect is proven in all detail in the paper \cite{An} by R.D. Anderson.
 
\noi   Let $\cS$ be the suspension of $f\in\mathcal{H}(K)$ and
 $f=[h_1,k_1]\cdots[h_g,k_g]$.
 Then, there exists a representation $\rho:\pi_1(\Sigma)\to{\mathcal{H}(K)}$,
 where $\Sigma$ is a smooth compact surface of genus $g$
 with connected boundary a circle, such that the restriction of $\rho$ to the element represented by the boundary circle is $f$.  
  \noi The lamination we want is the suspension of the representation
 $\rho$.
 More precisely, and generally, let  $N$ be a compact $n$-dimensional manifold and
$r:\pi_1(N)\to{\mathcal H}(K)$ a homomorphism from the fundamental group of $N$ to ${\mathcal H}(K)$, \ie $r$ is a representation of
$\pi_1(N)$ into ${\mathcal H}(K)$.
 
\vskip.1cm
\noi The $n$-dimensional solenoidal lamination ${\mathcal{L}}_\rho$ associated to $\rho$ is called the
{\it suspension} of $r$. It is obtained by taking the quotient of
$\tilde{N}\times{K}$ under the action of $\pi_1(N)$
given by $\gamma(x,k)=(\gamma(x),r(\gamma)(k))$ where $\tilde{N}$ is the universal covering of $N$ and the action of
of $\pi_1(N)$ on $\tilde{N}$ is by deck transformations.
 
\noi The action of $\pi_1(N)$ is free and totally discontinuous so the quotient is a nice space. The canonical
map $\Pi:\tilde{N}\times{K}\to{\mathcal{L}}_\rho$ is a local homeomorphism and
it has the path-lifting property.
 
\noi One has a natural locally trivial fibration $p:{\mathcal{L}}_\rho\to{N}$ with fiber $K$. Applying this construction to our $\rho$ we obtain
the proof of {\it proposition} \ref{nullcobordant} since the restriction of the representation to the boundary circle is the suspension of $f$. \end{proof}
 
\subsection{1-dimensional compact  topologically homogeneous solenoids}\label{hom1d}
 
 \begin{theorem} \label{tophom1d} Every compact, connected, 1-dimensional solenoidal manifold
 which is topologically homogeneous is homeomorphic to a compact, connected, abelian group $\Gamma$. There exists a dense (with respect to the standard topology) subgroup $A\subset\Q$ of the additive
 rationals $(\Q,+)$ such the Pontryagin dual of $A_\delta$ is isomorphic to
 $\Gamma$, where $A_\delta$ denotes $A$ equipped with the discrete topology.

  \end{theorem}

\begin{proof} If $\cS$ is a compact, topologically homogeneous, 1-dimensional solenoid
then by {\it theorem} \ref{ClarkHurder} it is a MacCord solenoid and therefore it is
obtained as the inverse limit of an infinite tower of regular finite coverings of the circle $\sS^1=\left\{z\in\C\,:\,|z|=1\right\}$. Every finite covering map of the circle is equivalent to a map of the form
$z\mapsto{z^n}$, $n\in\N$, therefore it is a group homomorphism from the circle to the circle. It follows that $\cS$ is the inverse limit of a sequence
(see {\it remark} \ref{chain-cofinal}):
\begin{equation}
 \cdots \xrightarrow{z\to{z}^{n_4}}{\mathbb S}^1\xrightarrow{z\to{z}^{n_3}}
 {\mathbb S}^1\xrightarrow{z\to{z}^{n_2}}\sS^1
 \xrightarrow{z\to{z}^{n_1}}\sS^1, \quad n_i\geq2.
\end{equation}

\noi Therefore the inverse limit is of a sequence of homomorphisms, hence this limit is a compact abelian group. Since $\cS$ is of topological dimension one, it follows by Pontryagin duality that $\cS$ is the Pontryagin
dual of a countable discrete abelian group $\Gamma$ of {\it rank one}, and therefore
$\Gamma$ is a subgroup of the rationals $\Q$.
 
\noi The nontrivial subgroups of $\Q$ are
either dense or infinite cyclic, the latter must be excluded since its Pontryagin dual is the circle and not a solenoid. \end{proof}
 
\noi By {\it theorem} \ref{tophom1d}, compact, connected, orientable, 1-dimensional topologically homogeneous solenoidal manifolds are in one-to-one correspondence with the isomorphism classes of subgroups of the additive rationals $\Q$ via the Pontryagin duality. These solenoids have a distinguished leaf, the component of the identity called the base leaf, \cite{Od}. A theorem by Reinhold Baer, \cite{Baer}, describes the subgroups of $\Q$ up to isomorphism by equivalence classes of sequences called types (see also \cite{Tho}). There are uncountably many such types. Hence there are also uncountably many isomorphism classes of
additive subgroups of $\Q$.  \noi On the other hand it is shown in
 \cite{Sche} that
 two homeomorphic locally compact, connected, abelian groups are topologically isomorphic, and therefore their Pontryagin duals are isomorphic subgroups of $\Q$.
  Hence non-isomorphic subgroups of the rationals correspond to non- homeomorphic solenoids.
 
  Therefore we have the following:
  \begin{corollary} There exist uncountably many homeomorphism classes of compact, connected, orientable, topologically homogeneous 1-dimensional solenoidal manifolds.
  They are all suspensions of a minimal translation on a Cantor abelian group.
  \end{corollary}

 \begin{remark}\label{integeradeles} Let $\mathbb{A} _{\mathbb{Z}}$ be the ring of integral adèles \cite{RV} then:
 $\mathbb{R}\times{\hat{\mathbb {Z}}}=\mathbb{A}_{\mathbb{Z}}=
 \mathbb{R}\times\prod_{p}\mathbb{Z}_{p}$.
 The map $\Z\overset{i}\hookrightarrow\mathbb{R}\times{\hat{\mathbb {Z}}}$,
 $\,n\mapsto(n,\mathbf{n})$
 (where $\mathbf{n}$ is the image of $n$ by the natural inclusion of $\Z$ into $\hat\Z$)
injects $\Z$ into a discrete co-compact subgroup $\Gamma$. The Pontryagin dual
 $\Q^*$ of the rationals is isomorphic, as a compact, abelian, topological group, to
$\widehat{\sS}^1:=(\R\times\hat\Z)/\Gamma$.
 
\noi The subgroup
$\left\{(0,\mathbf{n}): n\in\Z\right\}$, isomorphic to $\Z$, is dense in
$\left\{0\right\}\times\hat\Z$,  where $\hz$ is the profinite completion of the integers $\mathbb Z$.
 This implies that the canonical flow whose orbits are the translates of the component of the identity (isomorphic to $\R$) is minimal.
\end {remark}
\begin{definition}\label{adelesclassgroup} The group $\widehat\sS^1$ is called the
{\it àdele class group or universal 1-dimensional arithmetic solenoid}. It is the algebraic universal covering
of the circle $\sS^1$ ({\it definition} \ref{univ-alg-cov}).
\end{definition}
\noi The inclusion of a subgroup $\Gamma\subset\Q$ of $\Q$, induces by Pontryagin duality
an epimorphism of their duals $\widehat{\sS}^1=\Q^*\to\Gamma^*$.
Thus the universal solenoid $\hat{\sS}^1$ maps epimorphically onto any 1-dimensional compact, connected, abelian group.

  \section{Solenoidal Riemann surfaces}\label{srs}
\subsection{Surface laminations and uniformization via de Ricci flow}
\begin{definition}\label{sl} A {\em compact lamination by surfaces} $\cL$ is a compact metrizable space $\cL$ endowed with an atlas
$\cA=(U_\alpha, \varphi_\alpha)$
such that:
\begin{itemize}
 \item[(1)] Each $\varphi_\alpha$ is a homeomorphism from $U_\alpha$ to a product $D\times T_\alpha$, where $D$
 is the unit disk in the plane $\R^2$ and
 $T_\alpha$ is a locally compact space (not necessarily a Cantor set).
 \item[(2)] Whenever $U_\alpha\cap U_\beta \neq \emptyset$, the change of coordinates
 $\varphi_\beta\circ\varphi_\alpha^{-1}$ is of the form
 $$(z,\zeta)\mapsto (\lambda_{\alpha\beta}(z,\zeta), \tau(\zeta)),$$ where $\lambda_{\alpha\beta}$ is smooth in the $z$ variable. If
 the $\lambda_{\alpha\beta}$ preserve a fixed orientation of the 2-disk we say that the lamination is oriented. The solenoidal case is when $T_\alpha$ is a Cantor set
 for all $\alpha$.
\end{itemize}
 
\end{definition}
 
\noi In the following definition we identify $\R^2$ with $\C$ via $(x,y)\mapsto{z=x+\mathbf{i}y}$
and thus $D$  in {\it definition} \ref{definition:lamination} can be considered as the open unit disk $D=\left\{z\in\C:\,|z|<1\right\}$ in $\C$ and one can use complex notation.
\begin{definition}\label{definition:holomorphic_lamination} A 2-dimensional lamination $\cL$ is said to be a \emph{lamination by Riemann surfaces}
or simply a \emph{Riemann surface lamination}  if the atlas $\cA=(U_\alpha, \varphi_\alpha)_{_{\alpha \in \mathfrak{I}}}$
satisfies that the change of coordinates
 $\varphi_\beta\circ\varphi_\alpha^{-1}$ is of the form
 $(x,t)\mapsto (\lambda_{\alpha\beta}(x,t), \tau_{\alpha\beta}(t)),$ where $\lambda_{\alpha\beta}$ is a holomorphic diffeomorphism in the $x$ variable for each $t$. In this case, $\lambda_{\alpha\beta}$ preserves the preferred orientation of the disk induced by the orientation of $\C$, and the lamination is necessarily orientable.
 An atlas satisfying this condition is called a \emph{complex atlas}.
 If  $\cA=(U_\alpha, \varphi_\alpha)_{_{\alpha \in \mathfrak{I}}}$ is a complex atlas as in {\it definition} \ref{sl} the leaves are immersed copies of \emph{Riemann surfaces}.
\end{definition}
 
Several vector bundles over $\cL$ can be defined in a natural way using the fact that the $\lambda_{\alpha\beta}$ are smooth
in the variable $x$.
 
\noi These include the tangent bundle $T\cL$ to the lamination, tensor bundles, frame bundles, etc. Their fibers
vary continuously in the smooth topology along the direction transverse to the laminated structure, that is, the direction given by
the $T_\alpha$ and are smooth along the leaves.
 
\noi If $\cL$ is a Riemann surface lamination as in {\it definition} \ref{definition:holomorphic_lamination}
one can define laminated complex line bundles with fiber $\C$
like, for instance, the \emph{complex tangent line bundle} $T_\C\cL$.
The complex tangent line bundle is the line bundle obtained by providing the real tangent bundle $T\cL$ with the natural almost complex structure obtained by the complex atlas using multiplication by $\mathbf{i}$ in each real tangent plane $T_x L_x$ of the leaf $L_x$ through $x$.
 
\bigskip
\noi One can also define the \emph{laminated canonical line bundle} $K\cL$ and products of tensor bundles such as
$(T_\C\cL)^{\otimes{m}}\otimes{(K\cL)^{\otimes{n}}}$.  One has
 $(T_\C\cL)\otimes{K\cL}\cong\cL\times\C$ is the trivial line bundle.
 
 \noi Thus one can define the Picard group of Riemann surface lamination (\ie the group of line bundles with multiplication given by the tensor product). In particular the sections of
 $K\cL\otimes{K\cL}$
are laminated quadratic differentials which are important to study the Teichmüller laminated theory (in general the spaces are of infinite dimension).
 
\begin{definition}A {\em laminated (or foliated) Riemannian metric} $g$ is a continuous tensor such that restricted to each leaf is a Riemannian metric and it is smooth along the leaves. In terms of a foliated chart
$\varphi_\alpha:U_\alpha\to{D}\times{T_\alpha}$ the laminated metric is determined by a continuous function
$g_\alpha:{D}\times{T_\alpha}\to{\rm{Symm}_+(2)}$, where $\rm{Symm}_+(2)$ is the space of real positive definite
$2\times2$ symmetric matrices:
\begin{equation}\label{eq:rm}
 g_\alpha(z,t)=
 \begin{pmatrix}
E(z,t) & F(z,t)\\
F(z,t) & G(z,t)
\end{pmatrix},
\end{equation}
\noi where $E(z,t)>0$, $E(z,t)G(z,t)-F^2(z,t)>0$.  
 
The functions $E$, $F$ and $G$ are differentiable with respect to $z$.
Thus in this chart, the metric is written in the traditional notation
\[
ds^2=E(z,t)\,dx^2+2F(z,t)dx\,dy+F(z,t)\,dy^2, \quad z=x+\mathbf{i}y.
\]
 
In terms of complex notation
\begin{equation}\label{beltrami}
ds^2=\gamma(z,t)|dz+\mu(z,t)d\overline{z}|^2,
 \end{equation}
 
\noi where $z=x+\mathbf{i}y$, $\gamma(z,t)>0$, and $\mu(z,t)<1$.
\end{definition}
 
\begin{definition}\label{beltramicoeff} The function $\mu$ in formula (\ref{beltrami})
is called the {\it Beltrami coefficient} of the metric $g$.
In a coordinate-invariant fashion, we can regard the Beltrami coefficient
as a $(-1,1)$-form \ie
a section of $K^{-1}\cL\otimes{\bar{K}\cL}$.
 
\noi One must impose some regularity
in order to have a good description of quasi-conformal maps
  between Riemann surface laminations and their theory of Teichmüller spaces.
 
\end{definition}
\begin{remark}\label{metric=complex} If $g$ is a laminated Riemannian metric on
a surface lamination $\cL$ then $g$ induces an almost complex structure
$J:T\cL\to{T\cL}$ as follows: if $x$ is in the leaf $L$ of $\cL$ and $v\in{T_x{L}}$,
then $J(v)=w$ where $w$ is orthogonal to $v$
(with respect to $g$), $|v|_g=|w|_g$ and $(v,w)$ is an oriented frame of
$T_x{L}$ (with respect to the orientation induced by the almost complex structure).
In dimension 2 every almost complex is integrable, therefore an almost complex structure on $\cL$ induces on $\cL$ the structure of a Riemann surface lamination
and therefore the structure of a Riemann surface lamination.
 
\noi This is the laminated version of the Gauss-Korn-Lichtenstein on the existence of local isothermal coordinates and it follows by the Ahlfors-Bers theory \cite{AB}.
 
\end{remark}
 
\begin{remark}[Reeb stability]\label{nosphleaf} If a connected surface lamination $\cL$ contains a leaf
homeomorphic to the sphere $\sS^2$, then by Reeb's stability theorem
(which is valid for laminations by Theorem 5.3 in \cite{Ca}), the leaf space of $\cL$ is a compact Hausdorff space and $\cL$ homeomorphic to a fiber bundle over the compact leaf space
with fiber $\sS^2$.
{\it To avoid this trivial case we will assume heretofore that no leaf of the surface
laminations considered in this paper are spherical}.
 
\end{remark}
 
\begin{definition}\label{quasi-isom} Two laminated Riemannian metrics $g_1$ and $g_2$ with induced norms $|\xi|_{g_1}$, $|\xi|_{g_2}$, respectively, on a smooth surface
lamination $\cL$, are said to be {\it quasi-isometric} if there exists
 a constant $a\geqq 1$ such that:
\[
a^{-1}|\xi|_{g_1}\leqq|\xi|_{g_2}\leqq a|\xi|_{g_1},\,\, \forall \xi\in T\cL.
\]
\end{definition}
As explained in {\it remark} \ref{metric=complex} laminated Riemannian metrics on $\cL$ induce complex structures on the leaves
of $\cL$ so that they become Riemann surfaces and $\cL$ is a Riemann surface lamination. By the Koebe-Poincaré uniformization theorem, the universal cover $\tilde{L}$ of a leaf $L$, with the lifted complex structure, is conformally equivalent to either the Riemann sphere $\bar\C$, the complex plane $\C$ or the Poincaré disk $\Delta$;
we say that $L$ is of {\it elliptic}, {\it parabolic} or {\it hyperbolic}
type, respectively. We exclude the case of spherical leaves
({\it remark} \ref{nosphleaf}).
Let $S$ be a simply connected non-compact surface endowed with a complete Riemannian metric $g$.
Let $x\in{S}$ and $r>0$; define $A_g(r,x)$ to be the area of the geodesic disk of radius $r$ centered at $x$. One says that $g$ has  {\it polynomial area growth} if there exists
a constant $c>0$ and an integer $n$, such that $A_g(r,x)\leq{cr^n}$.
We say that the area {\it grows exponentially} if there exists
positive constants $c, b$ such that $A_g(r,x)\geq{c} e^{br}$. These two properties are independent of the point $x$.
By Theorem 3.3 of M. Kanai in \cite{Kan},
 these growth properties depend only on the quasi-isometry class of the Riemannian metric provided that the injectivity radius is positive and the Ricci curvature is bounded below. Thus, under these conditions, the conformal type of a noncompact simply connected surface
is determined by the growth rate of the area.
 If $\cL$ is a compact surface lamination with a laminated Riemannian metric $g$ then, by compactness, there exists $\delta>0$ such that the injectivity radius of every leaf is greater than $\delta$.
 In addition, the Ricci curvature of the leaves is uniformly bounded below.
 
 Leaves of parabolic type have polynomial growth (\ie their universal covering is conformally equivalent to $\C$) and leaves of a hyperbolic type have exponential area growth (\ie their universal covering is conformally equivalent to the unit disk
 $\Delta$). Therefore we have the unambiguous notion of the {\it conformal type} of a leaf
 $L$ of a smooth {\it compact} lamination by surfaces:
 \begin{proposition} Let $\cL$ be a $\text{\bf compact}$, smooth, surface lamination. Let $L$ be a leaf of $\cL$, then the conformal type of $L$, with respect to the complex structure induced by a laminated Riemannian metric $g$, is independent of $g$. This allows us to
speak about {\em hyperbolic} or {\em parabolic} leaves,
independently of the laminated metric (recall that we excluded spherical leaves).
\end{proposition}


\begin{definition}[Hyperbolic lamination]\label{hyperboliclamination}
If all leaves are of hyperbolic type with respect to a laminated metric $g$ we refer to
$\cL$ as a {\it hyperbolic lamination} (the property is independent of $g$).
\end{definition}
\noi We recall that by {\it remark} \ref{metric=complex} on an \emph{oriented} lamination a laminated Riemannian metric $g$ determines a conformal structure on every leaf, that is,
it turns every leaf into a Riemann surface.
 
\noi A complex leaf $L$ is hyperbolic if and only if it can be uniformized by the unit disk \ie there is a conformal covering map $\varphi:\Delta\to{L}$. It is easy to construct surface laminations with leaves of mixed types.
 
\noi When all leaves are hyperbolic, the uniformization maps of individual leaves vary continuously from leaf to leaf.
More precisely, the following {\em Uniformization Theorem}
holds (see \cite{Ca}, \cite{Ve}):
 
\begin{theorem}\label{uniformization}
 Let $\cL$ be a compact lamination by hyperbolic surfaces endowed with a laminated Riemannian metric $g$. Then there is
 a laminated Riemannian metric $g'$ which is conformally equivalent to $g$ and for which every leaf has constant curvature -1.
 The metric $g'$ has a continuous variation in the smooth topology in the direction transverse to the leaves of $\cL$.
\end{theorem}
\begin{definition}[Hyperbolic laminations]  Laminated Riemannian metrics $g$ of  $\cL$ with
all leaves of constant Gaussian curvature -1 are called
{\it laminated hyperbolic metrics}.
\end{definition}

\subsection{Laminated Ricci flow}\label{Ricci_Flow}
 
In this subsection, we will state a stronger version of
{\it theorem} \ref{uniformization} established in \cite{MV}, which uses the geometric flow determined
by the Ricci flow in the spirit of Hamilton \cite{Ha}.  Given a compact lamination
by surfaces $\cL$ with a laminated Riemannian
metric $g_0$ with all leaves of negative curvature, we find a laminated metric $g_\infty$ of class $C^{(\infty,0)}$
({\it definition} \ref{smoothfunctions}) which renders each leaf of constant negative curvature conformally equivalent to $g_0$.
 
\noi Let $\cL$ be a compact lamination and $g_0$ a $C^{(\infty,0)}$
laminated metric on it.
In two dimensions, the Ricci curvature for a given metric $g$ is equal to
${\frac12}Rg$, where $R$ is the scalar curvature or $R=2K$ where $K$ is the Gauss curvature.
We can consider the {\it ``normalized laminated Ricci flow''} as the evolution of the metric under
the equation
\begin{equation}\label{riccieq}
\left\{\begin{aligned} \frac{\partial g(t)}{\partial t} &=(c-R(t))g(t),\\
                       g(0)&=g_0
       \end{aligned}\right.
\end{equation}

\noi Here $R(t)$ is the scalar curvature of the leaves for the metric $g(t)$ and $c$ is a normalizing constant, {\it independent of $t$,
which is chosen conveniently}.
\noi Let us denote by $R_0$ the curvature (in the leaf
direction) of the metric $g_0$. Since $M$ is compact the leaves are complete and
$R_0$, being a
continuous function on $M$, is bounded.
From this, it is possible to conclude that there exists $\epsilon >0$ such that for each time $t$ in an interval
$[0,\epsilon)$ there is a solution $g(t)$ to the  Ricci flow equation; for
$g(t):=g_t$ to be a solution to (\ref{riccieq}) on $\cL$ it
has to vary continuously in the transverse direction; this is a consequence
of the continuous dependence of the solution to (\ref{riccieq}) with respect to the initial
condition.
 
\noi The curvature of a family of metrics $g(t)$ satisfying (\ref{riccieq}) evolves under the
 equation:
   \begin{equation}
\frac{\partial}{\partial t}R(t)=\Delta_t R(t)+(R(t)-c)R(t).
   \end{equation}
Here $\Delta_t$ denotes the Laplacian in the leaf direction
(with respect to $g(t)$), or equivalently, we
consider the above equation on each leaf.
 
 In the 2-dimensional case the metrics given by the equation
(\ref{riccieq}) leave invariant the conformal class of the initial metric $g_0$, hence we
can write the evolution as an evolution of a single function $u$.
 
\noi More precisely, by writing
$g=e^ug_0$ for a metric in the conformal class of $g_0$,
 we have that under the normalized Ricci flow $u$ evolves according to
  \begin{equation}
   \frac{\partial u}{\partial t}=c-R=\Delta_t u-e^{-u}R_0+c=e^{-u}(\Delta_0 u-R_0)+c,
  \end{equation}
 \noi where $\Delta_0$ denotes the Laplacian operator associated with $g_0$.
 Since $\Delta_t u=e^{-u}\Delta_0$ the equation follows from
 well-known fact that $\Delta_{e^ug_0}=e^{-u}\Delta_0$.
 
 \noi The following theorem in \cite{MV} shows that if we start with an initial laminated
 metric $g_0$ such that every leaf has negative curvature at every point then the Ricci flow
 exists for all positive time and defines a one-parameter family of metrics
 $\left\{g_t\right\}_{_{t\geq0}}$ such that $\underset{t\to\infty}\lim\,g_t=g_\infty$
 exists and defines a smooth metric with all leaves of constant negative curvature
 (rescaling we can assume that this constant curvature is -1).
 
 \noi The metric $g_\infty$
 is conformally equivalent to $g_0$. By the following theorem due to
Étienne Ghys (a proof can be found in \cite{AJ} and \cite{AS} Théorème 6.5) we can choose as the initial metric a globally hyperbolic metric:
 
\begin{theorem}[É. Ghys]\label{lemma:etienne}
 Let $\cL$ be a compact lamination by surfaces of hyperbolic type. Then, in each conformal class of $\cL$,
 there exists a  Riemannian laminated metric in such a way that the leaves of $\cL$ have negative curvature \end{theorem}
 
 \begin{remark} Let $g_0$ be an initial metric as in {\it theorem} \ref{lemma:etienne}.
 Let $R_0:M\to\R_{<0}$  be the Ricci curvature of $g_0$. Then, by compactness
there exists two constants $R_{\text{min}}$ and $R_{\text{max}}$ such that
$R_{\text{min}}\leq{R_0(x)}\leq{R_{\text{max}}}<0,\, \forall\,x\in{M }$. We may assume
that $R_0$ is not a (negative) constant since in such a case there is nothing to prove, hence we may suppose $R_{\text{min}}<R_{\text{max}}$. The normalizing constant $c$ is chosen such that $c\in(R_{\text{min}}, R_{\text{max}})$.
\end{remark}
 
\noi In \cite{MV} the following is proven:
\begin{theorem}[R. Muñiz, A. Verjovsky]\label{continuity}
 Let $g_0$ be a laminated Riemannian metric of non-constant negative curvature on a compact, connected, surface lamination $\cL$.
 Let $g_t=g(t)=e^{u_t}g_0$ be the leafwise solution to the normalized Ricci flow equation (\ref{riccieq}) on $\cL$, with a constant
 $c\in(R_{\text{min}},R_{\text{max}})$
 and initial condition $g_0$.
 
 \noi The function $u_t=u(\,\cdot\,,\,t)$ belongs to $C^{\infty,0}(\cL)$.
 Furthermore,  $\underset{t\to\infty}\lim\,g_t=g_\infty$
 exists and defines a laminated metric of class $C^{(\infty,0)}$ with all leaves of constant negative curvature
 (rescaling we can assume that this constant curvature is -1). The metric $g_\infty$
 is conformally equivalent to $g_0$.
\end{theorem}
\noi {\it Theorem} \ref{uniformization}
is a corollary of {\it theorem} \ref{continuity}:
\begin{theorem}[Uniformization via the Ricci flow]\label{ricciunif} Let
$g_0$ be any smooth laminated Riemannian metric on a compact
hyperbolic surface lamination $\cL$. Then the normalized Ricci flow exists for
all $t>0$.
 
\noi The corresponding metric $g_t$ satisfies
$\underset{t\to\infty}\lim=g_\infty$ exists and $g_\infty$ is a continuous,
laminated, Riemannian metric, which renders all leaves with constant negative curvature (which
can be normalized to be equal to -1).
Furthermore $g_\infty$ is of class $C^{\infty,0}$ and it is conformally equivalent to $g_0$.
\end{theorem}

\noi All Riemannian metrics are quasi-isometric. In particular, all leaves are quasi-isometric to the hyperbolic plane with the same comparison constants.
 
\noi The fact that we can start with any metric quasi-isometric to the constant hyperbolic metric, follows by the methods of \cite{MV} using the following theorem by S. Angenent (\cite{IMS} Proposition 4.5  or the appendix by S. Angenent in \cite{Wu}):
 
\begin{proposition}[S. Angenent] Let $g_0$ be a complete metric on the unit disc
$D$. Suppose $g_0$ is quasi-isometric to the hyperbolic metric $g_{hyp}$ \ie
there exists constants $0 < c_1 < c_2$ such that
 $c_1 |v|_{g_{hyp}} \leq |v|_{g_0} \leq c_2 |v|_{g_{hyp}}$ for any tangent vector on $D$. Then, $\forall{t>0}$
 $c_1 (1+2t)|v|_{g_{hyp}} \leq |v|_{g_t} \leq c_2 (1+2t)|v|_{g_{hyp}}$
 and in particular, for $c<0$ the solution $g_t$ of equation (\ref{riccieq}) exists for all time and converges, as $t\to\infty$, to a metric of constant negative curvature.
\end{proposition}
\medskip
\begin{corollary}\label{hyp-Lambda} The hyperbolic metric $g_\infty$ depends continuously
on the initial metric $g_0$. Therefore, if $\mathfrak{M}(\cL)$ denotes
the Fréchet manifold of laminated metrics of class $\C^{\infty,0}$
and $\text{Hyp}(\cL)$ denotes the closed
subset of hyperbolic metrics, then $\mathfrak{M}(\cL)$ retracts strongly
onto $\text{Hyp}(\cL)$ via the Ricci flow. Hence there is a retraction
map $\Pi:\mathfrak{M}(\cL)\to\text{Hyp}(\cL)$ such that the fiber
$\Pi^{-1}(\left\{g\right\})$ consists of laminated Riemannian metrics
conformally equivalent to the hyperbolic metric $g$.
\end{corollary}
 
\begin{corollary}\label{isothermal}
 Let $\cL$ be a compact, connected, hyperbolic surface lamination
 and $g$ any $C^{\infty,0}$ laminated Riemannian metric.
Then using theorem \ref{continuity} we can endow $\cL$ with a $C^{\infty,0}$ laminated complex structure which turns $\cL$ into a Riemann surface lamination.
 
\noi All metrics $g'$ conformally equivalent  to $g$ determine biholomorphic laminations.
Reciprocally, if $\cL$ is a hyperbolic Riemann surface lamination it
determines a conformal class of laminated metrics.
\end{corollary}
\noi As in the case of compact Riemann surfaces it is possible to define the corresponding
Teichmüller space:
\begin{definition}\label{lam-teich} The laminated Teichmüller space $T(\cL)$ of the hyperbolic compact
lamination $\cL$ is defined as the set of equivalence classes
$\text{Hyp}(\cL)/\sim\,\,$, where $\text{Hyp}(\cL)$
is as in {\it corollary} \ref{hyp-Lambda}
and $\sim$ is the equivalence relation $g_1\sim{g_2}$ if there exists
a laminated smooth diffeomorphism $f:M\to{M}$ which is isotopic to the identity
(therefore it preserves each leaf) and $f$ is an isometry on each leaf \ie
$g_2$ is the {\it pushforward} by $f$ of $g_1$.
\end{definition}
 
\subsection{Universal hyperbolic solenoid}\label{UHL}
 In \cite{Su2} Dennis Sullivan starts the study of the Teichmüller space of
 hyperbolic Riemann surface laminations $\cL$. For more information on this
 topic see also the papers \cite{BN}, \cite{BN1}, \cite{BNS}, \cite{BV}
 \cite{MSa}, \cite{Od}, \cite{PS},\cite{Sa},\cite{Sa1},\cite{Sa3}. Because of the correspondence between conformal classes of metrics, complex structures, and hyperbolic structures of compact
 hyperbolic laminations, one can define the Teichmüller space of $\cL$
 as the space of all continuous conformal structures along the leaves
 with transversally continuous Beltrami coefficients (definition \ref{beltramicoeff}), modulo the group of quasi-conformal isotopies leaving each leaf invariant. There is a “laminated” Bers embedding theorem.
 
 \noi The natural topology on the Teichmüller space is Hausdorff, and it is biholomorphic to a nonempty open subset of the Banach space of holomorphic quadratic differentials along the leaves, which are continuous on $M$. In general, the Teichmüller space of a compact lamination is infinite-dimensional.
In fact, in \cite{De} B. Deroin proves that the Teichmüller space of a compact surface lamination with
one simply connected leaf is infinite-dimensional.
 
\noi Solenoidal surfaces are important due to their connections with complex analysis and complex dynamics \cite{DNS, Gh1, LM, LMM, Su2}.
 
\noi Dennis Sullivan has constructed an important lamination whose Teichmüller space is remarkable: {\it The universal commensurability Teichmüller space}.
 \noindent The construction of {\it Sullivan's universal hyperbolic solenoid} is based on profinite constructions. More precisely: if $\Sigma$ is a compact orientable surface of genus $g\geq2$ and if we consider
 the inverse limit corresponding to the tower of {\it all} the finite-sheeted coverings of
 $\Sigma$ we obtain a 2-dimensional solenoidal manifold or surface lamination
 $\cL_\mathcal{h}$.

 \begin{remark}
 $\cL_\mathcal{h}$ is the McCord solenoidal manifold which is the algebraic universal covering of $\Sigma$
 discussed in section \ref{profinite-McCord}. However, following Dennis Sullivan, we call this solenoid
 {\it Universal Hyperbolic Solenoid} to emphasize the fact that $\cL_\mathcal{h}$,  together with its laminated complex or hyperbolic structures, encapsulates
 many of the features of Riemann surfaces (or, equivalently, hyperbolic surfaces).
 \end{remark}
 
 We can consider complex structures in this lamination so that each leaf has a complex structure and the complex structures vary continuously in the transversal direction.  There exists a canonical projection $\pi:{\cL}_\mathfrak{h}\to\Sigma$.  For a  dense set of laminated complex structures,
the restriction of $\pi$ to each leaf is a conformal map onto $\Sigma$ endowed with a complex structure independent of the leaf. Moreover the inverse limit of a point
 $K_z:=\pi^{-1}\{z\}, \,\,\,z\in\Sigma$ is a Cantor set.  
 In this construction, one gets the same inverse limit using any co-final set of finite coverings, for instance, normal subgroups or even characteristic subgroups. In these cases, $K_z$ is canonically a nonabelian Cantor group \cite{Sa, Su2, MS, BNS}.
 
 \noi The lamination $\cL_\mathcal{h}$ is the suspension of a homeomorphism $\rho:\pi_1(\Sigma)\to{\mathcal{H}(K)}$ (see definition \ref{nullcobordant}).
 If $\Sigma$ is a surface of genus two we can consider a simple closed curve $\gamma$ in $\Sigma$ which separates the surface into two surfaces of genus one with boundary $\gamma$.
 
 \noi The restriction of the representation to $\gamma$ produces an oriented 1-dimensional solenoid. Thus there exists four homeomorphisms  $f_1,\,\,f_2,\,\,g_1,\,\,g_2$ of the Cantor group $K_z$ such that $[f_1,f_2]=[g_1,g_2]:=h$ and the 1-dimensional solenoid is the suspension of $h$.
 These four homeomorphisms
 of the Cantor set satisfying the commutator relations above determine the universal solenoid.  
\noindent Sullivan constructs the universal Teichmüller space of the ``pointed''
 lamination $\cL_\mathcal{h}$
 obtained by taking the inverse limit of all finite pointed unbranched coverings of a compact surface of genus greater than one and a chosen base point. The base points upstairs in the coverings determine a point and a distinguished leaf called the {\it base leaf}.
\begin{definition}\label{UHLdef} The lamination $\cL_\mathcal{h}$ defined in the previous paragraph
as the inverse limit of the tower of pointed coverings of an orientable compact surface of genus 2 is called the {\it Universal Hyperbolic Solenoid}. It was defined by Dennis Sullivan in \cite{Su2}.
Note that if we take the inverse limit of the tower
of all regular finite covering surfaces over any, oriented, compact, surface
$\Sigma_g$ of genus $g\geq2$
we obtain a lamination that is diffeomorphic to $\cL_\mathcal{h}$. The diffeomorphism is given by a finite iteration of the {\it shift map} (defined in
{\it remark} \ref{chain-cofinal}).
Thus for each $g\geq2$ there
is a fibration
\begin{equation}\label{projection-g>2}
\Pi_g:\cL_\mathcal{h}\to\Sigma_g
\end{equation}
\end{definition}
 
 \begin{theorem} The space of hyperbolic structures on a  hyperbolic compact solenoidal surface (as in {\it definition} \ref{lam-teich}) up
to isometries isotopic to the identity has the structure of a separable complex Banach manifold. The metric is the natural Teichmüller metric based on the minimal conformal distortion of a
map between structures. The isotopy classes of homeomorphisms preserving a chosen leaf $L$ (the base leaf) in $\cL_\mathcal{h}$  act by
isometries on this Banach manifold.
 \end{theorem}
\noi This lamination is a {\it universal compact surface}, \ie it is the universal object in the category of finite unbranched coverings of a compact surface
This is the first example of a Teichmüller space that is separable but not finite-dimensional. We recall that Teichmüller spaces of compact Riemann surfaces of finite type are finite-dimensional complex manifolds and that Teichmüller spaces of geometrically infinite Riemann surfaces are non-separable infinite-dimensional complex Banach manifolds.
 
\noi In this space the commensurability automorphism group of the fundamental group of any higher genus compact surface acts by isometries. This group is independent of the genus by definition \cite{Sa, BNS, BPS, Od, Su2}.
 
\begin{theorem}[Sullivan]  
The space of hyperbolic structures up to isometry preserving the distinguished leaf on this solenoidal surface $\cL_\mathcal{h}$ is non-Hausdorff and any Hausdorff quotient is a point.  
 \end{theorem}

 \noindent The proof of this result relies on deep results by Jeremy Kahn and Vladimir Markovi\'c on the validity of the Ehrenpreis Conjecture \cite{KM}.
 Sullivan's observation is that the action of the {\it universal commensurability automorphism group of the fundamental group} is by isometries
(and thus biholomorphic modular transformations) and the action is minimal.
 The action is described in \cite{BNS} (Remark (\ref{vaut}) below).
 
    \smallskip  
 
 \subsection{The Earle-Eells Theorem for the universal hyperbolic solenoid}
 \label{EaE}
The fiber bundle description of Teichmüller space of a smooth, compact orientable
surface of genus $g\geq2$ given in \cite{EaE} has a version for the universal hyperbolic lamination $\cL_\mathcal{h}$ in {\it definition} \ref{UHLdef}.
 
\noi Let $\mathcal{D}_0(\cL_\mathcal{h})$ be the topological group of diffeomorphisms of $\cL_\mathcal{h}$ which are isotopic to the identity (hence they preserve each leaf). Endow this group
with the $C^\infty$-topology of uniform convergence of differentials
of all orders. The group becomes a Fréchet manifold \cite{Le}.
 
Let $M(\cL_\mathcal{h})$ be the set of all smooth complex structures on
on $\cL_\mathcal{h}$ compatible with the standard orientation of $\cL_\mathcal{h}$ induced by an orientation on the surface of genus 2 and give  $M(\cL_\mathcal{h})$ the $C^\infty$-topology. Then (viewing the elements of $M(\cL_\mathcal{h})$ as smooth tensor fields on $\cL_\mathcal{h}$) we have a natural right action: ${M(\cL_\mathcal{h})}\times\mathcal{D}_0(\cL_\mathcal{h})\to{M(\cL_\mathcal{h})}$, via pullback. Recall that a complex structure on $\cL_\mathcal{h}$ is a smooth bundle automorphism
$J:T(\cL_\mathcal{h})\to{T(\cL_\mathcal{h})}$,
such that for each $x\in\cL_\mathcal{h}$ the linear map $J_x:T_x(\cL_\mathcal{h})\to{T_x(\cL_\mathcal{h})}$
satisfies $J_x^2=-I_x$ ($I_x$ is the identity map on $T_x(\cL_\mathcal{h})$).
 
\noi The structure
$J\in{M(\cL_\mathcal{h})}$ defines
a conformal class of laminated Riemannian metrics. These are those metrics
such that for every point  $x\in\cL_\mathcal{h}$ and each non-zero vector $V_x\in{T_x(\cL_\mathcal{h})}$, the vectors $V_x$ and $J_x(V_x)$ are orthogonal.
 
\noi The uniformization {\it theorem} \ref{continuity} associates to each conformal class of metrics in $\cL_h$ a unique
laminated hyperbolic metric with all leaves of curvature -1.
 
\noi Reciprocally, the
laminated global isothermal coordinates of {\it corollary} \ref{isothermal} associates a complex structure on the lamination for each hyperbolic metric in $\text{Hyp}(\cL_\mathcal{h})$ (the set of hyperbolic metrics on $\cL_\mathcal{h}$). If $\rm{Conf}(\cL_\mathcal{h})$ denotes the set of conformal
equivalence classes of laminated metrics on $\cL_\mathcal{h}$,
we have canonical identifications:
\[
\rm{Conf}(\cL_\mathcal{h})\sim{M(\cL_\mathcal{h})}\sim\text{Hyp}(\cL_\mathcal{h}).
\]
\noi With these identifications each of the sets becomes a Fréchet manifold.
 
\begin{theorem}[Fiber bundle description of the Teichmüller space of $\cL_{\mathfrak{h}}$] The following holds:
 
\begin{enumerate}
\item $\text{Hyp}(\cL_\mathcal{h})$ is a contractible complex analytic manifold modeled on a
Fréchet space.
\item $\mathcal{D}_0(\cL_\mathcal{h})$
 acts (via pullback) continuously, effectively, and properly on $\text{Hyp}(\cL_\mathcal{h})$.
 
\item The quotient map
\begin{equation}
\Phi:\rm{Hyp}(\cL_\mathcal{h})
\to{\mathcal{T}}(\cL_\mathcal{h})=\rm{Hyp}(\cL_\mathcal{h})/\mathcal{D}_0(\cL_\mathcal{h})\,\,(=M(\cL_\mathcal{h})/\mathcal{D}_0(\cL_\mathcal{h})),
\end{equation}
where
$\mathcal{T}(\cL_\mathcal{h})$ the space of orbits (endowed with the quotient topology) is a universal
principal $\mathcal{D}_0(\cL_\mathcal{h})$ bundle.
\item  Both $\mathcal{D}_0(\cL_\mathcal{h})$ and the
Teichmüller space $\mathcal{T}(\cL_\mathcal{h})$ are contractible.
\end{enumerate}
\end{theorem}
 
\begin{proof} The proof uses harmonic maps and follows the steps in \cite{EaE}, \S{8E}.
 
\noi Let $g_0,g_1\in\rm{Hyp}(\cL_\mathcal{h})$, be two hyperbolic metrics. Let
  $f\in\mathcal{D}_0(\cL_\mathcal{h})$ and consider the Dirichlet energy of
 $f$ with respect to the two metrics \ie $f:(\cL_\mathcal{h},g_1)\to(\cL_\mathcal{h},g_0)$:
\begin{equation}\label{surfaceenergy}
E(f)=\frac12\int_{\cL_\mathcal{h}}||df(x)||^2\,\rm{d}\mu_{g_0}(x),
\end{equation}
 
\noi where $\mu_{g_0}$ is the laminated volume form ({\it definition} \ref{lamvol}) on $\cL_\mathcal{h}$ with respect to $g_0$, and $||\,\cdot\,||$ the Hilbert-Schmidt norm with respect to $g_0$
and $g_1$ ({\it definition} \ref{energycoordinates}).
 
\noi By {\it theorem} \ref{Solenoidal-Eells-Sampson} there exists a unique harmonic
map $f_{(g_1,g_0)}:(\cL_\mathcal{h},g_1)\to(\cL_\mathcal{h},g_0)$ homotopic to $f$ and depending continuously on $(g_1,g_0)$.
An adaptation of a theorem by R. Schoen and S. T. Yau in \cite{SY} implies that $f_{(g_1,g_0)}$ is a diffeomorphism. Their proof is for compact surfaces with negative curvature but it can be applied in our case because
their proof uses the fact that the Jacobian has only isolated zeros
and the maximum principle shows that the Jacobian does not vanish at any point.
 
 Therefore the Jacobian of $f_{(g_1,g_0)}$ does not vanish and it must be a local (laminated) diffeomorphism. The fact that $\cL_\mathcal{h}$ is compact and connected and
that $f_{(g_1,g_0)}$ is homotopic to the identity implies
it is a global diffeomorphism.
 
\noi In what follows we fix the hyperbolic metric $g_1$.

\noi Let
\[
F:\rm{Hyp}(\cL_\mathcal{h})\to\mathcal{D}_0(\cL_\mathcal{h})
\]
be the map
$F(g_0)=f_{(g_1,g_0)}$. If $h\in\mathcal{D}_0(\cL_\mathcal{h})$
and $g_0\cdot{h}$ is
the pullback under $h$ of the metric $g$ we have that
$h:(\cL_\mathcal{h}, g_0\cdot{h})\to(\cL_\mathcal{h}, g_0)$ is an isometry. The post-composition of a harmonic map with an isometry is a harmonic map, therefore by uniqueness, we have the
commutative diagram:
 
\begin{equation}\label{harmonic.isometry}
  \begin{tikzcd}
  (\cL_\mathcal{h}, g_1)   \arrow{r}{f_{(g_1,g_0)}} \arrow[swap]{dr}{f_{(g_1,g_0.h)}} &  (\cL_\mathcal{h}, g_0)  \\
     & (\cL_\mathcal{h}, g_0\cdot{h}) \arrow{u}{h},
  \end{tikzcd}
\end{equation}
so that
\begin{equation}\label{cdh}
f_{(g_1,g_0)}=h\circ{f_{(g_1,g_0\cdot{h})}}, \quad\forall\, h\in\mathcal{D}_0(\cL_\mathcal{h}).
\end{equation}
\noi Define $\Psi:\text{Hyp}(\cL_\mathcal{h})\to\mathcal{T}(\cL_\mathcal{h})\times \mathcal{D}_0(\cL_\mathcal{h})$
by the formula $\Psi(g_0)=(\Phi(g_0), f_{(g_1,g_0)}^{-1})$. The function $\Psi$
is continuous.
The group
$\mathcal{D}_0(\cL_\mathcal{h})$ acts on $\mathcal{T}(\cL_\mathcal{h})\times\text{Hyp}(\cL_\mathcal{h})$
as follows: $(\tau, g)\overset{h}\mapsto(\tau, g\circ{h}):=(\tau, g)\cdot{h}$.
 
\noi Equation (\ref{cdh}) implies
\begin{equation}\label{cdh2}
\Psi(g_0\cdot{h})=\Psi(g_0)\cdot{h}=
(\Phi(g_0), f_{(g_1,g_0)}^{-1}\circ{h}),\quad \forall\, h\in\mathcal{D}_0(\cL_\mathcal{h}).
\end{equation}
 
\vskip1cm
The function $\Psi$ is injective: if $\Psi(g_1)=\Psi(g_2)$  then $\Phi(g_1)=\Phi(g_2)$. Therefore
there exists  $h\in\mathcal{D}_0(\cL_\mathcal{h})$ such that
$g_2=g_1\cdot{h}$ so that by equation (\ref{cdh2}) $\Psi(g_1)=\Psi(g_2)=\Psi(g_1)\cdot{h}$ and therefore $h=\rm{Id}$ and $g_1=g_2$.
On the other hand, $\Psi$ is surjective: given
$(g,h)\in\text{Hyp}(\cL_\mathcal{h})\times\mathcal{D}_0(\cL_\mathcal{h})$ one has
\[
(\Phi(g),h)=\Psi(g)\cdot(f_{(g_1,g)}\circ{h})
=\Psi(g\cdot(f_{(g_1,g)}\circ{h})),
\]
\noi since $\Phi(g)=\Phi(g\cdot(f_{(g_1,g)}\circ{h}))$.
From this we conclude that $\Psi$ is a homeomorphism, and $\Phi$ is a
principal fiber bundle over $\mathcal{T}(\cL_\mathcal{h})$
with fiber the Fréchet group $\mathcal{D}_0(\cL_\mathcal{h})$.
 
\noi Formula (\ref{cdh2})
establishes a bundle isomorphism between $\Phi$ and the trivial bundle
$p_1:\mathcal{T}(\cL_\mathcal{h})
\times\mathcal{D}_0(\cL_\mathcal{h})\to\mathcal{T}(\cL_\mathcal{h})$. An explicit
section $\sigma:\mathcal{T}(\cL_\mathcal{h})\to\text{Hyp}(\cL_\mathcal{h})$
is given by the formula:
$\sigma(\Phi(g))=\Psi^{-1}(\Phi(g),\rm{Id})=g\cdot{f_{(g_1,g)}}$.
 
\noi The space of all smooth Riemannian metrics on $\cL_\mathcal{h}$ is a convex set
in the Fréchet space of sections of tensors of type $(0,2)$, since a convex linear combination of Riemannian metrics on $\cL_\mathcal{h}$ is a Riemannian metric.
 
\noi In addition, each conformal class of Riemannian metrics is also convex. In each of these conformal classes, there is a unique hyperbolic metric and this metric varies continuously with the conformal class. Therefore $\rm{Hyp}(\cL_\mathcal{h})$
is contractible. Hence the fibration
$\Phi:\mathcal{T}(\cL_\mathcal{h})\to\mathcal{T}(\cL_\mathcal{h})$
is universal. The existence of the locally trivial fiber bundle map $\Phi$
implies that both $\mathcal{T}(\cL_\mathcal{h})$
and $\mathcal{D}_0(\cL_\mathcal{h})$
are contractible.  \end{proof}
\begin{remark}\label{vaut} The Teichmüller theory of $\cL_\mathcal{h}$ has been studied
from the viewpoint  of quasi-conformal mappings in \cite{MSa} \cite{PS},\cite{Sa1},  \cite{Sa3}. See
also the results in \cite{BN}, \cite{BN1} \cite{BNS}, \cite{BV}, \cite{Od} and, of course,\cite{Su2}.
 
\noi A very important part of the Teichmüller theory of  $\cL_\mathcal{h}$
is played by the transversely {\it locally constant complex structures} (TLC).
Particular examples of these are obtained by taking any compact Riemann surface of an arbitrary genus and lifting the complex
structure to the leaves of $\cL_\mathcal{h}$ using the projection $\Pi_g$
in {\it remark} \ref{projection-g>2}, {\it definition} \ref{canonical-projection}.
By \cite{Su2} these TLC structures constitute a dense set of the universal
Teichmüller space of $\cL_\mathcal{h}$.
 
\noi The notions of {\it mapping class group} and {\it base leaf mapping class group}
exist for the solenoidal surface $\cL_\mathcal{h}$ and play an important role. The
{\it abstract virtual automorphism group} ({\it universal commensurability automorphism group}) of the fundamental group of a compact surface of genus 2 provides a solenoidal version of Nielsen's theorem that intertwines
the fundamental group of a surface and its mapping class group. See, for instance, the paper by Chris Odden \cite{Od} as well as \cite{BNS}, \cite{BN1}, \cite{BPS}, \cite{MSa}, \cite{Su2}.
\end{remark}
 
\subsection{The universal 2-dimensional euclidean solenoid}\label{hom2d}
 Theorem \ref{tophom1d} has a higher-dimensional analog for the topologically transitive McCord Solenoids
 which are obtained by an infinite tower of coverings of tori:
 \begin{theorem}\label{McCord-tori} Let $\cS$ be a compact McCord solenoid
 which is obtained as the inverse limit of finite coverings of an $n$-torus
 $\T^n=\sS^1\times\cdots\times\sS^1$ (we call it a McCord solenoid over a torus). Then $\cS$ is a compact connected abelian group. The Pontryagin dual of
 $\sS$ is a dense subgroup of the additive group of the vector space $\Q^n$
 over $\Q$. There is a bijective correspondence between the $n$-dimensional
 McCord solenoids over $\T^n$ and dense subgroups
 of $(\Q^n,+)$.
 \end{theorem}
 
 \begin{proof} The proof is the same as the proof of {\it theorem} \ref{tophom1d}.
 Every covering map $p:\T^n\to\T^n$ is regular and it is a surjective endomorphism corresponding
 to the image of an injective homomorphism $\phi:\Z^n\to\Z^n$, of the fundamental group
 of $\pi_1(\T^n)=\Z^n$. The injective homomorphism $\phi$ is
 determined by a nonsingular $n\times{n}$ matrix $A=(a_{ij})$, with integer coefficients.
 If $\sS^1=\{z\in\C:|z|=1\}$, then the endomorphism induced by $A$ is given by:
 \[
 (z_1,\cdots,z_n)\longmapsto(z_1^{a_{11}}z_2^{a_{12}}\cdots{z_n^{a_{1n}}},
 \cdots, z_1^{a_{j1}}z_2^{a_{j2}}\cdots{z_n^{a_{jn}}},  \cdots,z_1^{a_{n1}}z_2^{a_{n2}}\cdots{z_n^{a_{nn}}}).
 \]
 Therefore, the inverse limit is a compact abelian group, as the bonding maps are group homomorphisms between compact abelian groups. Since the topological dimension of $\cS$ is $n$, it follows that its Pontryagin dual is a subgroup of the additive group $\Gamma$ of $\Q^n$. Since the inverse limit is a solenoid
 it follows that $\Gamma$ must be dense in $\Q^n$. \end{proof}
 \noi Again by \cite{Sche}, to nonisomorphic, additive, dense subgroups of $\Q^n$
 correspond nonhomeomorphic $n$-dimensional solenoids.
  In particular, when
 $\Gamma=\Q^n$ we have the algebraic universal covering solenoid of $\T^n$:  
 
 \begin{definition}
 $\widehat\T^n$ will denote the {\it algebraic universal covering of $\T^n$} \ie the inverse limit
 of {\it all} its finite coverings.
 \end{definition}
\noi Since $\widehat\T^n$ is the Pontryagin dual of $\Q^n=\Q\oplus\cdots\oplus\Q$
and the dual of $\Q$ is the algebraic universal solenoidal covering of the circle
$\hat\sS^1$ we have:

 \begin{proposition}\label{profinite-is-a-product} $\widehat\T^n$, the algebraic universal covering of $\T^n$, is isomorphic to the product
 $\hat\sS^1\times\cdots\times\hat\sS^1$ ($n$ factors). The algebraic fundamental group of $\widehat\T^n$ is $\hat\Z^n=\hat\Z\times\cdots\times\hat\Z$, with
 $\hat\Z$ the profinite completion of the integers.
 \end{proposition}
 \noi Therefore, $\Z^n=$ acts freely and properly on $\R^n\times\hat\Z^n$. Using
 the notation in {\it proposition} \ref{algcover} we have:
\[
\widehat\T^n=\Z^n \backslash(\R^n\times\hat\Z^n)=\hat\sS^1\times\cdots\times\hat\sS^1,
\quad (n\,\,\, \text{factors})
\]

\noi In particular for the special case  $n=2$ we have:
 
 \begin{definition}[Universal euclidean solenoidal surface]  
 $\widehat\T^2$,  
the algebraic universal covering of the 2-torus,
will be called the {\em universal euclidean solenoidal surface}.
 $\widehat\T^2=\hat\sS^1\times\hat\sS^1$ and its algebraic fundamental group is
 $\widehat{\pi_1(\T^2)}=\hat\Z\times\hat\Z$, It is the Pontryagin dual of
 $(\Q^2,+)$. The component of the identity is called the {\it base leaf}.
\end{definition}
 
\begin{remark}\label{n>1=no-classificatiion} As shown in {\it theorem} \ref{tophom1d} the classification of compact topologically homogeneous 1-dimensional solenoids is equivalent to the classification of the dense subgroups of the additive rationals and there is a good understanding of these subgroups due to Baer \cite{Baer}.
One may try to classify the $n$-dimensional solenoids described in theorem \ref{McCord-tori} for $n>1$ but this could be an impossible task. The classification is equivalent to the classification of abelian, torsion-free, groups of finite rank $n>1$, whose Pontryagin dual is a solenoid \ie dense subgroups of the additive group $(\Q^n,+)$. For $n>1$, there is no hope for a classification,
the complexity increases with $n$.
We refer to work
due to S. Adams \cite{A}, G. Hjorth \cite{Hj}, A. S. Kechris \cite{Ke} and S. Thomas \cite{Tho}.
They work in the context of both abelian group theory and countable Borel equivalence relations and provide examples of countable Borel equivalence relations as well as a new way of thinking about the complexity of the classification problem for torsion-free abelian groups of finite rank.
\end{remark}
 
\begin{remark} By Bieberbach's theorem, given a compact flat $n$-manifold
$N$ with a flat metric $g$, there exists a flat
$n$-torus $(\T^n, \tilde{g})$ and a locally isometric covering projection
$\pi_1:\T^n\to{N}$. If $\cS^n_{\T^n}$ is the solenoid obtained
from an infinite tower of regular coverings of $(\T^n,\tilde{g})$ then we can endow
$\cS^n_{\T^n}$ with the metric $\hat{g}$ obtained by lifting the metric on the leaves by the canonical projection $\Pi:\cS^n_{\T^n}\to\T^n$.
 
\noi The metric $\hat{g}$ renders all leaves flat and it is homogeneous:
the group of isometries of $(\cS^n_{\T^n}, \hat{g})$ acts transitively. This is an example
of a {\it compact, homogeneous, flat $n$-dimensional solenoid}. In fact, it is a compact $n$-dimensional abelian group.
 
\medskip
\noi {\it Remark} \ref{n>1=no-classificatiion}
implies that the task of classifying flat, homogeneous, $n$-dimensional solenoids is hopeless if $n>1$. However the algebraic universal covering of a flat torus cover, in the solenoidal sense, every such flat solenoid: all are of the form
$(\R^n\times\widehat{\Z^n})/(\Z^n\times\Gamma_\alpha)$ where $\Gamma_\alpha$ is a closed subgroup
of $\widehat{\Z^n}$ with quotient $\widehat{\Z^n}/\Gamma_\alpha$ a Cantor group.
\end{remark}
 
\section{Geometric 3-dimensional solenoids} \label{g3dsm}
 \subsection{Locally homogeneous 3-dimensional solenoidal Laminations} Let $(N,G)$ be one of the eight Thurston geometries  where $N$
is an oriented, simply connected, geometric 3-manifold and $G$ the corresponding
group of {\it orientation-preserving} isometries acting transitively on $N$. Here is the list of these geometric manifolds with their corresponding group of
isometries which homotopic to the identity:
 
\begin{multiitem}\label{listgeom}
    \item[]{\bf Geometric Manifold}
    \item[(1)] $N=U(2)=\sS^3$ the 3-sphere
    \item[(2)] $N=\mathbb E^3$ euclidean 3-space
    \item[(3)] $N=\H^3$ hyperbolic 3-space
    \item[(4)] $N=\widetilde\PSL$
    \item[(5)] $N=\H^2\times\R$
    \item[(6)]$N=\sS^2\times\R$
    \item[(7)]$N=Nil_3$
    \item[(8)]$N=Solv_3$

    \item[] {\bf Group of Isometries}
    \item $G_0=SO(4)$
    \item $G_0={\mathbb E}^3\rtimes SO(3)$
    \item $G_0=PSL(2,\C)$
    \item $0\to\R\to{G_0}\to\PSL\to1$
    \item $G_0=\PSL\times\R$    
    \item $G_0=SO(3)\times\R$
    \item $G_0={Nil_3}\rtimes{SO(2)}$    
    \item $G_0=Solv_3$
 
  \end{multiitem}
 
\noi In this list $G_0$ denotes the connected component of the group of isometries of the corresponding geometric manifold $(N,G)$ (see \cite{Sc} for further details).
 
\noi See C. T. McMullen's paper \cite{McM1} for a
useful historical survey.
 
\noi ${\mathbb E}^3\rtimes SO(3)$ is the group of direct Euclidean motions,
 $\widetilde\PSL$ is the universal covering of $\PSL$ with a left-invariant
 Riemannian metric, the corresponding group $G_0$ is an extension of $\R$ by
 $\PSL=Isom_+(\H)^2$. The manifold
$N=Nil_3$ is the Heisenberg group \ie the group upper triangular matrices with real coefficients and
$N=Solv_3$, is the simply connected solvable Lie group with Lie algebra generated by $X$, $Y$, and $Z$ with the relation
$[X,Y]=Y$, $[X,Z]=-Z$, $[Y,Z]=X$, both endowed with left-invariant Riemannian  metrics.
\begin{remark} Every compact geometric 3-dimensional manifold modeled on $(N,G)$ admits a compact covering by a geometric manifold modeled on $(N,G_0)$, where the covering projection is a local isometry. Therefore by
{\it proposition \ref{smooth-McCord}} item \ref{shift-equivalent} we can consider
McCord solenoids based on geometric manifolds modeled on $(N,G_0)$.
 
\end{remark}
 
\begin{definition} A compact {\em $3$-dimensional geometric solenoidal lamination} modeled on the Thurston model $(N,G)$ or, briefly, a {\em geometric solenoidal 3-manifold}
$\M$ modeled on $(N,G)$, is a solenoidal 3-manifold in the sense of
Definition (\ref{definition:lamination}) with special charts.
It consists of a compact metrizable space $\M$, together
with a family $\{(U_\alpha, \varphi_\alpha)\}$
such that:
\begin{itemize}
\item[\textbullet] $\{U_\alpha\}$ is an open covering of $\M$,
\item[\textbullet] $\varphi_\alpha: U_\alpha \to V_\alpha\times T_\alpha$ is a
    homeomorphism, where $V_\alpha$ is an open subset of $N$
 and $T_\alpha$
is an open subset of the Cantor set $K$.
\item[\textbullet] for $(y,t)\in \varphi_\beta(U_\alpha\cap U_\beta)$,
$\varphi_\alpha \circ \varphi_\beta^{-1}(y,t) = (g_{\alpha\beta}(t)(y), h_{\alpha\beta}(t))$
and $g_{\alpha\beta}(t)\in{G}$.
\end{itemize}
\end{definition}
 
\noi Therefore, as in {\it definition} \ref{definition:lamination}, one has a lamination
which is the partition of $\M$ into 3-dimensional leaves and a Cantor transverse structure.
 
\noi
The leaves are immersed 3-dimensional manifolds with a geometric structure modeled on the geometric  Thurston model $(N,G)$,
where the local charts are the restrictions of $p_1\circ\varphi_\alpha$'s to plaques
and $p_1:V_\alpha\times T_\alpha\to V_\alpha$
is the projection onto the first factor. The {\it locally homogeneous}
Riemannian metric on the leaves varies continuously in the transverse direction so
we can endow $\M$ with a smooth laminated, locally homogeneous, Riemannian metric $g$.
 
\begin{remark} The case of geometric 3-dimensional solenoidal manifolds modeled on
$(\sS^3,SO(4))$ is not interesting, since any such solenoid is of the form
$M\times{K}$ where $M$ is one of the spherical manifolds (lens spaces and the finite set of manifolds corresponding to the finite non-cyclic subgroups of $SO(4)$).
 
\noi {\it We will not deal with the spherical case.}
\end{remark}
\begin{remark}\label{exceptional3solenoids} Every 3-dimensional compact, connected, geometric solenoid, modeled on $\sS^2\times\R$, which is a McCord solenoid is of the form $\sS^2\times\mathcal{S}$ with $\mathcal{S}$ a compact, connected, abelian 1-dimensional solenoidal group. Every 3-dimensional compact, connected, geometric McCord solenoid, modeled on $\H^2\times\R$ splits as a product $\mathcal{S}^2\times\mathcal{S}^1$, with $\mathcal{S}^2$ a compact hyperbolic lamination and
$\mathcal{S}^1$ a compact abelian group.
\end{remark}
 
\begin{example}[Algebraic universal coverings of geometric 3-manifolds]
Let $(M,G)$ be a compact geometric manifold modeled on the geometric model $(N, G)$ {\it different from $(\sS^3,SO(4))$}. The fundamental group
of $M$ is infinite and residually finite. Let $\mathbb{M}$ be its algebraic universal covering as in {\it definition} \ref{univ-alg-cov} formula (\ref{towercovers}) \ie the inverse limit of the tower of all its finite pointed coverings:
\begin{equation}
\mathbb{M}={\lim_{\overset{p_{_{(\beta_1,\beta_2)}}}{\longleftarrow} }} \,\,
\{p_{_{(\beta_1,\beta_2)}}:
(\tilde{M}_{\beta_1},\tilde{x}_{\beta_1})\to
(\tilde{M}_{\beta_2},\tilde{x}_{\beta_2}),\,\,\,
{\beta_1},\,{\beta_2}\in\mathfrak{N},\,\beta_1\preceq\beta_2\}
\end{equation}
\end{example}
\noi By {\it proposition} \ref{McCord-minimal},
 $\mathbb{M}$
is a minimal  {\em 3-dimensional solenoidal manifold}, \ie with dense leaves.
We endow  $\mathbb{M}$ with the laminated metric $g_{can}$ obtained as the pullback of the homogeneous metric under the
canonical projection $\Pi:\mathbb{M}\to{M}$. Each leaf is
 isometric to $N$ and the restriction of the projection $\Pi_\beta:\mathbb{M}\to{M}_\beta$ to
 any leaf is a local isometry.
 \begin{definition}\label{canonicalmetric}The laminated Riemannian metric $g_{can}$ is called the {\it canonical metric} of $\mathbb{M}$. In general, any McCord solenoid over a geometric 3-manifold $M$ has a {\it canonical laminated metric} obtained by the pullback
 of the homogeneous metric on $M$.
\end{definition}
 
\noi If $M$ is a geometric manifold modeled on $(N,G)$ the fundamental group of $M$ is identified with a subgroup of $G$ which acts properly discontinuously,
freely, and isometrically on $N$ and the quotient space is isometric to $M$ so that $\tilde{M}$, the universal covering of $M$, is isometric to $N$.

\noi The profinite completion $\widehat{\pi_1(M)}$ acts by fiberwise  right translations on $\tilde{M}\times{\widehat{\pi_1(M)}}$ (thus without fixed points). It acts simply-transitive on
$\widehat{\pi_1(M)}$ on the fibers. This shows again that the leaves on $\mathbb M$ are dense
(since $j\left(\pi_1(M)\right)$ is dense in $\widehat{\pi_1(M)}$) and isometric to $N$.
 
\begin{proposition} With $M$ and $(N,G)$ as above, the algebraic universal
covering $\mathbb M$ of $M$ with the canonical metric $g_{can}$ of
{\it definition} \ref{canonicalmetric} is a geometric solenoidal 3-manifold modeled on the Thurston model $(N,G)$. More generally any McCord solenoid obtained as the inverse
limit of an increasing tower of regular coverings of $M$ is also a geometric 3-dimensional solenoidal manifold with respect to the canonical laminated metric.
\end{proposition}

\begin{definition} Two compact 3-dimensional manifolds $M_1$ and $M_2$ are said to be {\em topologically commensurable} if there exists a common
finite covering \ie if there exists a compact manifold $M_3$ and coverings $p_1:M_3\to{M_1}$ and $p_2:M_3\to{M_2}$.
If $M_1$ and $M_2$ are commensurable, then their algebraic universal solenoidal covering spaces are diffeomorphic,
in the solenoidal sense, and vice versa.

\end{definition}
 \begin{definition} Let $M_1$ and $M_2$ be  geometric 3-manifolds
 which are both modeled on the Thurston model $(N,G)$.
 Then $M_1$ and $M_2$ are said to be {\em geometrically commensurable}
 if there exists a geometric manifold $M_3$ modeled also on $(N,G)$
 and coverings $p_1:M_3\to{M_1}$ and $p_2:M_3\to{M_2}$ which are locally isometries.
  \end{definition}
\noi The following proposition follows immediately from the definition of algebraic universal covering and its canonical solenoidal Riemannian metric:
 
\begin{proposition} If $M_1$ and $M_2$ are two geometrically commensurable geometric 3-manifolds
 which are both modeled on the Thurston model $(N,G)$, then their algebraic universal solenoidal covering spaces, with their canonical metrics, are isometric. Reciprocally, if their algebraic universal solenoidal coverings are isometric then the manifolds are geometrically commensurable.
\end{proposition}
 
\begin{corollary}
The geometric commensurability classes of 3-dimensional geometric manifolds (\`a la Thurston) are determined by the isometry classes of their
solenoidal algebraic covering spaces and reciprocally. Therefore, solenoidal algebraic universal covering spaces are in bijection with the commensurability classes.
 The commensurability class of the geometric 3-manifolds is determined by the profinite completion
of the fundamental group of one of its members.
\end{corollary}
The profinite completion of fundamental groups can detect many properties
of 3-manifolds as shown in the papers \cite{WZ1}, \cite{WZ2}, \cite{WZ3} by H. Wilton and P. Zalesskii.
 
\subsection{Hyperbolic Solenoids. The spherical solenoidal boundary}\label{hypsol}
 
\noi Let us identify $n$-dimensional hyperbolic space $\H^n$ with the concrete Poincaré disk model $(D^n,\mathcal{h})$
where $D^n=\{(x_1,\cdots,x_n)\in\R^n\,:\, x_1^2+\cdots+x_n^2<1\}$ and $\mathcal{h}$
is the Poincaré Riemannian metric
$\mathcal{h}(x_1,\cdots,x_n)=4(1-\sum_1^n\,x_i^2)^{-2}\sum_1^n\,dx_i^2\,$, of constant sectional curvature minus one. The group of orientation-preserving isometries of $(D^n,\mathcal{h})$ is the Lie group $\rm{SO}^+(n,1)$.
 
\noi If $M$ is an $n$-dimensional hyperbolic manifold, the Cartan-Hadamard theorem implies that its Riemannian universal covering manifold is the hyperbolic $n$-space $\H^n$. Using the Poincaré disk model, it can be compactified by adding the boundary unit sphere, which becomes the ``sphere at infinity'' $\sS^{n-1}$, with its natural conformal structure. The group of direct hyperbolic isometries of the Poincaré disk, $\rm{SO}^+(n,1)$, extends continuously as the group of conformal orientation-preserving diffeomorphisms of the sphere. We denote by
$\overline\H^n=\H^n\cup\sS^{n-1}$ this compactification which we can identify
with the closed unit $n$-disk $\overline{D}^n$ in $\R^n$.
\noi References \cite{BP}, \cite{Ka}, \cite{Th1}, for instance, contain the
information we need regarding hyperbolic manifolds.

\noi Let $M$ be a compact, orientable, hyperbolic $n$-manifold. Therefore, by the Cartan-Hadamard theorem, there is an injective representation
$\rho_M:\pi_1(M)\to{\rm{SO}^+(n,1)}$, whose image $\Gamma$ is a lattice, acting properly discontinuously and co-compactly by isometries on
$(D^n,\mathcal{h})$, with orbit space $D^n/\Gamma$ isometric to $M$.
 
\noi Henceforth $\wg$ will denote the profinite completion
$\widehat{\pi_1(M)}$ of $\Gamma=\pi_1(M)$.
 
\noi Then, $\wg$ is a Cantor group.
Let $\H^n_{_{\wg}}=D^n\times\widehat{\Gamma}=\H^n\times\widehat{\Gamma}$ be the lamination with leaves
$D^n_\kh=D^n\times\{\kh\}=\left\{\left((x_1,\cdots,x_n), \kh \right)\,:\, (x_1,\cdots,x_n)\in{D^n},\,\, \kh\in\widehat{\Gamma}\right\}$.
 
\noi Let $\mathcal{h}_\kh$ be
the Poincaré Riemannian metric
on the leaf $D^n_\kh$ given by $\mathcal{h}_\kh(x_1,\cdot,x_n)=4(1-\sum_1^n\,x_i^2)^{-2}\sum_1^n\,dx_i^2\,$ (\ie independent of $\kh$).
The family of metrics $\{\mathcal{h}_{\kh}\}_{\kh\in\wg}$ determines a laminated Riemannian metric
$\mathcal{h}_{\wg}$ on $\H^n_{_{\wg}}=D^n\times\widehat{\Gamma}$.
 
\begin{definition}[$\wg$-universal hyperbolic covering solenoid]\label{uhcs}
Let $\H^n_{_{\wg}}=\H^n\times\widehat{\Gamma}$ as in the previous paragraph. Using the Poincaré model, we identify also
$\H^n_{_{\wg}}$ with $D^n\times\wg$. For $n\geq3$ we call
the geometric solenoid $(\H^n_{_{\wg}},\mathcal{h}_{\wg})$ the
{\it $\wg$-universal $n$-dimensional Hyperbolic Solenoid}.
The reason for the name is (as we shall see in {\it proposition} \ref{SCHT}) that for $n\geq2$, $\H^n_{_{\wg}}$ covers (in the solenoidal sense)
any compact, $n$-dimensional, hyperbolic geometric solenoid which is a McCord solenoid
corresponding to an infinite tower of coverings of a compact hyperbolic manifold
$M$ with fundamental group $\Gamma$. Any manifold commensurable to
$M$ determines the same profinite group $\wg$, and therefore is also covered by
$\H^n_{_{\wg}}$.
\end{definition}
 \noi We can compactify $\H^n_{_{\wg}}=D^n\times\widehat{\Gamma}$
by adding to each leaf $L_{\kh}$ its boundary sphere at infinity
$\sS^{n-1}_\mathfrak{h}:=\sS^{n-1}\times\{\mathfrak{h}\}$. We have:
 
\begin{definition}[Universal spherical boundary solenoid]\label{compactification} The compact solenoid with boundary
\begin{equation}
\overline\H^n\times\wg\simeq\overline{D}^n\times\wg,
\end{equation}
\noi homeomorphic to the product of the closed $n$-disk $\overline{D}^n$ and the Cantor group $\wg$, is called the {\it compactification of the $\wg$-universal $n$-dimensional Hyperbolic Solenoid}. The compact subset $\sS^{n-1}_{\wg}$ defined as follows:
\begin{equation}\label{sph-bdy}
\sS^{n-1}_{\wg}:=\sS^{n-1}\times\wg\overset{\text{def}}=\partial\left(\overline\H^n\times\wg\right),
\end{equation}
\noi is homeomorphic to the product of the $(n-1)$-sphere and the Cantor group $\wg$, and it is called the
{\it spherical solenoidal boundary} of the $\wg$-universal $n$-dimensional Hyperbolic Solenoid. We will simply refer to it as the {\it spherical solenoidal boundary}
if $\wg$ is understood. Each spherical boundary leaf $\,\sS^{n-1}\times\{\mathfrak{h}\}$ has a natural conformal structure equivalent to the canonical conformal structure of the unit sphere in $\R^n$.
 \end{definition}
 
\begin{definition} [Hyperbolic McCord solenoidal manifolds] Let $\cL^n$ be a
compact, connected, oriented, $n$-dimensional McCord solenoidal manifold given by the fibration $\Pi:\cL^n\to{M}$  over the compact, connected, smooth manifold $M$ with fiber a Cantor group $K=\wg/\Gamma_\alpha$ ({\it proposition} \ref{uncountable-subgroups}). The pair
 $(\cL^n,g)$ s called a {\it Hyperbolic McCord solenoidal manifold} if $g$
 is a laminated Riemannian metric such  that each leaf has constant
 negative curvature equal to -1.  
 \end{definition}
 
 The universal hyperbolic lamination $\cL_\mathcal{h}$ in {\it definition} \ref{UHLdef} is an example of a 2-dimensional McCord solenoid that admits
 an infinite dimensional space of laminated hyperbolic Riemannian metrics because
 its Teichmüller space $\mathcal{T}(\cL_\mathcal{h})$ is infinite-dimensional.
 The example shows that not all laminated hyperbolic metrics on  $\cL_\mathcal{h}$
 are obtained by the pullback of a hyperbolic Riemannian metric on a compact hyperbolic surface. More generally if $M^n$ is a compact hyperbolic manifold then its
algebraic universal covering solenoidal manifold $\mathbb{M}$
becomes an $n$-dimensional McCord Hyperbolic solenoidal manifold if we endow $\mathbb{M}$ with the canonical metric $g_{cal}$ obtained by pulling back the unique (by Mostow's rigidity theorem) hyperbolic metric on $M$. In {\it theorem} \ref{base-hyp}
it will be shown that every Hyperbolic McCord solenoidal
manifold of dimension $n\geq3$ is isometric to one obtained by a sequence of finite coverings over a hyperbolic metric, and the metric is obtained by pullback
(see {\it remark} \ref{notallcanonical} below).
 \begin{remark}\label{lochom}One important feature of the McCord Solenoidal manifold
 $(\mathbb{M},g_{can})$ is that it is locally homogeneous.
 \end{remark}

\noi Let $(\cL^n,g)$ be a hyperbolic McCord solenoidal manifold.  By {it proposition} \ref{uncountable-subgroups} and {\it remark} \ref{UPM}, there is a compact, orientable manifold $M$ and a fibration $\Pi:\cL^n\to{M}$ with fiber a Cantor group
$\widehat{\pi_1(M)}/\Gamma_{\alpha}$ and a solenoidal covering map
$\tilde\Pi:\tilde{\mathbb{M}}=\tilde{M}\times\widehat{\pi_1(M)}\to\cL^n$, where
$\tM$ is the universal covering of $M$ ({\it proposition} \ref{uncountable-subgroups},  {\it formula} (\ref{form-of-solenoids}), and {\it proposition} \ref{algcover}, formulas (\ref{barM}), (\ref{pihat})).
Let $\tg$ be the lifting to $\mathbb{M}$ of the metric $g$.
We don't
know {\it a priori} that $M$ is a hyperbolic manifold (however this is true and
if $n\geq3$ it will be proven in {\it corollary} \ref{Mhyperbolic} and {\it Theorem} \ref{base-hyp}).
 
\begin{lemma}\label{sameuniversal} The lifted metric $\tilde{g}$ is leafwise isometric to the metric
$\mathcal{h}_{\wg}$ \ie there exists a homeomorphism
\begin{equation}\label{I}
I:(\H^n_{_{\wg}},\mathcal{h}_{\wg})\to(\tilde{\mathbb{M}},\tilde{g}),
\end{equation}
which is a leafwise isometry.
\end{lemma}
 
\begin{proof} Let $p:\H^n\times\wg\to\H^n$ be projection on the first factor. Then $p_\kh$, the restriction of $p$ to the leaf $L_\kh=\H^n\times\{\kh\}$, is an isometry. Let $x\in\H^n$ and let $F_\kh$ be an orthonormal
frame of the tangent space of the leaf $L_\kh$ at $(x,\kh)$, such that it has as image
the frame $F_x$, under the differential of $p_\kh$ \ie $dp_\kh(F_\kh)=F_x$.
Let $\tx\in\tilde{M}$ and
$K_{\tx}=\{({\tx},\mathfrak{h})
\,:\,\mathfrak{h}\in\wg\}$; for each point
${\tx}_{\mathfrak{h}}:=({\tx},\mathfrak{h})$
let $F'_{\kh}$ be an orthonormal framing, with respect to the metric $\tg$, at the tangent space ${\tx}_{\mathfrak{h}}$ of the leaf
$\tilde{M}\times\{\kh\}$ of $\tilde{\mathbb{M}}$. We can assume that the framing
$F'_{\kh}$ varies continuously with respect to $\kh$
(for instance by pulling back a framing at
the tangent space of $\tilde{M}$ at $\tx$ under the projection
$q:\tilde{M}\times\wg\to\tilde{M}$ and using Gram-Schmidt orthogonalization).
 
\noi For each $\mathfrak{h}$ there exists a unique isometry of leaves
$\mathfrak{I}_{\mathfrak{h}}:{L_{\kh}}=\H^n\times\{\kh\}\to\tilde{M}\times\{\kh\}$
such that the differential of $\mathfrak{I}_{\mathfrak{h}}$ sends
the frame $F_{\kh}$ onto the frame of $F'_\kh$. The desired leafwise isometry is given by the formula: $I(x,\kh)=\mathfrak{I}_{\mathfrak{h}}(x,\kh)\,$ ($x\in\H^n$, $\kh\in\wg$).
\end{proof}

\noi {\it Lemma} \ref{sameuniversal} implies the following solenoidal version of the Cartan-Hadamard theorem. It shows that $H^n_{\wg}=\H^n\times\wg$ covers
(in the solenoidal sense) every $n$-dimensional hyperbolic McCord solenoidal manifold which fibers over a compact manifold $M$ having
a fundamental group whose profinite completion is $\wg$.
\bigskip

\begin{proposition}[Solenoidal Cartan-Hadamard theorem]\label{SCHT} As in the
the previous paragraph, let $(\cL^n,g)$ be a hyperbolic McCord solenoidal manifold obtained as an infinite tower of regular coverings
of the compact, orientable manifold $M$ with fundamental group $\pi_1(M)=\Gamma$, with profinite completion $\wg$.
There exists a leafwise isometry between the hyperbolic lamination
$\,(\cL^n, g)$ and the orbit space of an action of $\pi_1(M)\times\Gamma_{\alpha}$ on      
$H^n_{\wg}=\H^n\times\wg$, given a formula of the type:
\begin{equation}\label{action}
F_{_{(\gamma,\kg)}} (x,\mathfrak{h})=(\Psi_{\mathfrak{h}}(\gamma,\kg)(x), j(\gamma)\cdot\mathfrak{h}\cdot\kg^{-1}) \quad (\gamma\in\pi_1(M),\,
\kg\in\Gamma_{\alpha}).
\end{equation}
\noi  Where $\Gamma_\alpha\subset\wpi$ a closed normal group such that $\wpi/\Gamma_\alpha$ is a Cantor group.
If $\mathfrak{e}$ is the identity element of $\Gamma_\alpha$ and
$\Phi_\kh(\gamma):=\Psi_{\kh}(\gamma,\mathfrak{e})$, then, for each $\kh\in\wg$, $\Phi_\kh$
satisfies:
\begin{equation}\label{phihomomorphism} \Phi_{\mathfrak{h}}(\gamma_2\gamma_1)=\Phi_{\mathfrak{h}}(\gamma_2)\circ\Phi_{\mathfrak{h}}(\gamma_1),
\quad \Phi_\kh(\gamma)\in\text{SO}^+(n,1).
\end{equation}

\noi Here, $j$ is the canonical inclusion of $\pi_1(M)$ into $\wg=\widehat{\pi_1(M)}$.

 The action of $\pi_1(M)$ on $\H^n\times\wg$ obtained by restricting the action
 (\ref{action}) to $\pi_1(M)\times\{\mathfrak{e}\}$ is given by the formula:
 \begin{equation}\label{phi}
 \varphi_\gamma(x,\kh)=F_{_{(\gamma,\mathfrak{e})}}(x,\kh)
 =(\Phi_\kh(\gamma)(x),j(\gamma)\cdot\kh).
 \end{equation}
\noi This action of $\pi_1$ is properly discontinuous, co-compact, and by isometries on the leaves endowed with the leafwise hyperbolic metric
 $\mathcal{h}_{\wg}$ on $\H^n\times\wg$ ({\it definition} \ref{uhcs}).
 Therefore all hyperbolic McCord solenoids over the compact manifold $M$
 are covered, in the solenoidal sense, by $(H^n_{\wg},\mathcal{h})$.
The covering solenoidal map is a leafwise isometry.

 \end{proposition}
 \begin{proof} By {\it remark} \ref{UPM}, there exists a closed normal group $\Gamma_\alpha\subset\wpi$ such that $\wpi/\Gamma_\alpha$ is a Cantor group and $\,\cL^n$ is obtained as the orbit space of the
 action of $\pi_1(M)\times\Gamma_\alpha$ on $\tilde{M}\times\wg$ given by the formula:
 \begin{equation}\label{action-gamma-alpha}
 f_{(\gamma,\kg)}(x,\kh)=(\gamma(x),j(\gamma)\cdot\kh\cdot\kg^{-1})
 \quad  (\gamma\in\pi_1(M),\, \,\,\kh\in\Gamma_\alpha),
  \end{equation}
 where $\pi_1(M)$ acts by deck transformations on $\tilde{M}$.
 
 \noi If we lift the laminated metric $g$ to a metric $\tg$
 on $\tilde{M}\times\wg$ then $\pi_1(M)$ acts by
 leafwise isometries with respect to this metric.
 Since the homeomorphism $I:\H_{\wg}^n\times\wg\to\cL^n$ of {\it lemma} \ref{sameuniversal}
 sends leaves onto leaves isometrically with respect to the
 laminated metrics, it follows that for each
 $\gamma\in\pi_1(M)$ and $\kg\in\Gamma_\alpha$ the homeomorphism
 $
 F_{(\gamma,\kg)}=I^{-1}\circ{f_{(\gamma,\kg)}}\circ{I}:\H^n_{\wg}\to\H^n_{\wg}
 $
 is a leafwise isometry.
 
 \noi The fact that
 $(\gamma,\kg)\mapsto{F}_{(\gamma,\kg)}$ is
 an action implies that it is of the form given by the formula (\ref{action}). Since
 $I$ conjugates the actions
 $\{f_{(\gamma,\kg)} :(\gamma,\kg)\in\pi_1(M)\times\Gamma_\alpha\}$ and
 $\{F_{(\gamma,\kg)} :(\gamma,\kg)\in\pi_1(M)\times\Gamma_\alpha\}$ and
 $I$ sends the leaf
 $\H^n\times\{\kh\}$ onto the leaf $\tilde{M}\times\{\kh\}$ it follows that the action has all the stated properties and $\Phi_{\kh}$ satisfies (\ref{phihomomorphism}).
 \end{proof}
 
 \vskip1cm
\begin{remark}\label{Mishyp}
 For a fixed $\mathfrak{h}$, the maps $\Phi_\kh:\pi_1(M)\to\text{SO}^+(n,1)$, $\kh\in\wg$,
 Given by formula  (\ref{phihomomorphism}) above,
 are monomorphisms, and we will show, in {\it corollary} \ref{Mhyperbolic}, that the image is discrete and acts cocompactly on $\H^n$. This implies that $\pi_1(M)$ is a uniform lattice in
$\text{SO}^+(n,1)$. Since $M$ is covered by a leaf of $\cL^n$, it follows that
$M$ is aspherical so that it is an Eilenberg-MacLane space $K(\pi_1(M),1)$, and
 $M$ is homotopy equivalent to a hyperbolic $n$-manifold. In \cite{Ka2} Misha Kapovich shows essentially the same result \ie that a hyperbolic $n$-dimensional McCord solenoid 
 $\cL^n$ is based over a manifold $M$ homotopy equivalent to a compact hyperbolic $n$-manifold. As he indicates, in dimensions different from 4, a closed $n$-dimensional manifold which is homotopy-equivalent to a hyperbolic manifold is actually homeomorphic to it. In dimension 3 it is a corollary of Perelman’s Geometrization theorem, in dimensions greater than 4 this result is due to Farrell and Jones \cite{FJ}.  However, in our case, we know that the manifold can be chosen to be  $M=\H^n/\Phi_\kh(\pi_1(M))$
that is a hyperbolic manifold.
See {\it remark} \ref{isomtocanonical} below.
\end{remark}
 
\noi By {\it remark} \ref {Mishyp} a hyperbolic McCord solenoid is homeomorphic to a McCord solenoid based ({\it definition }\ref{McCord-Solenoids}) on a compact hyperbolic
 $n$-manifold. In what follows we will show the stronger fact that if $n\geq3$ and $(\cL^n,g)$ is a compact, connected, hyperbolic lamination
 which is topologically homogeneous or, equivalently, a McCord solenoid,
 then there exists a compact, orientable, connected, hyperbolic $n$-manifold
 $(M^n,h)$ such that $(\cL^n,g)$ is the inverse limit of an increasing tower
 of regular coverings of finite order which are local isometries and the
 laminated metric $g$ on $\cL^n$ is leafwise isometric to the laminated metric obtained as the pullback of the hyperbolic metric $h$ on $M^n$ under the canonical projection $\Pi:\cL^n\to{M^n}$ (\ie it is the {\it canonical metric}). By {\it remark} \ref{UPM} about the universal property of the algebraic universal covering,
to prove that every laminated metric is transversally locally constant, it is enough to prove the case when $\cL^n$ is homeomorphic
to the algebraic universal covering $\mathbb{M}$. We need to show that any laminated hyperbolic Riemannian metric on $\mathbb{M}$ is isometric
to the canonical laminated hyperbolic metric $g_{\rm{can}}$ \ie the
pullback metric under the projection $\Pi:\mathbb{M}\to{M}$, restricted to the leaves, of the hyperbolic metric on $M$.  
\begin{remark}\label{notallcanonical} For $n\geq3$
the laminated hyperbolic metric $g_{\rm{can}}$ is invariant under the action of
$\pi_1(M)\times\widehat{\pi_1(M)}$.
For $n\geq3$ this
property characterizes this canonical metric \ie it is ``transversally locally constant'' (TLC) in Sullivan's notation \cite{Su2} page 549.
In fact, for $n\geq3$ the structure is actually transversally constant:  {\it  The Teichmüller space of hyperbolic metrics on $\cL^n$ reduces to one point}.

\noi This is false
when $n=2$, since not every metric is TLC although they are dense in the Teichmüller space. This is the reason the Teichmüller space
of $\cL_{\mathcal{h}}$  ({\it definition} \ref{UHL}) is infinite-dimensional. In fact, in \cite{BV} it is shown that
 Sullivan’s Teichmüller space is Kähler isometric biholomophic to the space of continuous functions from the profinite completion of the fundamental group of a compact Riemann surface of genus greater than or equal to two to the Teichmüller space of this surface; i.e. one finds natural Kähler coordinates for the Sullivan’s Teichmüller space. It is an important example of an infinite dimensional Kähler manifold.
 \end{remark}
 
 \noi Henceforth we let $n\geq3$. Let $(\mathbb{M},g)$ be a compact, connected, hyperbolic $n$-dimensional McCord solenoidal manifold based on the compact, connected, and oriented manifold $M$.
Then by proposition \ref{uncountable-subgroups} and {\it remark} \ref{UPM},
$\mathbb{M}$ is covered
{\it in the solenoidal sense}, by $\tilde{\mathbb{M}}:=\tilde{M}\times\wg$ where $M$ is a compact $n$-manifold and $\wg=\wpi$.
 
\noi If we lift the leafwise hyperbolic metric $g$ to $\tilde{\mathbb{M}}$ we obtain a laminated metric
$\tilde{g}$ on
$\tilde{M}\times\wg$. For each $\mathfrak{h}\in\wg$  let $L_{\mathfrak{g}}:=\tilde{M}\times\{\mathfrak{h}\}$ denote the leaf corresponding to $\mathfrak{h}$ and let $d_{\mathfrak{h}}$ denote  hyperbolic distance in the leaf $L_{\mathfrak{h}}$:
\begin{equation}\label{leaf-distance}
d_{\mathfrak{h}}\left((\tilde{x}_1,\mathfrak{h}),(\tilde{x}_2,\mathfrak{h})\right)
\end{equation}
Then the metric $\tilde{g}$ satisfies the following properties:
\begin{enumerate}
\item For each $\mathfrak{h}\in\wg$ the metric on the leaf
$L_{\mathfrak{h}}:=\tilde{M}\times\{\mathfrak{h}\}$ is a hyperbolic metric of sectional curvature -1.
\item Because the metric $\tilde{g}$ is the lifting of $g$, the action of $\pi_1(M)$ on $\tilde{M}\times\wg$ given by the formula
\begin{equation}\label{action2}
(\tilde{x},\mathfrak{h})\overset{\gamma}\mapsto(\gamma(\tilde{x}), j(\gamma)\mathfrak{h}),\quad  \tilde{x}\in\tilde{M}, \mathfrak{h}\in\wg,
\end{equation}
where $j$ is the canonical injection ({\it definition} \ref{inclusion-j}), is by isometries on the leaves:
$
d_{\mathfrak{h}}\left((\tilde{x}_1,\mathfrak{h}),(\tilde{x}_2,\mathfrak{h})\right)=
d_{_{j(\gamma)\cdot\mathfrak{h}}}\left(\gamma(\tilde{x}_1),
j(\gamma)\cdot\mathfrak{h}),
(\gamma(\tilde{x}_2),j(\gamma)\cdot\mathfrak{h})\right)
$
\end{enumerate}
\medskip
\noi Using the solenoidal Cartan-Hadamard theorem, {\it proposition} \ref{SCHT}, we have the following proposition:
\begin{proposition}\label{extension-at-infinity} The action
$\{\varphi_\gamma:\H^n\times\wg\to\H^n\times\wg :\gamma\in\pi_1(M)\}$ of $\pi_1(M)$
by the leafwise isometries
on $\H^n\times\wg$, given by formula (\ref{phi}), has the
following properties:
\begin{enumerate}
\item  It extends to a continuous action $\{\hat\varphi_\gamma :\gamma\in\pi_1(M)\}$
on the compactification
$\overline\H^n\times\wg$.
\item The extended action maps boundary spheres onto boundary spheres conformally,
\ie if $\gamma\in\pi_1(M)$ then $\varphi_\gamma$ maps $\sS^{n-1}\times\{\mathfrak{h}\}$ onto
$\sS^{n-1}\times\{j(\gamma)\mathfrak{h}\}$ conformally.
\item Furthermore, the action
of $\wg$ on $\H^n\times\wg$
given by the formula:
\begin{equation}
\label{action-hatGamma}
F_{\mathfrak{g}}(\tilde{x},\, \mathfrak{h})=
(\tilde{x}\,,\,\mathfrak{h}\cdot\mathfrak{g}^{-1}),
\quad  \mathfrak{g}\in\Gamma_\alpha
\end{equation}
\noi is by quasi-isometries \ie for each
$\mathfrak{g}\in\wg$
there is a positive constant $c$, independent of $\mathfrak{g}$ such that:
\[
c^{-1}\mathbf{d}_\mathfrak{h}((\tilde{x}_1,\mathfrak{h}),(\tilde{x}_2,\mathfrak{h}))\leq
\mathbf{d}_{\mathfrak{h}\cdot\mathfrak{g}^{-1}}((\tilde{x}_1,\mathfrak{h}\cdot\mathfrak{g}^{-1}),
(\tilde{x}_2,\mathfrak{h}\cdot\mathfrak{g}^{-1}))
\leq{c\,\mathbf{d}_\mathfrak{h}((\tilde{x}_1,\mathfrak{h}),(\tilde{x}_2,\mathfrak{h}))}
\]
where $\mathbf{d}_\mathfrak{h}$
is the distance in formula (\ref{leaf-distance}). Notice that the action is not {a priori} necessarily by isometries, since the lifted metric $\tilde{g}$
is not necessarily the product metric (however, this will be the case by
Theorem \ref{base-hyp}).
 
\noi In fact, the combined action
of $\pi_1(M)\times\wg$ on $\H^n\times\wg$:
\begin{equation}\label{combinedaction}
F_{(\gamma,\mathfrak{g})}(\tilde{x}, \mathfrak{h})=
({\gamma}(\tilde{x}),j(\gamma)\cdot\mathfrak{h}\cdot\mathfrak{g}^{-1}),
\quad \left(\gamma\in\pi_1(M), \,\, \mathfrak{g},
\in\Gamma_\alpha\right)
\end{equation}
\noi acts by quasi-isometries on leaves \ie for each
$(\gamma,\mathfrak{g})\in\pi_1(M)\times\wg$,
$F_{(\gamma,\mathfrak{g})}$ is a quasi-isometric diffeomorphism from the leaf $L_{\mathfrak{h}}$
onto the leaf $L_{_{j(\gamma)\cdot\mathfrak{h}\cdot\mathfrak{g}^{-1}}}
$.

\noi In other words, there is a positive constant $c$ such that:
\[
c^{-1}\mathbf{d}_{\mathfrak{h}}((\tilde{x}_1,\mathfrak{h}),(\tilde{x}_2,\mathfrak{h}))\leq
\mathbf{d}_{_{\kh'}}
((\gamma(\tilde{x}_1),\kh'),
(\gamma(\tilde{x}_2),\kh'))
\leq{c\,\mathbf{d}_{\mathfrak{h}}((\tilde{x}_1,\mathfrak{h}),(\tilde{x}_2,\mathfrak{h}))},
\]
where $\kh'=j(\gamma)\cdot\mathfrak{h}\cdot\mathfrak{g}^{-1}$.
 \end{enumerate}
\end{proposition}
 
\begin{proof}
The proof is a laminated version of the usual arguments used to prove the first steps in the Mostow rigidity theorem. In fact, it is almost identical to the proofs of Lemma 5.9.1 given by Thurston in  \cite{Th1}.
Since $\pi_1(M)$ acts by isometries on the leaves, properties (1) and (2) are evident. For the action of $\wg$ we proceed as follows:
 
\noi For each
$\mathfrak{g}\in\wg$ fixed,
the map
$F_{\mathfrak{g}}$ is a diffeomorphism of
$\H^n\times\wg$, in the
laminated sense, since this map is a lifting of the fiber-preserving right action of
$\wg$ by isometries on the fibers with respect to the bi-invariant metric $d$ on the Cantor fiber. Since
$\mathbb{M}$ is compact it follows that $F_{\mathfrak{g}}$ is uniformly bi-Lipschitz. This implies that the maps $F_{(\gamma,\mathfrak{g})}$ have uniformly  bounded derivatives, hence they are bi-Lipschitz and we can
choose $c$ to be a Lipschitz constant, with $c$ independent of
$(\gamma,\mathfrak{g})$. This implies the result.
\end{proof}

\begin{corollary}\label{qc-extension} The action of
$\pi_1(M)\times\wg$ on
$\tilde{M}\times\wg=\H^n\times\wg$,
given by formula (\ref{combinedaction}), extends to a continuous action on the compactification (given in {definition} \ref{compactification})
$\,\bar\H^n\times\wg=\bar{D}^n\times\wg$.
The extended action is by uniformly quasi-conformal maps on the boundary spheres.
\end{corollary}
 \noi  The existence of a uniformly quasiconformal extension is a standard fact
  originally proven by Mostow himself. However, for the sake of completeness
  in the solenoidal version, we recall
  the proof as adapted and presented in Lemma 5.9.2 in
Thurston's notes \cite{Th1}.

  \begin{proof} One uses the fact that our maps are quasi-isometries between hyperbolic leaves: the image of a geodesic
$\alpha:\R\to{L_{\mathfrak{h}}}$
in the leaf $L_{\mathfrak{h}}$ under $F_{(\gamma,\mathfrak{g})}$ is a
quasi-geodesic, $\beta:=F_{\gamma,\mathfrak{g}}\circ\alpha:\R\to L_{j(\gamma\cdot\mathfrak{h})\mathfrak{g}^{-1}}$, in $L_{j(\gamma\cdot\mathfrak{h})\mathfrak{g}^{-1}}$. Then, there exists a unique
geodesic $\alpha':\R\to L_{j(\gamma\cdot\mathfrak{h})\mathfrak{g}^{-1}}$ such that the image of the quasi-geodesic $\beta$ is contained
in a tubular neighborhood of points at a bounded distance
$D$ from $\alpha'$ (\ie quasi-geodesics are ``shadowed'' by actual geodesics).
This is Morse-Mostow lemma \cite{Mor}.
Let $\sS^{n-1}_\mathfrak{h}:=\sS^{n-1}\times\{\mathfrak{h}\}$. Let $z$ be a point in the
solenoidal boundary leaf ${\sS^{n-1}_\mathfrak{h}}$ and
$\alpha:\R\to{L_{\mathfrak{h}}}$ is a geodesic
in $L_{h}$ such that $\alpha(\infty)=z$. One defines the extension of
$F_{(\gamma,\mathfrak{g})}$ at the point $z$ by the formula
$\alpha'(\infty)=\beta(\infty)$.
 \noi This is exactly as Mostow did in his celebrated paper \cite{Mo}, (see also \cite{Mo1}) and is presented in an adapted form in Thurston notes \cite{Th1} Corollary 5.9.3. Since we can define this extension for
every point in the solenoidal boundary $\sS^{n-1}\times\widehat{\pi_1(M)}$,
one obtains for each $(\gamma,\mathfrak{g})$ a ``boundary'' map of the map
$F_{(\gamma,\mathfrak{g})}$ given by formula \ref{combinedaction}:
\begin{equation}\label{actionboundary}
\partial{F_{(\gamma,\mathfrak{g})}}:\sS^{n-1}\times\wg\to
\sS^{n-1}\times\wg
\end{equation}
For each $(\gamma,\mathfrak{g})$, the map $\partial{F_{(\gamma,\mathfrak{g})}}$
has the following properties:
\begin{enumerate}
\item It is a homeomorphism of the boundary solenoid.
\item It restricts to a quasi-conformal map from the sphere
$\sS^{n-1}_{\mathfrak{h}}$ to the sphere
$\sS^{n-1}_{j(\gamma)\cdot\mathfrak{h}\cdot\mathfrak{g}^{-1}}\,$, with respect to their natural conformal structures.
\end{enumerate}
The proof of both items is identical (in a leafwise sense) to the original one given by Mostow in \cite{Mo}, and is a consequence of geometric properties of hyperbolic geometry as explained in \cite{Th1}, Corollaries 5.9.5 and 5.9.6. We omit the details.\end{proof}
\begin{theorem}[Ergodicity of the laminated geodesic flow]\label{ELGF}
Let $(\cL,g)$ be a compact hyperbolic McCord solenoidal manifold so that $\cL$ is a principal
fiber bundle with fiber the Cantor group $\wg_\alpha$.
 Let $T^1\cL$ be the
laminated unit tangent bundle of $\cL$. Then, the laminated geodesic flow
$g_t:T^1\cL\to{T^1\cL}$ is ergodic with respect to the
product measure $\nu=\mu\times\rm{d}\mathbf{h}$, where $\mu$ is Lebesgue measure on the unit sphere bundle
and $\rm{d}\mathbf{h}$ is Haar measure on $\wg_\alpha$
\end{theorem}

\noi To motivate the proof, for $n=2$, let's refer to the following theorem in \cite{ADMV}:
\begin{theorem}[Theorem 1 in \cite{ADMV}] \label{thm:uniqueergodicity}
If $\mathcal{F}$ is a minimal Riemannian foliation of a closed manifold $M$ by hyperbolic surfaces, then the horocycle flow $h_s^+$ on the unit tangent bundle $\hat M = T^1\mathcal{F}$ is strictly ergodic \ie it is minimal and the and uniquely ergodic with unique invariant measure the measure that is locally the product of Liouville measure in the unit tangent bundles of the leaves and the. transversal  Haar measure inherited by the Riemannian foliation. In particular, the geodesic flow on the unit tangent bundle  $\hat M = T^1\mathcal{F}$  is ergodic. For the special case of the unit tangent bundle of Sullivan's universal hyperbolic surface $\cL_\mathcal{h}$  (\ref{UHL}) minimality of the horocycle flow (\ie Hedlund's theorem) was proven in \cite{MMV}.
\end{theorem}

\vskip1cm

\noi  McCord solenoids are minimal and equicontinuous ({\it Remark} \ref{Equicontinuity of McCord solenoids}) so that they behave as Riemannian foliations and the proof that the geodesic flow on the unit tangent bundle of a compact McCord lamination by hyperbolic leaves is ergodic  
given in \cite{ADMV} carries through {\it mutatis mutandis}.
 
\begin{proof}[Proof of Theorem \ref{ELGF}]
The measure $\mathbf{\nu}=\mathbf{\lambda}\times\rm{d}\mathbf{h}$, where $\mathbf{\lambda}$ is the Liouville measure on the unit tangent bundle of the compact hyperbolic lamination $\mathbb{M}$ and $\rm{d}\mathbf{h}$ is Haar measure on
$\wg$, is invariant under the laminated geodesic flow,
which is a laminated Anosov flow \cite{Ano, V2}.

\noi We will use a very general version of the Hopf argument by Yves Coudène in \cite{Cou, Cou3}.
 Let $(X,d)$, be a metric space endowed with a Borel probability measure $\nu$. Let $f_t : X\to{X}, t\in\R$ be a measure-preserving measurable flow.
The {\it stable distribution}  or {\it stable set} of $x\in{X}$ of the flow $f_t$ is defined as follows:
\[\label {SS}
W^{ss}(x) = \left\{y \in X : d(f_t(x),f_t(y)) \to0\quad as\quad t\to\infty\right\}.
\]
The {\it unstable distribution}  or {\it unstable set} is defined as follows

\[\label {UUS}
W^{uu}(x) = \left\{y \in X : d(f_t(x),f_t(y)) \to0\quad as\quad t\to-\infty\right\}.
\]

\begin{remark} \noi Let $N$ be an $n$-dimensional hyperbolic manifold. Its geodesic
flow $g_t:T^1{N}\to{T^1{N}}$, on its unit tangent bundle $T^1{N}$ is an Anosov flow \cite{Ano,V2} and the stable and unstable sets
$W^{ss}(x), W^{uu}(x)$ are immersed differentiable manifolds of dimension $n-1$. These
manifolds are the leaves of two smooth foliations $\mathcal{F}^{ss}$,
 $\mathcal{F}^{uu}$. The foliations are called the {\it strongly stable} and
{\it strongly unstable} foliation respectively.  The geodesic
flow preserves both foliations (\ie sends leaves into leaves). The
orbits of $W^{ss}(x)$ and $W^{uu}(x)$ under the geodesic flow are smooth immersed manifolds of dimension $n$.
The manifolds, $W^s(x)$ and $W^u(x)$ are invariant under $g_t$ and are the leaves of two smooth foliations
$\mathcal{F}^{cs}$, $\mathcal{F}^{cu}$ which meet transversally along the orbits of the geodesic flow. These  $n$-dimensional foliations
are called {\it central stable} and {\it central unstable} foliations respectively.

\noi If $g_t:T^1\cL\to{T^1\cL}$ is the laminated geodesic flow, the restriction of the flow to each leaf is an Anosov flow since the leaves
on the ($2n-1$)-lamination $T^1\cL$  are the unit tangent bundles of the hyperbolic leaves of $\cL$. Therefore $T^1\cL$ admits the four foliations $\cF^s,\cF^u, \cF^{ss}$ and $\cF^{uu}$ with leaves
$W^s(x),W^u(x), W^{ss}(x)$ and $W^{uu}(x)$. We call the foliations $\cF^{ss}$ and $\cF^{uu}$ the {\it laminated horocycle foliations}.

\noi The laminated geodesic flow $g_t$ admits a
{\it local product structure}: for each $x\in{T^1\cL}$ there exists a homeomorphism
$\varphi:D^{n-1}\times{D^n}\times{K}\to\mathcal{U}$, where $\mathcal{U}$ is
an open neighborhood of $x$ in $T^1\cL$, $D^{n-1}$ is the open unit ball in $\R^{n-1}$,
$D^{n}$ is is the open unit ball in $\R^{n}$, $K \subset\wg_\alpha(x)$ is a Cantor set which is an open subset of the Cantor group fiber $\wg_\alpha(x)$
through $x$. Furthermore:
\begin{enumerate}

\item $\varphi(0,0,k)=k, \,\, \forall k\in{K}$,
\item $\varphi(D^{n-1},\{0\},\{k\})\subset {W^{ss}(k)},\,\,\forall k\in{K}$
\item $\varphi(\{x\}, D^{n},\{k\})\subset {W^{u}(\varphi(x,0,k))}
,\,\,\forall (x,k)\in{D^{n-1}\times{K}}$
\end{enumerate}
\begin{remark}\label{product}Let $\mathbf{p}:T^1\cL\to \cL$ be the natural projection. The map
$\mathbf{p}$ is a locally trivial fibration over $\cL$ with fiber $\sS^{n-1}$.
The image $\mathbf{p}(W^{u}(x))$ of a horospherical leaf $W^{u}(x)$
(respectively the image $\mathbf{p}(W^{s}(x))$)
is the whole leaf $L(\mathbf{p}(x))$ of $\cL$. This means that in the leaf
$L(\mathbf{p}(x))$ of
$\cL$ there exists a point $y\in{L(\mathbf{p}(x))}$ and a unit vector $v_y$ at $y$ such that $v_y\in{W^u(x)}$. Since the leaf $L(\mathbf{p}(x))$ is dense in $\cL$ it follows
that $W^u(x)\cap{K}:=K'$ is dense in $K$, and therefore $W^u(x)$ meets $\cU$ in infinitely many plaques. By the local product structure, every point  $z\in\cU$ of the form $z=\varphi(x,y,k)$ with $k\in{K'}$ has the property that
$W^{ss}(z)\cap{{W^u(x)}}\neq\emptyset$.
 \end{remark}

\medskip

\noi Theorem 1.3 in \cite{Pl} of J. Plante establishes that for an Anosov flow on a compact manifold $N$ with non-wandering set equal to $N$,
both $\mathcal{F}^{cs}$ and
$\mathcal{F}^{cu}$ are minimal.  He also proves that if either $\cF^{ss}$ or $\cF^{uu}$ is not minimal then the Anosov flow is the suspension of an Anosov diffeomorphism. When $n=2$, $\mathcal{F}^{ss}$ and
$\mathcal{F}^{uu}$ are the stable and unstable horocycle orbits and thus Plante's, theorem implies that both the
stable and unstable horocycle flows on the unit tangent bundle of a hyperbolic surface are minimal because the flow is not a suspension
(providing another, and easier,  proof of Hedlund's theorem).
\end{remark}
\noi The minimality of the foliations $\cF^{cs}$ and $\cF^{cu}$ also holds for the laminated geodesic flow on $T^1\cL$. Indeed, since the laminated geodesic flow, $g_t$ preserves the measure $\nu$, that is positive on open sets, it follows from Poincaré's recurrence theorem that the set of recurrent points of the geodesic flow is dense. Theorem 1.3 of Plante  in \cite{Pl} can be easily adapted using the density of recurrent points and the product neighborhood,  described above, to show that for each
 $x\in{T^1\cL}$, the closure $\overline{W^u(x)}$ is both open and closed. Then, since $T^1\cL$ is connected,
 $\overline{W^u(x)}=T^1\cL$. Indeed,
 if $z\in\overline{W^u(x)}$ and  $\cU$ is a product neighborhood of $z$ then, for every  recurrent point $y$ in $\cU$
 since $\overline{W^u(x)}$ is closed and invariant under the flow,
 and $W^{ss}(y)\cap{W^u(x)}\neq\emptyset$  ({\it remark} \ref{product}) one has $y\in\overline{W^u(x)}$
 (since $\lim_{t\to\infty}d(g_t(y),x)=0$). Since the recurrent points are dense it follows that $\cU\subset\overline{W^u(x)}$, hence $\overline{W^u(x)}$ is open.
   Similarly one shows that $\overline{W^s(x)}=T^1\cL$. This implies that the foliations $\cF^{cs}$ and  $\cF^{cu}$ are minimal.

\begin{definition} A measurable function $g : X \to \R$ is $W^{ss}$- saturated
(respectively  $W^{uu}$- saturated) if there is a set $\Omega\subset{X}$
with $\mu(\Omega)=1$ such that
$x,y\in\Omega$ and $y\in{W^{ss}}$ (respectively $y\in{W^{uu}}$) implies that
$g(x) = g(y)$, \ie $g$ is constant for points on the same stable and unstable sets in a set of full measure.
\end{definition}


 
\noi The following is a strong version,  due to Y. Coudène \cite{Cou3}, of the Eberhard Hopf argument (\cite{Hop},
\cite{Cou}, \cite{Cou1},\cite{Wilk}).
 \begin{theorem} \label{coudène} Let $X $be a metric space, $\nu$ a Borel probability measure on $X$, and $f_t: X \to X$ a measure-preserving flow on $X$. Then any $f_t$ invariant $L^2$-function is both $W^{ss}$-saturated and  $W^{uu}$-saturated.
 \end{theorem}
\begin{corollary}\label{stableinvariant} Let $X$ be a metric space, $\nu$ a Borel probability measure on $X$ and  $f_t: X \to X$ a measure preserving flow on $X$. Assume that any measurable function which is $W^{ss}$- saturated and $W^{uu}$- saturated and invariant under the flow is constant almost everywhere. Then $f_t$ is ergodic.
 \end{corollary}
 
 \noi {\it Theorem \ref{coudène}} , {\it Corollary} \ref{stableinvariant}  and minimality of the foliation $\cF^{cs}$
 imply that any continuous function $f:T^1\cL \to\R$ which is $g_t$-invariant is constant. Standard facts of analysis
 (\ie the fact that continuous functions are dense in $L^2$ functions, see \cite{Wilk}, page 12, for instance) imply that any measurable $g_t $-invariant function is constant
 \noi This finishes the proof of {\it theorem} \ref{ELGF}.\end{proof}
\begin{remark} The ergodicity proof of {\it theorem} \ref{ELGF} given here holds for any compact minimal lamination by leaves of negative curvature, not necessarily constant if the lamination has a transverse measure invariant under holonomy: the geodesic flow on the unit tangent bundle is ergodic with respect to the measure that is locally the product of the transverse measure and the leafwise Liouville measure along the unit tangent bundle of the leaves.
\end{remark}
\noi The group $G=\pi_1(M)\times\wg\times{\text{SO}^+(n,1)}$ acts transitively
on the set
\[
\mathcal{D}:=\{((z,\mathfrak{h}),(w,\mathfrak{h}))\,: z,w\in\sS^{n-1},\mathfrak{h}\in \times\wg) \}
\subset(\sS^{n-1}\times\wg)\times(\sS^{n-1}\times\wg),
\]
via by the diagonal action. The set $\mathcal{D}$ can be identified with the set of oriented geodesics
on the hyperbolic solenoid $(\cL,g_{can})$ with the canonical metric. The stabilizer $B$ of a point
$((z_0,\mathfrak{h}_0),(w_0,\mathfrak{h}_0))$ is non-compact and both $G/\pi_1(M)$ and
$G/(\pi_1(M)\times \wg)$ have finite measure. {\it Theorem} \ref{ELGF}
 implies the stronger fact that the diagonal action of $\pi_1(M)\times\wg$ on the subset
is weakly ergodic \ie any measurable invariant set is of measure 0 or 1. The action of $G$ on $\sS^{n-1}\times\wg$ is also weakly ergodic. Both facts follow by a
theorem of C. C. Moore (Proposition 3 in \cite{Moore}). From this it follows:

\begin{lemma}\label{ergodic} The action $\partial{F_{(\gamma,\mathfrak{g})}}$ of $\pi_1(M)\times\wg$, given by the formula
(\ref{actionboundary}), on the boundary
solenoid $\sS^{n-1}\times\wg$, is weakly ergodic with respect to the
product measure $\mu\times\rm{d}\mathbf{h}$, where $\mu$ is Lebesgue measure on the unit sphere
and $\rm{d}\mathbf{h}$ is Haar measure on $\wg$.
\end{lemma}

\begin{corollary}\label{Mhyperbolic} {\rm Still under the hypothesis $n\geq3$}.
For each $(\gamma,\mathfrak{g})$,
the map $\partial{F_{(\gamma,\mathfrak{g})}}$
restricts to a conformal map from the sphere
$\sS^{n-1}_{\mathfrak{h}}$ to the sphere
$\sS^{n-1}_{j(\gamma)\cdot\mathfrak{h}\cdot\mathfrak{g}^{-1}}\,$, with respect to their natural conformal structures \ie the restriction belongs to $\text{SO}^+(n,1)$.
In particular, choosing $\mathfrak{h}=\mathfrak{e}$, the identity element
of $\pi_1(M)$, and $\mathfrak{g}=j(\gamma)$ we have that
$\partial{F_{(\gamma,j(\gamma))}}$ maps conformally $\sS^{n-1}_{\mathfrak{e}}$
onto itself.
 
\noi Denoting by
$f_\gamma:\sS^{n-1}_{\mathfrak{e}}\to\sS^{n-1}_{\mathfrak{e}}$ this restriction
one has that $f_\gamma\in{\text{SO}^+(n,1)}$. Furthermore, if
$\gamma_1,\gamma_2\in\pi_1(M)$ then
$f_{\gamma_1\cdot\gamma_2}=f_{\gamma_1}\circ{f_{\gamma_2}}$, in other words,
the subgroup $\pi_1(M)\times{j(\pi_1(M))}\subset\pi_1(M)\times\wg$
is isomorphic to $\pi_1(M)$, leaves invariant  the sphere $\sS^{n-1}_{\mathfrak{e}}$.
The group $G=\{f_\gamma \,:\, \gamma\in\pi_1(M)\}\simeq\pi_1(M)$
extends to the closed disk
$\bar\H^n\times\{\mathfrak{e}\}$ and acts on the interior
$\H^n\times\{\mathfrak{e}\}$ by hyperbolic isometries, freely, co-compactly
and $(\H^n\times\{\mathfrak{e}\})/G=M$. So that $M$ is a $n$-hyperbolic Manifold
and we conclude that the fundamental group of $M$ is isomorphic to a uniform lattice of $\text{SO}^+(n,1)$.\end{corollary}

\begin{proof} By {\it Corollary} \ref{qc-extension} each
$\partial{F_{(\gamma,\mathfrak{g})}}$ is quasi-conformal, therefore by
the theorem of Rademacher-Stepanov (\cite{RM},\cite{Ste}, \cite{Va} page 97)
$\partial{F_{(\gamma,\mathfrak{g})}}$ is differentiable almost everywhere with respect to the measure
$\mathbf{\mu}$. Hence the differentials of the group elements, which are defined at a set of total measure one, send round $\sS^{n-2}$ spheres tangent to spheres on the boundary leaves onto  $(n-2)$-ellipsoids of bounded eccentricity. To prove the corollary one uses the fact that if the action were not 1-quasi-conformal then the action
would leave invariant a field of ellipsoids of eccentricity different from on, on a set of positive measure (see  \cite{Su3} and
\cite{Tu}), and this fact would violate, by {\it lemma}, \ref{ergodic} ergodicity of the action of $\pi_1(M)\times\wg$. Therefore all elements are 1-quasi-conformal and well-known and deep analytic theorems imply that they are conformal diffeomorphisms \cite{Va}.\end{proof}

\begin{remark}\label{isomtocanonical}
By {\it corollary} \ref{Mhyperbolic} it follows that if $n\geq3$ then the action of the subgroup
$\pi_1(M)\times{j(\pi_1(M))}\simeq\pi_1(M)$ of $\pi_1(M)\times\wg$ on $\H^n\times\wg$ given by formula \ref{combinedaction}
acts conformally and leaves invariant  the sphere $\sS^{n-1}_{\mathfrak{e}}$.
The group $G=\{f_\gamma=F_{(\gamma,j(\gamma))} \,:\, \gamma\in\pi_1(M)\}\simeq\pi_1(M)$
extends by isometries to $\H^n\times\{\mathfrak{e}\}$ and the action is properly discontinuous, free, and uniform. Therefore, we can assume that if $n\geq3$ the hyperbolic McCord solenoidal manifold $(\cL^n,g)$ is based on the hyperbolic
$n$-manifold $M=(\H^n\times\{\mathfrak{e}\})/G$ and $(\cL^n,g)$ is isometric to
$(\cL^n,g_{can})$.

\end{remark}
 
 \begin{remark}[Structural stability of the hyperbolic solenoidal geodesic flow] One of the most remarkable and celebrated theorems of D. V. Anosov \cite{An} is
the fact that the geodesic flow on a compact manifold with a Riemannian metric of negative sectional curvature is structurally stable.
In fact, given any two Riemannian metrics of strictly negative sectional curvature on the smooth compact manifold $N$, {\it if the metrics are sufficiently close in the $C^1$ topology,} then their corresponding geodesic flows  $g_t, g'_t$  are topologically equivalent
\ie there exists a homeomorphism $h:T^1{N} \to T^1{N}$ which transforms the oriented trajectories of  $g_t$  into the oriented trajectories of $g_t$.
Furthermore, the homeomorphism $h$ depends continuously on the metrics.
A proof  using functional analysis originally due to
J. Mather and J. Moser \cite{Math, Mos} can be found in  \cite{RoVe}
 and \cite{V2}. The proof in \cite{RoVe}, which corrects an error in \cite{Math}, can be easily adapted to show structural
 stability for the laminated geodesic flow $g_t:T^1\cL\to{T^1\cL}$.
\end{remark}

\subsection{Mostow rigidity theorem for hyperbolic solenoids}\label{Mostow}
The proofs of the results in the previous section were very natural and could have been done by following
several other approaches, for instance Mostow \cite{Mo, Mo1}, Prasad \cite{Pr}, Gromov \cite{Gr} (with an appropriate notion of
solenoidal simplicial volume \cite{Gr1}) or Besson, Courtois, and Gallot \cite{BCG}. We have used the proof in the book by Thurston \cite{Th1}. Of course, the idea was to adapt a leafwise version of what is done in \cite{Th1}. With the aid of the results in section \ref{hypsol} one can prove the following solenoidal version
of Mostow's rigidity theorem.

\begin{theorem}[Mostow rigidity for hyperbolic McCord solenoids]\label{mostow}
Let $\cL_1$ and $\cL_2$ be two
compact, McCord hyperbolic, $n$-dimensional  solenoidal manifolds with $n\geq3$
and simply connected leaves. Thus the two laminations are endowed with Riemannian metrics with leaves homeomorphic to $\H^n$. Let $f:\cL_1\to\cL_2$ be a homotopy equivalence (in particular a homeomorphism). Then, there exists
an isometry $g:\cL_1\to\cL_2$ homotopic to $f$. \end{theorem}

\noi Under the hypothesis, the holonomy pseudogroup of local homeomorphisms of the Cantor set are Morita equivalent \cite{BC}, \cite{Hae}.
Since the leaves of both solenoids are
homeomorphic to $\R^n$-- it follows from \cite{CHL}
that $\cL_1$ and $\cL_2$ are homeomorphic. The proofs in \cite{CHL}, for the general case of solenoids with contractible leaves, use some subtle facts by Farrel and Jones about the Borel Conjecture for manifolds with negative curvature, valid for dimensions greater than four. However, in our case, the leaves are homeomorphic to $\R^n$ and we avoid exoticism. Therefore, the proof of {\it theorem} \ref{mostow} is reduced to the following lemma:

\begin{lemma}\label{homisometry}
 Let $f:(\cL^n,h_1)\to(\cL^n,h_2)$ be a continuous map homotopic to the identity,  where
$\cL^n$ is a compact McCord hyperbolic, $n$-dimensional  
solenoidal manifold, with leaves isometric to $\H^n$, with $n\geq3$ and $h_1$, $h_2$ are laminated hyperbolic metrics.
Then $f$ is homotopic to a laminated isometry. In particular if two compact,
topologically transitive, hyperbolic solenoids of dimension $n\geq3$ are homeomorphic
then there exists a homeomorphism which is a leafwise isometry.
\end{lemma}
 
\begin{proof}The proof of {\it lemma} \ref{homisometry} is almost identical to
{\it proposition} \ref{extension-at-infinity} and its corollaries. By {\it proposition} \ref{SCHT},
 $\cL^n$ is covered by the algebraic universal covering
$\tilde{\mathbb{M}}=\tilde{M}\times\wg$
of a hyperbolic $n$-manifold $M$. Using the solenoidal Cartan-Hadamard we lift the homotopy
equivalence $f$ to a homotopy equivalence
$\tilde{f}:\tilde{M}\times\wg\to\tilde{M}\times\wg$. Using the same arguments as  
{\it corollary} \ref{qc-extension}, and ergodicity ({\it lemma} \ref{ergodic}), this homotopy equivalence extends to a homeomorphism
$F:\sS^{n-1}_{\wg}\to\sS^{n-1}_{\wg}$ of the spherical boundary
solenoid $\sS^{n-1}_{\wg}=\sS^{n-1}\times\wg$ as a homeomorphism which is conformal on each sphere. Therefore this boundary map extends to a map $\pi_1(M)$-equivariant $\hat{F}:\bar\H^n\times\wg\to\bar\H^n\times\wg$ which is an isometry on
$\H^n\times\wg$. The map $\bar{F}$ descends to a leafwise isometry, isotopic to the identity, between
$(\cL^n,h_1)$ and $(\cL^n,h_2)$.
\end{proof}

\noi {\it Theorem} (\ref{mostow}) implies

 \begin{theorem}\label{base-hyp} If $(\cL^n,\hat{g})$ is a hyperbolic, $n$-dimensional McCord solenoidal manifold with $n\geq3$ then $\cL^n$ is isometric to the solenoid
 $(\cL^n,g_{can})$ obtained as the inverse limit of an infinite tower of regular coverings of a compact hyperbolic
 $n$-manifold $(M,g)$, where the solenoid
 is endowed with the canonical metric ({\it definition} \ref{canonicalmetric}).
 
\noi We recall that this metric is the leafwise pullback of the hyperbolic metric $g$ of $M$
under the canonical projection $\Pi:\cL^n\to{M}$. The solenoidal $n$-manifold corresponds to a conjugacy class
 of a closed normal subgroup $\Gamma$ of the Cantor group $\wg=\wpi$,
 such that the quotient group $\wg/\Gamma$ is a Cantor group. In other words,
 $(\cL^n,\hat{g})$ is isometric to the quotient of $\H^n\times{M^n}$
 under the free,
 properly discontinuous and co-compact action of $\pi_1(M)$ on $\H^n\times\wg$ of the form
 $f_\gamma(x,\kh)=(\gamma(x),j(\gamma)\cdot\kh)\,$ $($$\gamma\in\pi_1(M)$$)$,
 where $\gamma$ acts by isometries on $\H^n$. Any hyperbolic manifold in the same commensurability class of $M$ determines $(\cL^n,\tg)$.
 \end{theorem}
 \begin{remark} The McCord Solenoidal manifold
 $(\cL^n,g_{can})$ is locally homogeneous since it is isometrically covered by the universal solenoid $(\mathbb{M},g_{cal})$ ({\it remark} \ref{lochom}).
 \end{remark}
 
 \begin{corollary} The Teichmüller space of a compact $n$-dimensional
 topologically transitive hyperbolic solenoid (\ie a McCord solenoid)
  is one point if $n\geq3$.
 \end{corollary}
 
 \begin{remark} We have not dealt with geometric solenoidal manifolds
which are not compact but are of finite solenoidal volume
({\it definition} \ref{solenoidalmeasure}). In the case of
$n$-dimensional, connected, hyperbolic McCord solenoidal manifolds of finite volume
when $n\geq3$, most of our results hold. In particular the following
version of Mostow's rigidity theorem for hyperbolic McCord solenoids of finite
volume holds:
\begin{theorem}
{\it Any connected $n$-dimensional hyperbolic solenoidal manifold of
finite solenoidal volume (given by {\it definition} \ref{solenoidalmeasure}) is isometric to the solenoid obtained as the inverse limit of an increasing tower of finite regular coverings of a complete hyperbolic manifold of finite volume. The number
of solenoidal ``cusps'' is finite and they are $(n-1)$-dimensional solenoidal nil-manifolds}.
\end{theorem}
\end{remark}
\begin{remark} The remarkable results by T. Farrell and P. Ontaneda \cite{Farr} imply
that for $n\geq10$ the space of negatively curved metrics of a closed negatively curved Riemannian $n$-manifold, is highly non-connected. In particular, they show in
Corollary 2 \cite{Farr}, that for any closed hyperbolic manifold $(M^n,g)$, $n\geq10$, there is a hyperbolic metric $g'$ on $M$ such that $g$ and $g'$ cannot be joined by a path of negatively curved metrics. This phenomenon holds true for compact hyperbolic
McCord solenoidal manifolds of dimension $n\geq10$.

\end{remark}

\noi Mostow's rigidity theorem does not hold for the groups $G=Solv_3$    
and $G=Nil_3$.  
\noi Theorem A in the paper by W.D. Neumann \cite{Neu} implies that all compact, {\it orientable}, geometric non simply connected
3-manifolds with ``Seifert geometries''
$\mathbb{E}^3$, $\widetilde\PSL$, $\sS^2\times\R$, $\H^2\times\R$,
 and $Nil_3$ have only one commensurability class and thus to each class
 corresponds a unique (up to homeomorphism) algebraic universal covering
 $\mathbb{M}$. See Theorem \ref{neu} by W. D. Newmann below.
 \begin{remark} The solenoidal Mostow rigidity theorem fails in these ``Seifert'' cases (except for the ``trivial'' case of geometries modeled on $\sS^3$ and $\sS^2\times\R$).
 The reason is the following. From \cite{WZ1}
(Theorem B) one has that if $M$ is a closed, orientable, aspherical 3–manifold, then $M$ is Seifert fibered if and only if the profinite completion $\widehat{\pi_1(M)}$ has a nontrivial procyclic normal subgroup (\ie a normal subgroup that is topologically generated by a single element). This uses the {\it Seifert Fiber Space Theorem}, of the 1990s: Let $M$ be an orientable irreducible 3-manifold whose $\pi_1$ is infinite and contains a nontrivial normal cyclic subgroup. Then $M$ is a Seifert fiber space.
 This theorem is due to several authors (see the survey \cite{Pre} {Theorem 5}). Due to the work of Perelman one can omit the reducibility hypothesis. It follows that if $\mathcal{S}^3$ is a geometric McCord solenoidal manifold modeled on the geometries $\widetilde\PSL$, $\sS^2\times\R$, $\H^2\times\R$,
 and $Nil_3$ then  $\mathcal{S}^3$ fibers over a geometric solenoidal manifold
 $\mathcal{S}^2$ (either euclidean or hyperbolic) with fiber a 1-dimensional solenoid and in both cases the Teichmüller space of flat or hyperbolic metrics in infinite dimensional.
\end{remark}

\subsection{Hyperbolization of McCord solenoids}  
\label{hyperbolization} Thurston's Hyperbolization Theorem, and Theorem A in \cite{WZ1},
imply that if $M$ is a closed, orientable, aspherical 3–manifold, then $M$ is hyperbolic if and only if the profinite completion
$\widehat{\pi_1}(M)$ does not contain a subgroup isomorphic to $\widehat\Z^2$.
Then,  {\it proposition} \ref{uncountable-subgroups} and {\it remark} \ref{UPM} imply:
\begin{theorem}
Let $\cL^3$ be a compact, connected and orientable McCord solenoid which is a principal
fiber bundle over the compact, oriented, and connected manifold $M$ with fiber the Cantor group $\widehat\Gamma$. Suppose that the leaves are aspherical and suppose that every maximal free abelian subgroup of  $\Gamma$ is isomorphic to $\Z$ (hence
$\wg$ does not have a  subgroup isomorphic to $\widehat\Z^2$). Then $\cL^3$ has a laminated hyperbolic metric
which is unique.
 
\noi In fact, $M$ is a compact hyperbolic manifold and the unique
laminated metric is the canonical metric obtained by pullback under the
canonical projection $\Pi:\cL^3\to{M}$.
\end{theorem}
 
\subsection{McCord hyperbolic solenoids and virtual fibration theorem} \label{Virtual fibration}

\noi Using {\it theorem} \ref{ClarkHurder} in {\it subsection} \ref{hom1d} and the deep results of Ian Agol \cite{IA} and Daniel Wise\cite{Wis} on the virtual fibration theorem of compact hyperbolic manifolds we obtain:
 
\begin{theorem} {\bf (Solenoidal virtual fibration theorem).}\label{vft} Let
$\mathfrak{S}_{{\H}^3}$ be a compact, connected, hyperbolic solenoidal manifold
\ie a compact geometric solenoidal manifold modeled on $(\H^3, PSL(2,\C))$.  
Suppose  $\mathfrak{S}_{{\H}^3}$  is topologically homogeneous (in particular it is minimal).
Then  there exists a locally trivial smooth fibration $\mathfrak{P}$ of $\mathfrak{S}_{{\H}^3}$ of over the circle ${\mathbb S}^1$:
$$
\mathfrak{P}: \mathfrak{S}_{{\H}^3}\to{\sS}^1.
$$
\noi Here we define smoothness in the sense of laminations.
 
\noi Thus $\mathfrak{P}$ is a smooth submersion. The fiber $\mathfrak{P}^{-1}(\{1\})$ is a 2-dimensional solenoidal surface $\mathcal{S}^2$  which is a McCord solenoidal surface with every leaf simply connected. Therefore, there exists a homeomorphism $f:\mathcal{S}^2\to\mathcal{S}^2$ such that $\mathfrak{S}_{{\H}^3}$ is obtained by suspending $f$.
\end{theorem}
 
\begin{proof} By {\it theorem} \ref{ClarkHurder}, $\,\mathfrak{S}_{{\H}^3}$ is a McCord solenoid and by {\it theorem} \ref{base-hyp} the laminated metric is the canonical metric.
Hence there exists a directed set of finite regular coverings (which can be assumed to be local isometries) of a compact hyperbolic manifold $M$, such that the inverse limit is leafwise isometric to $\mathfrak{S}_{{\H}^3}$. By {\it remark} \ref{chain-cofinal}  we can choose the directed set to be a chain of coverings $\{p_{n+1}:M_{n+1}\to{M_n}:\,\in\N\}$. By the theorem of  Ian Agol \cite{IA} and Daniel Wise \cite{Wis} any hyperbolic manifold is finitely covered by a compact hyperbolic manifold that fibers over the circle with connected fiber. Thus we can always choose the sequence of coverings $\{p_{n+1}:M_{n+1}\to{M_n}:\,\in\N\}$ such that the inverse limit is indexed by the naturals $\N$ and the first manifold $M_1$ fibers over the circle: $\pi^1:M_1\to\sS^1$ with connected fiber a compact surface $S$ of genus greater than one and
monodromy isotopic to a pseudo--Anosov homeomorphism.
If $\Pi_1:\mathfrak{S}_{{\H}^3}\to{M_1}$ is the canonical projection then we can choose $\mathfrak{P}=\pi^1\circ\Pi_1$. Then $\mathfrak{P}$ is a locally trivial fibration with fiber $\mathcal S=\mathfrak{P}^{-1}(S)$. Since $
\mathfrak{S}_{{\H}^3}$ is topologically homogeneous, it follows that
$\mathcal S$ is a minimal
(therefore connected) 2-dimensional solenoidal McCord solenoid with simply connected leaves. {\em Of course the manifold $M$ is not unique, it can be replaced
by a  compact hyperbolic manifold  commensurable to $M$.}
\end{proof}
\begin{remark}
Each $M_{n+1}$ fibers over the circle by the map $\pi^{n+1}:M_{n+1}\to\sS^1$ defined by the composition  
$\pi^{n+1}=\pi^1\circ{p_{1}}\circ\cdots\circ{p_{n+1}}$. The fiber is a compact orientable surface
$S_{n+1}$ of genus greater or equal to 2.
\end{remark}
\begin{remark} {\em Theorem \ref{vft} can be thought as the generalization of Ian Agol
\cite{IA} and Daniel Wise's \cite{Wis}
 {\bf virtual fibration theorem} to the case of compact hyperbolic 3-dimensional solenoidal manifolds.
The only requirement is topological homogeneity.
But of course, it depends directly
on the depth of the results by Agol and Wise and {\it theorem} \ref{ClarkHurder}}.
\end{remark}
 
 \noi The properties of the homeomorphism $f:\mathcal{S}^2\to\mathcal{S}^2$ in
 {\it theorem} \ref{vft} above are very interesting. It is the lifting, in the tower of coverings,
 of the fiber $p^{-1}(\{1\})$, of the virtual fibration of the pseudo-Anosov homeomorphism of the fiber which determines the fibration over the circle.
 
\noi By {\it proposition} \ref{uncountable-subgroups} the surface solenoid $\mathcal{S}^2$ is covered, in the solenoidal sense, by Sullivan's universal hyperbolic lamination $\cL_\mathcal{h}$ in
{\it definition} \ref{UHLdef}. In other words, if
$\wg$ is the profinite completion of the fundamental
group of the surface $\Sigma$ of genus 2, there exists a closed normal subgroup
$\Gamma_\alpha$ of $\wg$ such that $\Gamma_\alpha$ acts fiberwise and properly
on $\cL_\mathcal{h}$ (considered as a principal bundle over
$\Sigma$ with fiber $\wg$) and $\mathcal{S}^2=\cL_\mathcal{h}/\Gamma_{\alpha}$
\noi Therefore the algebraic fundamental group of $\mathfrak{S}_{{\H}^3}$ is a semidirect product $\hat\Gamma_{\alpha}\rtimes\hat{\Z}$.
 
\smallskip
\noi To each fibration $\pi^{n+1}:M_{n+1}\to\sS^1$ we can associate an infinite cyclic covering \cite{Mi}
$\mathfrak{p}_{n+1}:S_{n+1}\times\R\to{M_{n+1}}$. The group of deck transformations of this covering is isomorphic to $\Z$ and
it is generated by a homeomorphism of the form $\mathfrak{t}_{n+1}(x,t)=(h_{n+i}(x),t+1)$ where $h_{n+1}:S_{n+1}\to{S_{n+1}}$ is a
pseudo-Anosov homeomorphism leaving invariant two measured foliations $\cF^u$ and $\cF^s$ with dilation $\lambda_{n+1}>1$ thus the sequence $\{\lambda_n\}_{n\in\N}$ is a commensurability invariant. By Thurston \cite{Th} (see also \cite{Fr}) each $\lambda_n$ is an algebraic unit in a number field. In fact, it is a bi-Perron number.
Of course, the sequence of Alexander polynomials corresponding to the sequence of infinite cyclic branched coverings \cite{McM}
in the spirit of J. W. Milnor \cite{Mi}
is also a commensurability class invariant for compact hyperbolic 3-manifolds.
 
\subsection{Commensurability invariants for compact hyperbolic Solenoids} \label{Invariants hyperbolic case}
\noi In \cite{Bo}, \cite{MacR},\cite{NR}, \cite{CSW}, \cite{Neu} invariants of commensurability classes of hyperbolic manifolds
(especially for arithmetic manifolds) are described. For instance if $G\subset{PSL(2,\C)}$
is a Kleinian group of finite co-volume then the field $\Q(\{trace(g^2):\,\, g\in G\}$ is an invariant for the commensurability class of $G$.
 
\noi We have the sequence $\{p^*_{n+1}: S_{n+1}\times\R \to{S_{n}}\times\R  ,\,\, n\in\N\}$
of finite coverings of hyperbolic manifolds obtained by pullback:
 
$$
\begin{CD}
S_{n+1}\times\R    @>p^*_{n+1}>>  S_{n}\times\R\\
@VV{\mathfrak{p}_{n+1}}V        @VV{\mathfrak{p}_{n}}V\\
M_{n+1}     @>p_{n+1}>>  M_{n}
\end{CD},
$$
\noi In this diagram, $M_1=M$ and $p^*_{n+1}(x,t)=(f_{n+1(x)},t)$.
 
\noi Therefore we have the inverse limit:
$$
\widehat{\mathfrak{S}_{{\H}^3}}={\lim_{\overset{p^*_{n+1}}{\longleftarrow}}} \,\,\,\{p^*_{n+1}:S_{n+1}\times\R \to {S_{n}}\times\R ,\,\, n\in\N\}
\simeq\cL_{\mathcal{h}}\times\R
$$
 
\noi We have an infinite cyclic covering (in the sense of solenoids):
 
$$
\hat{\mathfrak{P}}: \cL_{\mathcal{h}}\times\R  \to\mathfrak{S}_{{\H}^3}
$$
 
\noi The group of deck transformations is $\Z$ and it is generated by the homeomorphism
$$
\mathfrak{T}: \cL_{\mathcal{h}}\times\R \to\cL_{\mathcal{h}}\times\R , \quad\quad (s,t)\overset{\mathfrak{T}}\mapsto(\mathfrak{t}(x),t+1)
$$
 
\noi The homeomorphism $\mathfrak{t}:\cL_{\mathcal{h}}\to\cL_{\mathcal{h}}$ is the lifting of the initial basic pseudo-Anosov diffeomorphism
$h:S\to{S}$ whose suspension gives the base hyperbolic manifold $M$. The contracting and expanding measured foliations $\cF^s$ and $\cF^u$ corresponding to the quadratic differential associated with $h$ lift to very interesting foliations $\hat{\cF}^s$ and $\hat{\cF}^u$ on the solenoidal surface $\cL_{\mathcal{h}}$.
 
\noi The homeomorphism $\mathfrak{T}$ induces an isomorphism of the profinite completion $\Gamma:=\widehat{\pi_1(S)}$. If we denote by ${H^s_1}(\Gamma)$
the first homology group of the solenoidal surface in the sense of Steenrod \cite{St1}, \cite{St2}, then since
$\mathfrak{T}$ induces an isomorphism $t$ of the abelian group $H_1(\Gamma)$
we have, as usual, that, $H^s_1(\Gamma)$ is a module over
$\Z[t,t^{-1}]$. We call this module the {\em Alexander invariant or Alexander module} of the commensurability class.
 
\noi Of course the commensurator group should play a role in the classification of commensurability classes.
This module is a complete
invariant of the commensurability class for hyperbolic manifolds.
 
\begin{remark} The fibered commensurability class has been studied in depth by Danny Calegari,
 Hongbin Sun and Shicheng Wang in \cite{CSW} in the hyperbolic case. They show that there is a minimal orbifold representative in each commensurability class. Properties of profinite rigidity are studied, for instance, in \cite{BRW},  \cite{BRW1}, \cite{R}, \cite{Wil}.
\end{remark}

\subsection{3-Dimensional solvmanifolds  and real quadratic fields} \label{Invariants solvable case}

\noi  A matrix $A\in{SL(2,\Z)}$ is hyperbolic if $|\text{trace}(A)|>2$. Let $f_A:\mathbb{T}^2\to\mathbb{T}^2$ be the hyperbolic linear automorphism of the 2-torus, induced by $A$ for example, the one given by the
matrix
\[
\left(\begin{array}{cc}2&1\\1&1\end{array}\right).
\]
 
The suspension of $f_A$ is a 3-manifold $\mathbb{T}^3_A$ which fibers
over the
circle with fiber the torus $\mathbb{T}^2$ and monodromy $f_A$. It is a solvmanifold whose universal covering is a solvable Lie group whose Lie
algebra is generated by $X$, $Y$, and $Z$ satisfying:
$$[X,Y]=0,\,\, [Z,X]=-X, \,\, [Z,Y]=Y.$$
Therefore, $Z$ and $Y$ generate a locally free action of the affine group, as well as $Z$ and $X$. The flow generated by $Z$
is a classic example of an Anosov flow \cite{Ano}.
 
\noi Since the matrix $A$ is diagonalizable with positive eigenvalues there exists a 2 by 2 matrix $C$ such that $e^C=A$ so that for every integer  $n$ one has $e^{nC}=A^n$.
 
\noi Let ${\mathbb R}^{3}_A$ be the solvable group whose elements  $((x,y),t)$ belong to ${\mathbb R}^2\times\mathbb R$ and with operation $\star$ defined by:
$$
((x_1,y_1),t_1)\star((x_2,y_2),t_2)=(e^{t_1C}(x_2,y_2)+(x_1,y_1),t_1+t_2)
$$
 
\noi Thus ${\mathbb R}^{3}_A$ is the semidirect product  $\R^2\rtimes _\phi \R$ of $\mathbb R$ with $\mathbb R^2$ corresponding to the monomorphism $\Phi:\R\to GL(2,\R)=Aut(\R^2)$ given by $t\mapsto {e^{tC}}$. The 3-dimensional Lie group ${\mathbb R}^{3}_A$  
is isomorphic to the unique 3-dimensional
unimodular solvable non-nilpotent Lie group, usually denoted by $Solv_3$. It has as discrete and uniform subgroup the group $\Z^3_A$ consisting of elements $((n,m),l)$ of $\R^3_A$ where  $m, n$ and $l$  are integers. The group $\Z^3_A$ is the semidirect product  $\Z^2\rtimes _\phi \Z$ of $\Z^2$ with $\Z$
corresponding to the monomorphism $\phi:\Z\to{SL(2,\Z)}=Aut(\Z^2)$ given by $n\mapsto{A^n}$.
 
\begin{remark} $\Z^2\rtimes _A \Z$ in \cite{BG} denotes the semidirect  
product $\Z^3_A=\Z^2\rtimes _\phi \Z$. We will use the same notation sometimes.
\end{remark}
 
\noi We have $\mathbb{T}^3_A=\R^3_A/\Z^3_A$ and the commutative diagram of short exact sequences

$$
\begin{tikzcd}
0 \arrow{r} & \Z^2 \arrow{r} \arrow{d} & \Z^3_A \arrow{r} \arrow{d} & \Z  \arrow{r}  \arrow{d} & 0\\
0 \arrow{r} & \R^2 \arrow{r} & \R^3_A \arrow{r} &\R \arrow{r} & 0
\end{tikzcd}
$$
 
\noi Therefore, one obtains the sequence  of differentiable maps
$$
0\longrightarrow{\mathbb T}^2\longrightarrow \mathbb{T}^3_A   \overset{F}{\longrightarrow}{\mathbb S}^1\longrightarrow0
$$
 where $F: \mathbb{T}^3_A\longrightarrow {\mathbb S}^1$ is a locally trivial fibration.
There are two subgroups $B_1$ and $B_2$ of  the simply connected solvable Lie group $\mathbb{R}^3_A$
$$
B_1=\{((sa_1,sa_2),t)\in\mathbb{R}^3_A \,:\,s,t\in \R \}
$$

$$
B_2=\{((sb_1,sb_2),t)\in\mathbb{R}^3_A \,:\,s,t\in \R \}
$$
 where $(a_1,a_2)$ and $(b_1,b_2)$ are eigenvectors of $A$ corresponding to the eigenvalues $\lambda$
 and $\lambda^{-1}$.
 
\noi Both $B_1$ and $B_2$ are isomorphic to the real affine group $GA(\R)$:
 $$
 GA(\R)=\left\{  \left(\begin{array}{cc}a&b\\0&{a^{-1}}\end{array}\right):\quad a>0, \,\,b\in\R    \right\}.
 $$
 
\noi $B_1$ and $B_2$ intersect in the
one-parameter subgroup $\{((0,0),(\log\lambda)^{-1}t): t\in\R\}$. This one-parameter subgroup generates the vector field $Z$ and the one-parameter subgroups
$$\{((sb_1,sb_2),0)\in\mathbb{R}^3_A \,:\,s\in \R \}$$
$$\{((sa_1,sa_2),0)\in\mathbb{R}^3_A \,:\,s\in \R \}$$
 generate the vector fields $X$ and $Y$.   These three vector fields descend to the three vector fields described at the beginning.
This is an algebraic description of the suspension of the Anosov diffeomorphism of the 2 torus induced by the matrix $A$.
 
\noi We have two locally free actions of the affine group $B$ on $\mathbb{T}^3_A$ so that we have two hyperbolic laminations i.e.,
 the induced metrics on the orbits of each of the actions of the affine group are hyperbolic.

\noi We also have a locally free action of ${\mathbb R}^2$ induced by the commuting vector fields $X$ and $Y$. The orbits of this action on $\mathbb{T}^3_A$ are the
2-tori which are the fibers of $F$. The orbits of both $X$ and $Y$ are minimal in the torus fiber which contains them.
 
\noi Thus in the 2-torus fibers there are two irrational foliations with ``slope'' an equivalence class of quadratic surds under the action of $PSL(2,\Z)$ on
$\sS^1=\R\cup{\infty}$. This follows from the fact that the eigenvalues of $A\in SL(2,\Z)$ are the roots
of the equation $x^2-(\text{trace}A)x+1$.
 
\noi If $M$ is a compact manifold with $Solv_3$ geometry, then either $M$ fibers over $\sS^1$ with fiber the torus $\T^2$ and monodromy
given by a hyperbolic matrix $A\in{SL(2,\Z)}$ or it is a quotient of such a torus bundle over the circle by a group of order at most 8. Since we are interested in commensurability classes it is enough to consider torus bundles $\T_A$. The following proposition can be found in  \cite{Ba},\cite{BG} or \cite{Fu}.
\begin{proposition} \label{AB-matrices} Let $A,B\in{SL(2,\Z)}$ be two hyperbolic matrices. Then:
\begin{enumerate}
\item $\T_A$ is isometric to $\T_B$ if and only if $A$ is conjugate to $B$ or $B^{-1}$ by a matrix in $GL(2,\Z)$.
 
\noi This is equivalent to $\Z^3_A$  is isomorphic to  $\Z^3_B$.
 
\item $\T_A$ is commensurable to $\T_B$ if and only if there exist integers $n,m\in\Z-\{0\}$ such that
$A^n$ is conjugate to $B^m$ by a matrix in $GL(2,\Q)$.
 
\noi This is equivalent to:  $\Z^3_A$ is commensurable to $\Z^3_B$.
\end{enumerate}
\end{proposition}

\begin{corollary}\label{sol=quad} (Compare T. Barbot \cite{Ba}, M.  Bridson and S. Gersten \cite{BG})
Commensurability classes among the solvable groups of the form $\Z^3_A=\Z^2\rtimes _A \Z$,
where $A\in{SL(2,\Z)}$ is hyperbolic, are in one-to-one correspondence with the commensurability classes
relative to the multiplicative group of the positive reals $\R^\bullet=\{t\in\R:\,\,t>0\}$
 of cyclic subgroups
$G_\lambda=\{\lambda^n:\,\,n\in\Z\}$ where $\lambda+\frac1{\lambda}\in\N,\,\, \lambda\in(0,\infty),\,\,\lambda\neq1$.
\end{corollary}
\noi Corollary \ref{sol=quad} is a fancy way to say the following:
 
\noi $\T^3_A$ is commensurable to $\T^3_B$ if and only if there exist non-zero integers
$n$ and $m$ such that $\lambda_1^n\lambda_2^m=1$ where $\lambda_1>1$ and $\lambda_2>1$ are the stretching
factors of the linear torus Anosov diffeomorphisms induced by $A$ and $B$, \ie they are the quadratic irrationalities which are roots
of $x^2-trace(A)x+1$  and $x^2-trace(B)x+1$ and are bigger than one. In dynamical systems this is very natural:
If $\T^3_A$ and $\T^3_B$ are commensurable then a compact manifold that covers both is necessarily of the form $\T^3_C$ and
its Lyapunov exponents are commensurable in $\R^\bullet$.
 
\noi Equivalently:
\begin{corollary}\label{sol=quad2} There exists a one-to-one correspondence between geometric commensurability classes of compact 3-dimensional
manifolds modeled in $Solv_3$ and real quadratic fields.
\end{corollary}
\noi First let's notice that there are {\em oriented} manifolds modeled on $Sovl_3$ which do not fiber over the circle: one takes two nontrivial
interval bundles over the Klein bottle and glues appropriately their boundaries (which are two tori). However any compact 3-manifold with a
$Solv_3$ geometry is double covered by one that  {\em does} fiber over the circle. Then this double covering is of the form $\T^3_A$ and if we take $\T^3_{A^2}$
which is commensurable with $\T^2_A$ we can assume that $A$ or $A^2$ have positive trace then we can assume that $\lambda$ is in
the quadratic field $\Q\left(\sqrt{(trace(A))^2-4}\right)$.  
By proposition \ref{AB-matrices} this field is an invariant of the commensurability class.
 Reciprocally, given a square-free positive integer $d$ we can consider the totally real quadratic field $F=\Q(\sqrt{d})$ and its two real embeddings
 $i_1:\Q(\sqrt{d})\to\R$ and $i_2:\Q(\sqrt{d})\to\R$. If we take the embedding
 $I:\Q(\sqrt{d})\to\R^2,\quad I(k)=(i_1(k), i_2(k))$  we see that
 $I(\mathfrak{O}_F)$  is a lattice $\Lambda$  in $\R^2$.
 
 \noi By Dirichlet's unit theorem the group of units $\mathfrak{O}_F^*$ of  $\mathfrak{O}_F$ is infinite cyclic plus a
 group of order 2 and it acts linearly by multiplication
 $(t,s)\overset{\mathfrak{u}}\mapsto(\mathfrak{u}t,\mathfrak{u}^{-1}s)\quad (\mathfrak{u}\in\mathfrak{O}_F)$ on
 $\R^2$ preserving $\Lambda$. So it descends to an Anosov linear map $f_\mathfrak{u}:\R^2/\Lambda\to\R^2/\Lambda$
 of the torus $\R^2/\Lambda$. The suspension of this map is a
 solvmanifold that fibers over the circle with hyperbolic monodromy. Its commensurability class is independent of the unit since the
 group of units is free cyclic plus 2-torsion. This description implies: that the fundamental group of any solvmanifold
 $\T^3_A$ is a subgroup of the real affine group:
  $$
 GA(\Q(\sqrt{d}))=\left\{  \left(\begin{array}{cc}a&b\\0&{a^{-1}}\end{array}\right):\quad a>0, \,\, a,b\in\Q(\sqrt{d})    \right\}
$$
 
 \subsection{Adelic solenoids and solenoidal manifolds modeled in $Solv_3$}
 
 \noi Consider $\T^3_A$ in the spirit of geometric solenoidal 3-manifolds modeled on $Solv_3$. Then we denote its algebraic universal covering
 by the symbol $\mathfrak{S}_{A}$ and we have:  
 \begin{equation}
 \mathfrak{S}_{A}=(\widehat\sS^1\times\widehat\sS^1)\times _{\widehat{A}}\widehat\sS^1
 \end{equation}
 
 \noi where $\widehat\sS^1$ is the 1-dimensional
 universal solenoid which is the algebraic  universal covering of the circle and it is the Pontryagin dual of the additive rationals $(\Q,+)$ as
 described in definition \ref{adelesclassgroup}.
 (Thus $\widehat\sS^1=\mathbb{A}_\Q/\Q$ is the {\em ad\`ele class group of the rationals}).
 The 3-dimensional solenoid $\mathfrak{S}_{A}$ fibers over $\widehat\sS^1$ (the ad\`elic circle) with fiber
 $\widehat\sS^1\times\widehat\sS^1$ ({\em ad\`elic torus}).
The algebraic fundamental group of $\sS^1$ is the profinite completion of the integers $\widehat{\Z}$.
 Therefore the algebraic fundamental group of $(\widehat\sS^1\times\widehat\sS^1)\times _{\widehat{A}}\widehat\sS^1$  is
 the semidirect product $(\widehat{\Z}\times\widehat{\Z})\rtimes_{\widehat{A}}\widehat{\Z}$.
 
 \noi This encapsulates the commensurability classes
 of solvmanifolds.
 \begin{remark} Corollary \ref{sol=quad2} was proved by Walter D. Neumann in \cite{Neu},
Theorem A. \end{remark}
 In fact, in \cite{Neu} Walter D. Neumann has given a classification of topological commensurability classes as follows:
 \begin{theorem}[Theorem A in Walter D. Newmann \cite{Neu}]\label{neu}
 For each of the six ``Seifert geometries" modeled on
 $\sS^3$, $\mathbb E^3$, $\H^2\times\R$,
 $\sS^2\times\R$, $\widetilde\PSL$ and $Nil_3$
 there is just one topological commensurability class of compact geometric 3-manifolds with the given geometric structure (two for the last two geometries if orientation-preserving commensurability of oriented manifolds is considered).
For the remaining non-hyperbolic geometry $Solv_3$, the geometric commensurability classes are in one-one correspondence with real quadratic number fields.
\end{theorem}
{\bf Acknowledgement,} Thanks are due to the referee who made some useful suggestions for improving the first version of this article.

\end{document}